\documentclass[oneside]{amsart}
\usepackage{comment}
\usepackage{amssymb,amsmath,amsthm,mathrsfs,amstext,amssymb}
\usepackage{tabularx}
\usepackage{array}
\usepackage[square,sort,comma,numbers]{natbib}
\usepackage{url} 
\usepackage{hyperref}
\usepackage{accents}
\usepackage {xcolor} 
\usepackage{upgreek}
\usepackage{mathtools}

\usepackage{lscape}

\theoremstyle{plain}
\newtheorem{lemma}{Lemma}[section]
\newtheorem{prp}[lemma]{Proposition}
\newtheorem{thm}[lemma]{Theorem}
\newtheorem{crl}[lemma]{Corollary}
\newtheorem{question}[lemma]{Question}
\newtheorem{claim}[lemma]{Claim}
\newtheorem{dfn}[lemma]{Definition}

\newcommand{\cl}{\mathrm{cl}}
\newcommand{\power}{\mathcal{P}}

\newcommand{\ga}{\alpha}
\newcommand{\gb}{\beta}

\newcommand{\gw}{\omega}

\newcommand{\cof}{\mathtt{cof}}

\makeatother 
\usepackage{babel}

\title{On the Isbell Problem}

\author[J. Cancino-Manr\'irquez]{J. Cancino-Manr\'iquez}

\address[Jonathan Cancino-Manr\'{i}quez]{Institute of Mathematics, Czech Academy of Sciences \\ \ \v{Z}itn\'{a} 25, 110 00 Praha 1, Czech Republic}

\address{Departamento de Matem\'aticas, Facultad de Ciencias, Universidad Nacional Aut\'onoma de M\'exico. Circuito Exterior S/N, Ciudad Universitaria, CDMX, 04510, M\'exico}
\email{cancino@math.cas.cz,jcancino@ciencias.unam.mx,mhacajoh@gmail.com}

\author[Jind\v{r}ich Zapletal]{J. Zapletal}
\address{Department of Mathematics, 358 Little Hall, University of Florida, Gainesville FL32611}

\email{zapletal@ufl.edu}

\subjclass[2020]{ 03E35, 03E04, 03E05}

\keywords{Tukey ordering, Tukey types of ultrafilters, Tukey top ultrafilter, nowhere dense ultrafilter, ideal, ultrafilter, forcing}

\begin{document}

\begin{abstract}
    We present three models concerning Tukey types of ultrafilters on $\omega$. The first model is built via a countable support iteration and we show there is no basically generated ultrafilter in such model. The second and third models are built upon different and novel techniques, and in such models all ultrafilters are Tukey top, thus providing an answer to the Isbell problem. In all models there is no $\mathsf{nwd}$-ultrafilter.
\end{abstract}

\maketitle

\section{Introduction.}\label{introduction}

After a long oblivion period  - around 40 years -, Tukey types of ultrafilters on the natural numbers have been extensively investigated in the last two decades. J. R. Isbell, in his paper \cite{Isbell} from 1965, shortly mentioned the question on the number of non-equivalent Tukey types for ultrafilters on $\omega$ as Problem 2: when we consider an ultrafilter ordered by reverse inclusion, how many Tukey types of orderings $(\mathcal{U},\supseteq)$ do exist, where $\mathcal{U}$ is a non-principal ultrafilter on $\omega$? He also provided an example of ultrafilter having maximal cofinal type. Such ultrafilters, nowadays known as Tukey top, are those whose cofinal type is equivalent to $([\mathfrak{c}]^{<\omega},\subseteq)$. Let us briefly recall the Tukey ordering. Let $\mathbb{Q}=(Q,\leq_{\mathbb{Q}})$  and $\mathbb{P}=(P,\leq_{\mathbb{P}})$ be two directed partial orderings, we say that $\mathbb{Q}$ is Tukey reducible to $\mathbb{P}$ if there is mapping $\varphi:P\to Q$ which sends cofinal sets in $\mathbb{P}$ to cofinal sets in $\mathbb{Q}$, in which case we write $\mathbb{Q}\leq_T\mathbb{P}$; $\mathbb{P}$ and $\mathbb{Q}$ are Tukey equivalent if $\mathbb{Q}\leq_T\mathbb{P}$ and $\mathbb{P}\leq_T\mathbb{Q}$. Isbell's problem, in its most fundamental aspect - the existence of a non-Tukey top ultrafilter just on the basis of $\mathsf{ZFC}$ - has remained as an open question. The question called the attention of mathematicians after D. Milovich published a paper \cite{milovich} in which he addressed the problem under additional assumptions to $\mathsf{ZFC}$, such as Martin's Axiom or Jensen's diamond.
A few years later, N. Dobrinen and S. Todor\v{c}evi\'c joined the endeavour with their paper \cite{dt1}, providing an extensive research on Tukey types of ultrafilters on $\omega$ and introducing the most prominent class of ultrafilters which are not Tukey top, the basically generated ultrafilters, which are not even Tukey above $[\omega_1]^{<\omega}$. This class of ultrafilters includes all the known types of ultrafilters which are not Tukey above $[\omega_1]^{<\omega}$. The research has been actively extended as can be seen in papers such as \cite{dobrinen_1,dobrinen_2,dt1,dobrinen_todorcevic_2,dobrinen_todocevic_3,dobrinen_trujillo_mijares,dilip_shelah,dilip_stevo,milovich}.

Most of the research in this direction has focused on ultrafilters with small Tukey type, or the structure of Tukey order among ultrafilters under additional assumptions to $\mathsf{ZFC}$, and little has been done regarding big Tukey types of ultrafil\-ters. It is not hard to prove that any non $p$-point ultrafilter is Tukey above $(\omega^\omega,\leq^*)$, where $\leq^*$ is the almost dominating relation. From this it follows that consistently, all ultrafilters are Tukey above $(\omega^\omega,\leq^*)$: any model with no $p$-points suffices, and the existence of such models is widely known(see \cite{Wimmers},\cite{silver_p_points}). This remark allows us to conceive the possibility of all ultrafilters having a big Tukey type in the sense that consistently all of them may have a common lower bound. On the other hand, there have been just a few examples of ultrafilters being Tukey top, which are essentially the constructions from \cite{Isbell}, \cite{dow_zhou_noetherian}, \cite{milovich}, \cite{dt1}, and more recently, the authors of \cite{tom_fanxin} have achieved a different construction. In this paper we address two problems on big Tukey types of ultrafilters:
\begin{enumerate}
    \item The consistent non-existence of basically generated ultrafilters.
    \item The consistency of all ultrafilters being Tukey top.
\end{enumerate}

The basic result from which these two problems are approached, is the existence, for any uncountable cardinal $\kappa\leq \mathfrak{c}$, of an ideal on $2^\kappa$ which is critical for having Tukey type above $[\kappa]^{<\omega}$, that is, there is an ideal $\mathscr{I}_\kappa$ on $2^\kappa$ such that for any ultrafilter $\mathcal{U}$, $[\kappa]^{<\omega}\leq_T\mathcal{U}$ if and only if $\mathscr{I}_\kappa\leq_K\mathcal{U}^*$, where $\leq_K$ is the Kat\v{e}tov order. These ideals are closely related to the ideal of nowhere dense subsets of $2^\kappa$(with the product topology), so it is natural to ask if Shelah's forcing to destroy nowhere dense ultrafilters can be reformulated to blow up the Tukey type of an ultrafilter. As we will see, this is indeed the case. The main results of this article are the following:
\begin{thm}\label{thm1}
    It is relatively consistent with $\mathsf{ZFC}$ that all ultrafilters are Tukey above $[\omega_1]^{<\omega}$.
\end{thm}
\begin{thm}\label{thm_2}
    It is relatively consistent that all ultrafilters are Tukey top. Moreover, we can get the Continuum to be arbitrarily large.
\end{thm}

Theorem \ref{thm1} is built via a countable support iteration of proper forcings resulting from a reformulation of Shelah's forcing from \cite{shelah_nwd}. Theorem \ref{thm_2} is a more demanding task and needs the development of new forcing techniques. In all the models we construct, there is no $\mathsf{nwd}$-ultrafilter, and therefore, Theorem \ref{thm_2} gives an answer to Problem 3 from \cite{brendle_problems_continuum_large} of J. Brendle on the consistency of the non-existence of nowhere dense ultrafilters and the continuum being larger than $\omega_2$. We produce two models in which all ultrafilters are Tukey top. In both constructions, we first make a forcing extension by an $\omega_2$-length countable support iteration of the Sacks forcing. Then we define suborder of carefully chosen countable support product; this is the step in which both constructions differ: one construction is made through a generalization of countable support iterations of strongly proper forcings\footnote{Here we refer to Shelah's strongly proper forcing notion, instead of the more well known Mitchell's strongly proper forcing notion; see Definition \ref{strongly_proper_def} and the paragraph before it.} having one additional property, while the other relies on the existence of square sequences and one theorem of Woodin about the universe $L[S]$, where $S$ is a sequence of ordinals. However, forcings of the kind $\mathbb{P}*\dot{\mathbb{Q}}$, where $\mathbb{P}$ is strongly proper and $\dot{\mathbb{Q}}$ is Cohen preseving, play a central role in both construction. Both constructions produce forcings which preserve $\omega_1$, are $\aleph_2$-c.c. and Cohen preserving. 

The purpose of the Sacks iteration is to make the size of the Continuum $\aleph_2$ and produce a family of special filters which we name as \emph{suitable filters}. These filters appear along the iteration of the Sacks forcing. Let us call this fixed family of filters $\mathscr{F}$. Any $2^\omega$-saturated filter is an example of a suitable filter. Roughly speaking, for a fixed cardinal $\kappa>\omega_1$ and one suitable filter $\mathcal{F}\in\mathscr{F}$, we define a forcing $\mathbb{P}_{\kappa}(\mathcal{F})$ which makes sure that any ultrafilter extending $\mathcal{F}$ has Tukey type at least $[\kappa]^{<\omega}$ (the definition of $\mathbb{P}_\kappa(\mathcal{F})$ actually depends on the intermediate stage of the Sacks iteration in which it appears); moreover, $\mathbb{P}_\kappa(\mathcal{F})$ forces $\mathcal{F}$ to have this property in any further Cohen preserving extension. Then, in the Sacks model, we consider a suborder of the countable support product of the forcings $\mathbb{P}_\kappa(\mathcal{F})$ with $\mathcal{F}\in\mathscr{F}$, let us call it $\mathbb{P}$ by now. The forcing $\mathbb{P}$, forces the Continuum to have size $\kappa$, has the $\aleph_2$-c.c., is Cohen preserving and adds a generic for each forcing $\mathbb{P}_\kappa(\mathcal{F})$, where $\mathcal{F}\in\mathscr{F}$; thus, $\mathbb{P}$ forces that any ultrafilter extending one of the filters in $\mathscr{F}$ is Tukey top. In the first construction, this suborder is build along the iteration of the Sacks forcing, while the second construction is build after the Sacks iteration. Then we make use of a simple but powerful property which was first noted in \cite{Wimmers}: if $\mathbf{P}$ is a c.c.c forcing, then any ultrafilter from $V^{\mathbf{P}}$ contains a saturated filter from the ground model. More generally, if $\mathbf{P}$ is $2^\omega$-c.c. then any ultrafilter from $V^{\mathbf{P}}$ contains a $2^\omega$-saturated filter from the ground model. Thus, $\mathbb{P}$, having the $2^\omega$-c.c., forces that any ultrafilter from $V^{\mathbb{P}}$ contains a $ 2^\omega$-saturated filter from the Sacks model. It turns out that each one of these $2^\omega$-saturated filters contains one of the suitable filters from the family $\mathscr{F}$. Therefore, any ultrafilter in the forcing extension by $\mathbb{P}$ contains a suitable filter from $\mathscr{F}$. Thus, any ultrafilter in $V^{\mathbb{P}}$ is Tukey top.

The structure of the paper is as follows: section 2 introduces the basic results on
the critical ideals. These ideals pose the question on the possibility of non-nowhere
dense ultrafilters being Tukey top (or at least non basically generated), which we
answer by providing a construction of a basically generated ultrafilter which is not
nowhere dense. Section 3 introduces a kind of filters generalizing the notion of
saturated filters and which will be of importance for the construction of our forcings from sections \ref{no_basically_generated}, \ref{one_tukey_type_start} and \ref{squares_and_products}. Section 4 is dedicated to the development of the basic forcing
$\mathbb{P}_\kappa(\mathcal{U})$, its properties and the proof of Theorem 1.1.  Section \ref{one_tukey_type_start} introduces the first version of the poset we use to prove Theorem 1.2 and
proves the forcing has the $\aleph_2$-c.c. Section \ref{cohen_properness_notion} introduces the notion of \emph{Cohen properness} which is the central property of our iterations, and prove some equivalent formulations of it. In section \ref{iteration_cohen_properness} we introduce one way to iteration Cohen proper forcings and in section \ref{A_preservation_theorem} we prove these iterations preserve the notion of Cohen properness. Section \ref{getting_all_together} completes the construction of the first model for Theorem 1.2. Section \ref{squares_and_products} is dedicated to the construction of the second model for Theorem 1.2; this construction is based on the existence of square sequences and one theorem of Woodin, and provides an existence theorem of a reduced product (that is, a suborder of a countable support product) which is interesting by itself. The paper finishes with a section of remarks and open questions.

\section{Some basic remarks.}\label{basic_remarks}

In this section we introduce basic results serving as motivation for the rest of the paper. Recall from the introduction that given two directed partial orders $\mathbb{Q}=(Q\leq_\mathbb{Q})$ and $\mathbb{P}=(P,\leq_{\mathbb{P}})$, we write $\mathbb{Q}\leq_T\mathbb{P}$ if there is a map $f:P\to Q$ sending $\leq_\mathbb{P}$-cofinal subsets of $\mathbb{P}$ to $\leq_\mathbb{Q}$-cofinal subsets of $\mathbb{Q}$. The next equivalence of this definition is well known. As notation to be fixed along the paper, given a function $f:X\to Y$, we define $rng(f)=\{f(x):x\in X\}$.

\begin{lemma}
Let $\mathbb{P}$ and $\mathbb{Q}$ be directed partial orders. The following are equivalent:
\begin{enumerate}
    \item $\mathbb{Q}\leq_T\mathbb{P}$.
    \item There is a map $f:Q\to P$ sending $\leq_\mathbb{Q}$-unbounded subsets of $Q$ to $\leq_\mathbb{P}$-unbounded subsets of $P$.
\end{enumerate}
\end{lemma}

\begin{lemma}
Let $\mathcal{U}$ be a non-principal ultrafilter on $\omega$. Then $[\kappa]^{<\omega}\leq_T\mathcal{U}$ if and only there is $\{A_\alpha:\alpha\in\kappa\}\subseteq\mathcal{U}$ such that for all $X\in[\kappa]^{\omega}$, $\bigcap_{\alpha\in X}A_\alpha\notin \mathcal{U}$.    
\end{lemma}

\begin{dfn}
Let $\mathscr{I,J}$ be two ideals. We say that $\mathscr{I}$ is Kat\v{e}tov below $\mathscr{J}$, denoted by $\mathscr{I}\leq_K\mathscr{J}$, if there is a function $f:\bigcup\mathscr{J}\to\bigcup\mathscr{I}$ such that for all $A\in\mathscr{I}$, $f^{-1}[A]\in\mathscr{J}$.

\end{dfn}
\begin{dfn}Let $\kappa$ be an infinite cardinal. Consider $2^\kappa$ with the product topology. Then:
\begin{enumerate}
    \item  The ideal of nowhere dense subsets of this space is denoted by $\mathsf{nwd}(\kappa)$.
    \item For each $\alpha\in\kappa$ and $i\in 2$, define the prebasic open set $B_i^\alpha=\{x\in 2^\kappa:x(\alpha)=i\}$.
\end{enumerate}
In the case $\kappa=\omega$ we only write $\mathsf{nwd}$ instead of $\mathsf{nwd}(\omega)$.
\end{dfn}

\begin{dfn}
Let $\mathcal{U}$ be an ultrafilter on $\omega$. We say that  $\mathcal{U}$ is a $\mathsf{nwd}(\kappa)$-ultrafilter if $\mathsf{nwd}(\kappa)\nleq_K\mathcal{U}^*$.
\end{dfn}

The following theorem is the starting motivation of the results presented in this paper.

\begin{thm}
Let $\kappa\leq 2^\omega$. If $\mathcal{U}$ is not a $\mathsf{nwd}(\kappa)$-ultrafilter, then $[\kappa]^{<\omega}\leq_T\mathcal{U}$.   
\end{thm}

\begin{proof}
Let $f:\omega\to 2^\kappa$ be such that for all $X\in\mathsf{nwd}(\kappa)$, $\omega\setminus f^{-1}[X]\in\mathcal{U}$. We have $D^\alpha_0=f^{-1}[B^\alpha_0]\in\mathcal{U}$ or $D^\alpha_1=f^{-1}[B^\alpha_1]\in\mathcal{U}$. Let $g:\kappa\to 2$ be such that for all $\alpha$, $D^\alpha_{g(\alpha)}\in\mathcal{U}$. Now define $G:[\kappa]^{<\omega}\to\mathcal{U}$ as:
\begin{equation}
    G(H)=\bigcap_{\alpha\in H} D^\alpha_{g(\alpha)}
\end{equation}
We claim that $G$ is a Tukey reduction. First note that $G$ is monotone. Let $D\subseteq [\kappa]^{<\omega}$ be an unbounded subset and assume towards a contradiction that $G[D]$ is a bounded by $A_0\in\mathcal{U}$.  Then there is an infinite countable $E\subseteq \kappa$ such that for all $\alpha\in E$, there is $F\in D$ such that $\alpha\in F$. Thus, we have that for all $\alpha\in E$, $A_0\subseteq G(\{\alpha\})=D^\alpha_{g(\alpha)}$, so $A_0\subseteq \bigcap_{\alpha\in E}D^\alpha_{g(\alpha)}$. Define $Z=\{x\in 2^\kappa:(\forall\alpha\in E)(x(\alpha)=g(\alpha))\}$ and note that $A_0\subseteq\bigcap_{\alpha\in E}D^\alpha_{g(\alpha)}=f^{-1}[Z]$, so $f^{-1}[Z]\in\mathcal{U}$. But we also have that $Z\in\mathsf{nwd}(\kappa)$, so $f^{-1}[Z]\notin \mathcal{U}$, which is a contradiction.
\end{proof}

Let us recall the definition of basically generated ultrafilter.

\begin{dfn}[see \cite{dt1}]
An ultrafilter $\mathcal{U}$ is basically generated if there is $\mathcal{B}\subseteq\mathcal{U}$ which generates $\mathcal{U}$, and for any $\{X_n:n\in\omega\}\subseteq\mathcal{B}$ such that $\lim_{n\in\omega}X_n\in\mathcal{B}$, there is $A\in[\omega]^\omega$ such that $\bigcap_{n\in\omega}X_n\in\mathcal{U}$.
\end{dfn}

The following is a well known result.

\begin{prp}
If $[\omega_1]^{<\omega}\leq_T\mathcal{U}$, then $\mathcal{U}$ is not basically generated.
\end{prp}

\begin{crl}\qquad
\begin{enumerate}
    \item If there is no $\mathsf{nwd}(\mathfrak{c})$-ultrafilter, then all ultrafilers are Tukey top.
    \item If there is no $\mathsf{nwd}(\omega_1)$-ultrafilter, then there is no basically generated ultrafilter.
\end{enumerate}

\end{crl}

Although $\mathsf{nwd}(\kappa)$ is an useful ideal to find out the position of an ultrafilter in the Tukey order, there are other ideals which are useful. In particular, Theorem \ref{critical_ideal} identifies a critical ideal for having Tukey type above $[\kappa]^{<\omega}$.

\begin{dfn}Let $\kappa$ be an infinite cardinal. 
\begin{enumerate}
    \item For $A\in [\kappa]^\omega$ and $f:A\to 2$, define $\chi(A,f)=\{x\in 2^\kappa: x\upharpoonright A=f\}$.
    \item For $A\in[\kappa]^{\omega}$ and $f:A\to 2$, define $\chi_\omega(A,f)=[\chi(A,f)]^{\omega}$.
\end{enumerate}
Then define,
\begin{enumerate}
    \item $\tau_\kappa$ is the ideal generated by sets of the form $\chi(A,f)$, where $A\in[\kappa]^\omega$ and $f:A\to 2$.
    \item $\tau_\kappa^\diamond$ is the ideal generated by $\bigcup_{A\in[\kappa]^\omega}\bigcup_{f\in 2^A}\chi_\omega(A,f)$.
\end{enumerate}
\end{dfn}

\begin{dfn}
Let $\kappa$ be an infinite cardinal. We define $\mathcal{C}_\kappa$ as the collection of all sequences $\vec{\nu}=\langle\nu_n:n\in\omega\rangle$ such that:
\begin{enumerate}
    \item For all $n\in\omega$, $\nu_n:dom(\nu_n)\to 2$.
    \item For all $n\in\omega$, $dom(\nu_n)$ is finite.
    \item For all $n\in\omega$ $\max(dom(\nu_n))<\min(dom(\nu_{n+1}))$.
\end{enumerate}
\end{dfn}

\begin{dfn}\qquad
\begin{enumerate}
    \item For $\vec{\nu}\in\mathcal{C}_\kappa$, define $X_{\vec{\nu}}=\{x\in 2^\kappa:(\forall n\in\omega)(\nu_n\nsubseteq x)\}$.
    \item Let $\mathsf{nwd}_\divideontimes(\kappa)$ be the ideal generated by $\{X_{\vec{\nu}}:\vec{\nu}\in\mathcal{C}_\kappa\}$.
\end{enumerate}
\end{dfn}

\begin{lemma}\label{nwd_divide_def}\qquad
\begin{enumerate}
    \item $\mathsf{nwd}_\divideontimes(\kappa)$ is a proper ideal.
    \item For any infinite set $X\subseteq\kappa$ and $f:X\to 2$, there is $\vec{\nu}\in\mathcal{C}_{\kappa}$ such that
    \begin{equation*}
        \bigcap_{\alpha\in X}B_{f(\alpha)}^\alpha\subseteq X_{\vec{\nu}}
    \end{equation*}
    \item For any $\vec{\nu}\in\mathcal{C}_\kappa$, $X_{\vec{\nu}}$ is a nowhere dense set.
\end{enumerate}
\end{lemma}

\begin{proof}
(1) is clear. (2) Let $Y\subseteq X$ be a subset with order type $\omega$, and let $Y=\{\alpha_n:n\in\omega\}$ be its increasing enumeration. Then define $\nu_n=\{(\alpha_n,1-f(\alpha_n))\}$. To prove (3), fix $\vec{\nu}\in\mathcal{C}_{\kappa}$ and let $s:F\to 2$ be such that $F\in[\kappa]^{<\omega}$, whose corresponding basic set we denote by $\langle s\rangle=\{x\in 2^\kappa:s\subseteq x\}$. Note that there are at most finitely many $n\in\omega$ such that $F\cap dom(\nu_n)\neq\emptyset$, so we can choose $k_0\in\omega$ such that $F\cap dom(\nu_{k_0})=\emptyset$. Define $r=s\cup \nu_{k_0}$, and note that $\langle r\rangle\cap X_{\vec{\nu}}=\emptyset$. Thus, any basic open set can be extended to a basic open set disjoint from the set $X_{\vec\nu}$, which means $X_{\vec{\nu}}$ is a nowhere dense set.
\end{proof}

\begin{crl}
For any infinite cardinal $\kappa$, $\tau_\kappa^\diamond\leq_K\tau_\kappa\leq_K\mathsf{nwd}_\divideontimes(\kappa)\leq_K\mathsf{nwd}(\kappa)$.    
\end{crl}

\begin{proof}
The first Kat\v{e}tov inequality is just because $\tau_\kappa^\diamond\subseteq\tau_\kappa$. The second inequality follows from (2) of Lemma \ref{nwd_divide_def}. For the third inequality, by (3) of Lemma \ref{nwd_divide_def} $X_{\vec{\nu}}$ is a nowhere dense set for any $\vec{\nu}\in\mathcal{C}_\kappa$, which implies  $\mathsf{nwd}_\divideontimes(\kappa)\subseteq\mathsf{nwd}(\kappa)$.
\end{proof}

\noindent\textbf{Remark.} Note that $\mathsf{nwd}_\divideontimes(\omega)=\mathsf{nwd}$.

\begin{thm}\label{critical_ideal}
Let $\mathcal{U}$ be a non-principal ultrafilter on $\omega$. Then, for any uncountable cardinal $\kappa$, $[\kappa]^{<\omega}\leq_T\mathcal{U}$ if and only if $\tau_\kappa^\diamond\leq_K \mathcal{U}^*$.    
\end{thm}

\begin{proof}
Let $\mathcal{U}$ be an ultrafilter with Tukey type above $[\kappa]^{<\omega}$, and let $\{D_\alpha:\alpha\in\kappa\}\subseteq\mathcal{U}$ be a witness of this. Define a map $f:\omega\to 2^\kappa$ as follows:
\begin{equation*}
    f(n)(\alpha)=1\Longleftrightarrow n\in D_\alpha
\end{equation*}
Clearly, $f[\omega]\in[2^\kappa]^{\omega}$. Let now $X\in \tau_\kappa^\diamond$. Since $\mathcal{U}$ is an ultrafilter, and each $X\in\tau_\kappa^\diamond$ is contained in the union of finitely many sets of the form $\chi(A,f)$, we can assume there are $A\in [\kappa]^\omega$ and $g:A\to 2$ such that $X\subseteq\{x\in 2^\kappa:x\upharpoonright A=g\}$. Assume $f^{-1}[X]\in\mathcal{U}$, so we also have $Z=f^{-1}[\chi(A,g)]\in\mathcal{U}$. Let us see that $Z=\bigcap_{\alpha\in A}D_\alpha$. Suppose there are $n\in Z$ and $\alpha_0\in A$ such that $n\notin D_{\alpha_0}$. Then, by definition of $f$ and the choice of $g$, we have that $f(n)(\alpha_0)=0=g(\alpha_0)$. Now, for any $m\in Z\setminus\{n\}$, we have $f(m)(\alpha_0)=g(\alpha_0)=0$, which means $m\notin D_{\alpha_0}$. Therefore, $Z\cap D_{\alpha_0}=\emptyset$, which means $Z\notin \mathcal{U}$, which contradicts our assumption of $Z\in\mathcal{U}$. Thus, we have $Z\subseteq\bigcap_{\alpha\in A}D_{\alpha}$. Now, suppose there is $n\in\bigcap_{\alpha\in A}D_\alpha\setminus Z$. Then, by the definition of $Z$, there is $\alpha_0\in A$ such that  $f(n)(\alpha_0)\neq g(\alpha_0)$. Since $n\in \bigcap_{\alpha\in A}D_\alpha$, we have $f(n)(\alpha_0)=1$, which implies $g(\alpha_0)=0$. By the choice of $Z$, we have that for any $k\in Z$, $f(k)(\alpha_0)=g(\alpha_0)=0$, which means that for any $k\in Z$, $k\notin D_{\alpha_0}$, that is $Z\cap D_{\alpha_0}=\emptyset$, which, again, is a contradiction. Therefore, we have $\bigcap_{\alpha\in A}D_\alpha= Z$. Then we have that $\bigcap_{\alpha\in A} D_\alpha\in\mathcal{U}$, which contradicts that $\{D_\alpha:\alpha\in\kappa\}$ is a witness for $[\kappa]^{<\omega}\leq_T\mathcal{U}$. Since these contradictions come from assuming that $f^{-1}[X]\in\mathcal{U}$ for some $X\in\tau_{\kappa}^\diamond$, we should have that for all $X\in\tau_\kappa^\diamond$, $f^{-1}[X]\in\mathcal{U}^*$, that is, $\tau_\kappa^\diamond\leq_K\mathcal{U}^*$.

Let us assume now that $\tau_\kappa^\diamond\leq_K\mathcal{U}^*$ and let $\varphi:\omega\to 2^\kappa$ be a witness of this. For each $\alpha\in \kappa$, we have $\varphi^{-1}[B_0^\alpha]\in\mathcal{U}$ or $\varphi^{-1}[B_1^\alpha]\in\mathcal{U}$. Define $g:\kappa\to 2$ such that for each $\alpha\in \kappa$, $\varphi^{-1}[B_{g(\alpha)}^\alpha]\in\mathcal{U}$. We claim that $\{\varphi^{-1}[B_{g(\alpha)}^\alpha]:\alpha\in \kappa\}$ witnesses $[\kappa]^{<\omega}\leq_T\mathcal{U}$. Suppose there is $A\in[\kappa]^{\omega}$ such that $Z=\bigcap_{\alpha\in A}\varphi^{-1}[B_{g(\alpha)}^\alpha]\in\mathcal{U}$. Then we have, $$Z=\bigcap_{\alpha\in A}\varphi^{-1}[B_{g(\alpha)}^\alpha]=\varphi^{-1}\left[\bigcap_{\alpha\in A}B^{\alpha}_{g(\alpha)}\right]$$ but clearly we have $\bigcap_{\alpha\in A}B_{g(\alpha)}^\alpha=\chi(A,g\upharpoonright A)\in \tau_\kappa$, also note that $\varphi[Z]\subseteq\chi(A,g\upharpoonright A)$, so $\varphi[Z]\in\tau_\kappa^\diamond$. Since $\varphi$ witnesses $\tau_\kappa^\diamond\leq_K\mathcal{U}^*$, we have $Z=\varphi^{-1}[\chi(A,g\upharpoonright A)]\in\mathcal{U}^*$, which contradicts our assumption of $Z\in\mathcal{U}$. Therefore, $\{\varphi^{-1}[B_{g(\alpha)}^\alpha]:\alpha\in k\}$ witnesses $[\kappa]^{<\omega}\leq_T\mathcal{U}$.
\end{proof}

\begin{crl}
Let $\mathcal{U}$ be an ultrafilter on $\omega$. If $\mathsf{nwd}_\divideontimes(\kappa)\leq_K\mathcal{U}^*$, then $[\kappa]^{<\omega}\leq_T\mathcal{U}$.
\end{crl}

When talking about ultrafilters on $\omega$, the ideals $\tau_\kappa^{\diamond}$ and $\tau_\kappa$ look essentially the same,
in the sense that an ultrafilter is Kat\v{e}tov above $\tau_\kappa^\diamond$ if and only if it is Kat\v{e}tov above $\tau_\kappa$. However, $\mathsf{nwd}(\kappa)$ seems to be strictly higher since it may not be reached by Tukey top ultrafilters, as pointed in Corollary \ref{nwd_Tukey_not_nwd_kappa}; we need a few facts to prove this.

\begin{lemma}
If $\mathsf{nwd}\nleq_K\mathcal{U}^*$, then $\mathsf{nwd}(\kappa)\nleq\mathcal{U}^*$.
\end{lemma}

\begin{proof}
In this lemma, given $\vec{\nu}\in\mathcal{C}_\omega$, to distinguish between nowhere dense sets in $2^\omega$ and nowhere dense sets in $2^\kappa$, we write $X_{\vec{\nu}}^\omega$ for the definition of $X_{\vec{\nu}}$ in $2^\omega$ and $X_{\vec{\nu}}^\kappa$ for the definition of $X_{\vec{\nu}}$ in $2^\kappa$. A similar argument to that of (3) from Lemma \ref{nwd_divide_def} proves that $X_{\vec{\nu}}^\kappa$ is nowhere dense for any $\vec{\nu}\in\mathcal{C}_\omega$.

Now, let us see that $\mathsf{nwd}(\kappa)\leq_K\mathcal{U}^*$ implies $\mathsf{nwd}\leq_K\mathcal{U}^*$. Let $\varphi:\omega\to 2^\kappa$ be a witness of $\mathsf{nwd}(\kappa)\leq_K\mathcal{U}^*$ and define $\varphi^*:\omega\to2^\omega$ as $\varphi^*(n)=\varphi(n)\upharpoonright\omega$. Let $\vec{\nu}\in\mathcal{C}_\omega$ be arbitrary. Then $\varphi^{-1}[X_{\vec{\nu}}^\kappa]\in\mathcal{U}^*$. But we also have that $(\varphi^*)^{-1}[X^\omega_{\vec\nu}]=\varphi^{-1}[X_{\vec{\nu}}^\kappa]$, which implies $(\varphi^*)^{-1}[X_{\vec{\nu}}^\omega]\in\mathcal{U}^*$. Therefore, $\mathsf{nwd}\leq_K\mathcal{U}^*$.
\end{proof}

The following theorem was first proved by O. Guzm\'an-Gonz\'alez, we include the result here with his permission.

\begin{thm}[O. Guzm\'an-Gonz\'alez]\label{osvaldo_theorem_Z}
For any ultrafilter $\mathcal{U}$ such that $\mathcal{Z}\leq_K\mathcal{U}^*$, $\mathcal{U}$ is Tukey top.
\end{thm}

\begin{crl}
For any ultrafilter $\mathcal{U}$, if $\mathcal{Z}\leq_K\mathcal{U}^*$, then $\tau_\mathfrak{c}^\diamond\leq_K\mathcal{U}^*$.
\end{crl}

\begin{thm}[see \cite{hong_zhang_relations}]
Assume $\mathsf{CH}$. There is an ultrafilter $\mathcal{U}$ such that:
\begin{enumerate}
    \item $\mathsf{nwd}\nleq\mathcal{U}^*$ and,
    \item $\mathcal{Z}\leq \mathcal{U}^*$.
\end{enumerate}
\end{thm}

\begin{crl}\label{nwd_Tukey_not_nwd_kappa}
Assume $\mathsf{CH}$. There is an ultrafilter $\mathcal{U}$ such that 
\begin{enumerate}
    \item $\tau_{\mathfrak{c}}^\diamond\leq_K\mathcal{U}^*$.
    \item $\mathsf{nwd}(\mathfrak{c})\nleq\mathcal{U}^*$.
\end{enumerate}
\end{crl}

Our aim is to prove that consistently all ultrafilters are Kat\v{e}tov above $\tau_{\omega_1}^\diamond$ (in the first model); and  all ultrafilters are Kat\v{e}tov above $\tau_{\mathfrak{c}}^\diamond$ (in the second model). In doing so, the ideal $\mathsf{nwd}_\divideontimes(\kappa)$ will be useful. Given a model $V$ of $\mathsf{ZFC}+\diamondsuit(S)+\mathsf{GCH}$, we are going to construct generic extensions $V[G_0]$ and $V[G_1]$ such that:

\begin{enumerate}
    \item [($\star$)] $V[G_0]\vDash$ For any non-principal ultrafilter $\mathcal{U}$, $\mathsf{nwd}_\divideontimes(\omega_1)\leq_K\mathcal{U}^*$.
        \item [($\star\star$)] $V[G_1]\vDash$ For any non-principal ultrafilter $\mathcal{U}$, $\mathsf{nwd}_\divideontimes(\kappa)\leq_K\mathcal{U}^*$ and $2^\omega=\kappa>\omega_1$.
\end{enumerate}

By Theorem \ref{critical_ideal}, this will imply the next theorem.

\begin{thm}\label{main_theorem_reformulation}\qquad
    \begin{enumerate}
        \item It is consistent that for any non-principal ultrafilter $\mathcal{U}$ on $\omega$, $\tau_{\omega_1}^\diamond\leq_K\mathcal{U}^*$.
        \item It is consistent that for any non-principal ultrafilter $\mathcal{U}$ on $\omega$, $\tau_{\mathfrak{c}}^\diamond\leq_K\mathcal{U}^*$.
    \end{enumerate}
\end{thm}

One natural question that arises is the existence of an ultrafilter which is not a $\mathsf{nwd}$-ultrafilter but it is a $\mathsf{nwd}(\kappa)$-ultrafilter for some $\kappa>\omega$. Propositions \ref{not_the_same_1} and \ref{not_the_same_2} answer this question up to some degree, but we first need some preliminary facts, some for which we include the proofs for completeness.

In the rest of this section, if $(X,\tau)$ is a topological space, $\mathsf{nwd}(\tau)$ denotes the family of $\tau$-nowhere dense subsets of $X$, which in all the spaces that appear below is a proper ideal. If $\lambda$ is a cardinal, $\mathsf{nwd}(\lambda)$ keeps the previously defined meaning. For any $A\subseteq X$, $\tau-int(A)$ denotes the interior of the set $A$ in the topology $\tau$. The notation $\mathsf{nwd}(\tau)$ will be used only in this section.

\begin{dfn}
Let $(X,\tau)$ be a topological space. We say that $(X,\tau)$ is irresolvable if any two dense subsets of $X$ have non-empty intersection; otherwise it is called resolvable.
We say that $(X,\tau)$ is strongly irresolvable if any non-empty open subset $U$ is irresolvable with the subspace topology.
\end{dfn}

\begin{lemma}\label{nwd_positive_interior}
Let $(\omega,\tau)$ be a strongly irresolvable space. Then, for any $X\in\mathsf{nwd}(\tau)^+$, $\tau-int(X)\neq\emptyset$.
\end{lemma}

\begin{proof}
Let $X\subseteq$ be a $\mathsf{nwd}(\tau)$-positive set, and let $V=\tau-int(\overline{X})$, so $V$ is a non-empty set. Suppose $\tau-int(X)=\emptyset$. Then, for any $\tau$-open set $W\subseteq V$, we have $W\setminus X\neq\emptyset$ and $X\cap W\neq\emptyset$, which means that $X$ is dense in $V$ and $V\setminus X$ is dense in $V$, which is in contradiction with $(\omega,\tau)$ being strongly irresolvable.
\end{proof}

\begin{lemma}[see \cite{ferreira_hrusak_resolvability_extraresolvability}]\label{nwd_saturated_irresolvable}
If $(\omega,\tau)$ is strongly irresolvable, then $\mathsf{nwd}(\tau)^*$ is a saturated filter.
\end{lemma}

\begin{proof}
Let $(\omega,\tau)$ be strongly irresolvable and let $\{A_\alpha:\alpha\in\lambda\}\subseteq\mathsf{nwd}(\tau)^+$ be such that for all different $\alpha,\beta\in\lambda$, $A_\alpha\cap A_\beta\in\mathsf{nwd}(\tau)$. Then, by Lemma \ref{nwd_positive_interior}, for each $\alpha<\lambda$, there is $B_\alpha\in\tau$ such that $B_\alpha\subseteq A_\alpha$. Note that for different $\alpha,\beta\in\lambda$, we should have $B_\alpha\cap B_\beta=\emptyset$. Thus, $\{B_\alpha:\alpha\in\lambda\}$ is a disjoint family, and $\omega$ is countable, so $\lambda$ should be countable.
\end{proof}

\begin{lemma}\label{nwd_filter_of_dense}
Let $(\omega,\tau)$ be a strongly irresolvable space. Then, $\mathsf{nwd}(\tau)^*$ is the filter of $\tau$-dense subsets of $\omega$.    
\end{lemma}
\begin{proof}
Clearly, if $X\in \mathsf{nwd}(\tau)$, then $\omega\setminus X$ is a $\tau$-dense subset of $\omega$. Now, pick $D\subseteq\omega$ a $\tau$-dense set and let us see that $\omega\setminus X\in\mathsf{nwd}(\tau)$. Suppose $\omega\setminus X$ is $\mathsf{nwd}(\tau)$-positive. By Lemma \ref{nwd_positive_interior}, we have $V=\tau-int(\omega\setminus X)\neq\emptyset$, which implies $V\cap X=\emptyset$, which is contradicts that $X$ is $\tau$-dense. Therefore, $\omega\setminus X\in\mathsf{nwd}(\tau)$.
\end{proof}

\begin{lemma}[see \cite{union_of_resolvable}]\label{union_resolvable}
The union of resolvable spaces is resolvable.
\end{lemma}

\begin{lemma}\label{strongly_irresolvable_subspace}
Let $(\omega,\tau)$ be an irresolvable space. Then there is an open subset $X$ such that $(X,\tau\upharpoonright X)$ is strongly irresolvable.
\end{lemma}

\begin{proof}
Let $\Omega$ be the collection of all $\tau$-open subsets which are resolvable, and define $V_0=\bigcup\Omega$. By Lemma \ref{union_resolvable} the union of resolvable spaces is resolvable, so we have that $V_0$ is an open resolvable subspace of $(\omega,\tau)$. Let $C_0=\omega\setminus V_0$. $C_0$ is closed, and since $(\omega,\tau)$ is irresolvable, we should have $\tau-int(C_0)\neq\emptyset$. Define $V_1=\tau-int(C_0)$. Let $V\subseteq V_1$ be an arbitrary $\tau\upharpoonright V_1$-open set. Since $V_1$ is $\tau$-open, $V$ is $\tau$-open as well, and it can not be resolvable since if it were, we would have $V\subseteq V_0$, which implies $V\cap V_1=\emptyset$, which is impossible.
\end{proof}

\begin{dfn}\qquad
\begin{enumerate}
    \item $Dense(\mathbb{Q})$ denotes the family of all dense subsets of $\mathbb{Q}$.
    \item A filter $\mathcal{F}\subseteq Dense(\mathbb{Q})$ is a maximal filter on $\mathbb{Q}$ if for all $X\in Dense(\mathbb{Q})$, either $X\in\mathcal{F}$ or there is $Y\in \mathcal{F}$ such that $X\cap Y\notin Dense(\mathbb{Q})$.
\end{enumerate}
\end{dfn}

\begin{dfn}
Let $\mathcal{F}\subseteq\mathbb{Q}$ be a filter. We say that $\mathcal{F}$ is $\mathbb{Q}$-basically generated if $\mathcal{F}$ has a base $\mathcal{B}$ such that for all $\langle X_n:n\in\omega\rangle\subseteq\mathcal{B}$, a sequence of open sets $\langle U_n:n\in\omega\rangle$ in $\mathbb{Q}$, $V$ an open set and $Y\in\mathcal{B}$ such that $$\lim_{n\in\omega}X_n\cap U_n=Y\cap V$$ there are $A\in[\omega]^\omega$ and $Z\in\mathcal{B}$ such that $Z\cap V\subseteq\bigcap_{n\in A}X_n\cap V$.

\end{dfn}

\begin{prp}\label{not_the_same_1}
Assume $\mathsf{CH}$. Then there is a maximal $p$-filter $\mathcal{F}$ which is $\mathbb{Q}$-basically generated.
\end{prp}

\begin{proof}
Let $E$ denote the set of all ordinals $\alpha$ in $\omega_1$ which are limit ordinals or there is a limit ordinal $\beta$ and a natural number $n$ such that $\alpha=\beta+2n$, and let $O=\omega_1\setminus E$.

Denote by $\tau_0$ the usual topology on $\mathbb{Q}$ and let $\mathcal{B}_0$ a countable base of clopen sets for $\tau_0$. Let $\{A_\alpha:\alpha\in E\}$ be an enumeration of $Dense(\mathbb{Q})$, and let $\{(\vec{B}_\alpha,\vec{U}_\alpha,V_\alpha):\alpha\in O\}$ be an enumeration of $(Dense(\mathbb{Q}))^\omega\times (\tau_0\setminus\{\emptyset\})^\omega\times (\tau_0\setminus\{\emptyset\})$ such that each element repeats cofinally often. We construct an $\subseteq$-increasing sequence $\langle\mathcal{F}_\alpha:\alpha\in\omega_1\rangle$ of countable families as follows:
\begin{enumerate}
    \item Let $\mathcal{F}_0$ be the closure of $\{A_0\}\cup \mathsf{fin}(\mathbb{Q})^*$ under finite intersections, where $\mathsf{fin}(\mathbb{Q})$ is the ideal of finite subsets of $\mathbb{Q}$.
    \item Suppose $\mathcal{F}_\alpha$ has been defined and we have to define $\mathcal{F}_{\alpha+1}$. Then:
    \begin{enumerate}
        \item \underline{$\alpha\in E$}: if for all $X\in\mathcal{F}_\alpha$ we have that $A_\alpha\cap X\in Dense(\mathbb{Q})$, then define $\mathcal{F}_{\alpha+1}$ to be the closure of $\{A_{\alpha+1}\}\cup\mathcal{F}_\alpha$ under finite intersections. Otherwise, there is $X\in\mathcal{F}_\alpha$, such that $A_\alpha\cap X\notin Dense(\mathbb{Q})$, and we can define $\mathcal{F}_{\alpha+1}=\mathcal{F}_\alpha$.
        
        \item \underline{$\alpha\in O$}: let $\vec{B}_\alpha=\langle B_n^\alpha:n\in\omega\rangle$ and $\vec{U}_\alpha=\langle U_n^\alpha:n\in\omega\rangle$. We have the following subcases:
        \begin{enumerate}
            \item If for some $n\in\omega$, $B_n^\alpha\notin \mathcal{F}_\alpha$, define $\mathcal{F}_{\alpha+1}=\mathcal{F}_\alpha$.
            
            \item If for all $n\in\omega$, $B^\alpha_n\in\mathcal{F}_\alpha$ but there is no $Y_\alpha\in \mathcal{F}_\alpha$ such that $\lim_{n\to\infty} B_n^\alpha\cap U_n^\alpha=Y_\alpha\cap V_\alpha$, then define $\mathcal{F}_{\alpha+1}=\mathcal{F}_\alpha$. 
            
            \item If for all $n\in\omega$, $B^\alpha_n\in\mathcal{F}_\alpha$ and there is $Y_\alpha\in\mathcal{F}_\alpha$ such that $\lim_{n\to\infty} B_n^\alpha\cap U_n^\alpha=Y_\alpha\cap V_\alpha$, and $(\vec{B}_\alpha,\vec{U}_\alpha,V_\alpha)$ has appeared before at step $\gamma<\alpha$ and this resulted in $\mathcal{F}_{\gamma+1}$ having more elements than $\mathcal{F}_\gamma$, then define $\mathcal{F}_{\alpha+1}=\mathcal{F}_\alpha$.
            
            \item Otherwise, let $Y_\alpha\in\mathcal{F}_\alpha$ be such that $\lim_{n\in\omega}B_n^\alpha\cap U_n^\alpha=Y_\alpha\cap V_\alpha$. We are going to construct $A\in[\omega]^{\omega}$ and $X_0\in Dense(\mathbb{Q})$ such that:
            \begin{enumerate}
                \item For all $Y\in\mathcal{F}_\alpha$, $X_0\cap Y\in Dense(\mathbb{Q})$.
                \item $X_0\cap V_\alpha\subseteq \bigcap_{n\in A}B_n^\alpha\cap V_\alpha$.
            \end{enumerate}
            Let $\langle F_n:n\in\omega\rangle$ be an enumeration of $\mathcal{F}_\alpha$, $\langle W_n^0:n\in\omega\rangle$ be the enumeration of all basic open sets such that $W\subseteq V_\alpha$, and let $\langle W_n^1:n\in\omega\rangle$ be the enumeration of all open basic sets such that $W\subseteq\mathbb{Q}\setminus \overline{V}_\alpha$. Note that the second list may be empty, in which case $V_\alpha$ is an open dense subset of $\mathbb{Q}$. Let us first assume $int(\mathbb{Q}\setminus \overline{V}_\alpha)\neq\emptyset$. Pick $q_0\in F_0\cap Y_\alpha\cap W_0^1$ and $r_0\in F_0\cap Y_\alpha\cap W_0^0$, and let $k_0\in\omega$ be big enough so for all $l\ge k_0$, $r_0\in B_l^\alpha\cap V_\alpha$ (such $k_0$ exists because $\lim_{n\in\omega}B_n^\alpha\cap U_n^\alpha=Y_\alpha\cap V_\alpha$ and $W_0^0\subseteq V_\alpha$). Suppose we have defined $\{q_k:k\leq n\},\{r_k:k\leq n\}$ and $\{k_l:l\leq n\}$. Then, pick $q_{n+1}$ and $r_{n+1}$ such that:
            \begin{equation*}
                q_{n+1}\in Y_\alpha\cap\bigcap_{l\leq n+1}F_l\cap W_{n+1}^1
            \end{equation*}
            \begin{equation*}
                r_{n+1}\in Y_\alpha\cap V_\alpha\cap \bigcap_{l\leq n+1}F_l\cap W_{n+1}^0
            \end{equation*}
            then, let $k_{n+1}\in\omega$ be big enough so for all $m\ge k_{n+1}$, $r_{n+1}\in B_{m}^\alpha$.
            Let $X_0=\{q_n:n\in\omega\}\cup\{r_n:n\in\omega\}$ and $A=\{k_n:n\in\omega\}$. It is clear from the construction that for all $Z\in\mathcal{F}_\alpha$, $X_0\cap Z\in Dense(\mathbb{Q})$. Also, note that $X_0\cap V_\alpha=\{r_n:n\in\omega\}$, and $X_0\cap V_\alpha\subseteq\bigcap_{n\in A}B_n^\alpha\cap V_\alpha$. Now, let $Z\in Dense(\mathbb{Q})$ be a pseudointersection of $\{X_0\}\cup\mathcal{F}_\alpha$. Finally, we define $F_{\alpha+1}$ as the closure of $\mathcal{F}_\alpha\cup\{X_0,Z\}$ under finite intersections.
            The case when $int(\mathbb{Q}\setminus \overline{V}_\alpha)=\emptyset$ is handled similarly but we omit to select the points $q_n$.
        \end{enumerate}
    \end{enumerate}
    \item At limit steps $\alpha$ we take $\mathcal{G}_\alpha=\bigcup_{\beta<\alpha}\mathcal{F}_\beta$. If $A_\alpha$ has dense intersection with each element from $\mathcal{G}_\alpha$, then define $\mathcal{F}_\alpha$ to be the closure of $\mathcal{G}_{\alpha}\cup\{A_\alpha\}$ under finite intersections. Otherwise let $\mathcal{F}_\alpha=\mathcal{G}_\alpha$.
\end{enumerate}
Let $\mathcal{F}=\bigcup_{\alpha\in\omega_1}\mathcal{F}_\alpha$. By the construction we have that $\mathcal{F}$ generates a maximal filter on $Dense(\mathbb{Q})$. Let $\mathcal{H}$ be the filter generated by $\mathcal{F}$. Note that $\mathsf{nwd}(\tau_0)^*\subseteq\mathcal{H}$. From the construction it also follows that $\mathcal{F}$ has the following properties:
\begin{enumerate}
    \item[$\circledast$] Let $\langle X_n:n\in\omega\rangle$ be a sequence of elements from $\mathcal{F}$, $\langle U_n:n\in\omega\rangle$ a sequence of open sets, $V$ an open set and $Y\in\mathcal{F}$ such that $\lim_{n\in\omega}X_n\cap U_n=X\cap V$. Then there are $A\in[\omega]^\omega$ and $Z\in\mathcal{F}$ such that $Z\cap V\subseteq\bigcap_{n\in A}X_n\cap V$.
    \item [$\circledast\circledast$] For any $\{X_n:n\in\omega\}\subseteq\mathcal{F}$, there is $X_\omega\in\mathcal{F}$ such that for all $n\in\omega$, $X_\omega\subseteq^* X_n$.
\end{enumerate}

Thus, $\mathcal{H}$ is a $\mathbb{Q}$-basically generated $p$-filter on $Dense(\mathbb{Q})$.
\end{proof}

\begin{prp}\label{not_the_same_2}
Let $\mathcal{H}$ be a maximal $p$-filter $\mathbb{Q}$-basically generated filter. Then there is an ultrafilter $\mathcal{U}$ extending $\mathcal{H}$ such that:
\begin{enumerate}
    \item $\mathsf{nwd}\leq_{K}\mathcal{U}^*$.
    \item $\mathcal{U}$ is a basically generated ultrafilter. Thus, $\tau_{\omega_1}^\diamond\nleq_K\mathcal{U}^*$, which implies $\mathsf{nwd}(\omega_1)\nleq_K\mathcal{U}^*$.
\end{enumerate}
\end{prp}

\begin{proof}
Following the same notation of the previous lemma, $\tau_0$ denotes the usual topology on $\mathbb{Q}$ and $\mathcal{B}_0$ a countable base of clopens. Let $\mathcal{H}$ be a maximal $p$-filter $\mathbb{Q}$-basically generated filter and let $\mathcal{F}$ be a base for $\mathcal{H}$ witnessing $\mathcal{H}$ is $\mathbb{Q}$-basically generated. Let $\tau_\mathcal{H}$ the topology on $\mathbb{Q}$ generated by $\mathcal{B}_\mathcal{H}=\{U\cap X:X\in\mathcal{H}\land U\in\mathcal{B}_0\}$. Then $(\mathbb{Q},\tau_\mathcal{H})$ is a Hausdorff irresolvable space. By Lemma \ref{strongly_irresolvable_subspace}, there is a $\tau_\mathcal{H}$-open $W'\subseteq\mathbb{Q}$ such that $(W',\tau_\mathcal{H}\upharpoonright W')$ is strongly irresolvable. Moreover, since $(W',\tau_{\mathcal{H}}\upharpoonright W')$ is strongly irresolvable and $W'$ is $\tau_\mathcal{H}$-open, we can take $W=B_*\cap D_*\subseteq W'$, where $B_*\in\mathcal{B}_0$ and $D_*\in\mathcal{H}$, and such that $(W,\tau_\mathcal{H}\upharpoonright W)$ is strongly irresolvable; define $\tau_W=\tau_\mathcal{H}\upharpoonright W$ and let $\mathcal{B}_W$ be the restriction of the base $\mathcal{B}_{\mathcal{H}}$ to $W$. Note that $\mathcal{B}_W=\{B\in\mathcal{B}_\mathcal{H}:B\subseteq W\}=\{B\cap X:X\in\mathcal{H}\land X\subseteq D_*\land B\in\mathcal{B}_0\land B\subseteq B_* \}$.

Let $\mathcal{U}$ be any ultrafilter on $W$ extending $\mathcal{H}\upharpoonright W$. We claim that $\mathcal{U}$ is not a $\mathsf{nwd}$-ultrafilter. First we need to prove a few facts. By Lemma \ref{nwd_filter_of_dense}, we have that $\mathsf{nwd}(\tau_W)^*$ is the filter of $\tau_W$-dense sets, and it is not difficult to see that each element $X\in\mathcal{H}\upharpoonright W$ is a $\tau_W$-dense set, so we have $\mathcal{H}\upharpoonright W\subseteq(\mathsf{nwd}(\tau_W))^*$.

Let us see that $\mathsf{nwd}(\tau_0)\upharpoonright W=\mathsf{nwd}(\tau_0\upharpoonright W)$. Pick $X\in\mathsf{nwd}(\tau_0)\upharpoonright W$ and let $V\in\tau_0\upharpoonright W$ be an open set. Then $V= V'\cap W$, where $V'\in\tau_0$. Note that we can assume $V'\subseteq B_*$. Since $X\in\mathsf{nwd}(\tau_0)\upharpoonright W$, we have $X\in\mathsf{nwd}(\tau_0)$, so there is $V''\in\tau_0$ such that $V''\subseteq V'$ and $X\cap V''=\emptyset$. Then $V''\cap W$ is a $\tau_0\upharpoonright W$-open subset of $V$ and has empty intersection with $X$. Thus any $\tau_0\upharpoonright W$-open set can be extended to an open set having empty intersection with $X$, which means $X\in\mathsf{nwd(\tau_0\upharpoonright W)}$. Pick now $X\in \mathsf{nwd}(\tau_0\upharpoonright W)$. It is enough to prove that $X$ is $\tau_0$-nowhere dense. Let $V\in\mathcal{B}_0$ be arbitrary. If $V\cap(\mathbb{Q}\setminus B_*)\neq\emptyset$, then clearly $V\cap(\mathbb{Q}\setminus B_*)$ is an open set having empty intersection with $X$, so let us assume that $V\subseteq B_*$. Then $V\cap W$ is a $\tau_0\upharpoonright W$-open set, so there is $V'\in\tau_0\upharpoonright W$ such that $V'\subseteq V\cap W$ and $V'\cap X=\emptyset$. Let $V''\in\tau_0$ be such that $V'=V''\cap W$, we can assume $V''\subseteq B_*$. Note that we should have $V''\cap X=\emptyset$: otherwise, $V'\cap X=(V''\cap W)\cap X\neq\emptyset$, which contradicts the choice of $V'$. We also have that $V''\subseteq V$: otherwise, $V''\setminus V$ is a non-empty $\tau_0$-open (because $V\in\mathcal{B}_0$), which implies $W\cap (V''\setminus V)\neq\emptyset$; since $V'=V''\cap W$, we get $V'\setminus (V\cap W)\neq\emptyset$, which contradicts $V'' \cap W=V'\subseteq V\cap W$. Therefore, $V''\subseteq V$ is $\tau_0$-open and $V''\cap X=\emptyset$. Thus, any $\tau_0$-open set can be extended to an open set having empty intersection with $X$, so $X$ is $\tau_0$-nowhere dense. Thus, $\mathsf{nwd}(\tau_0\upharpoonright W)\subseteq\mathsf{nwd}(\tau_0)\upharpoonright W$.

Now we continue to prove that $\mathcal{U}$ is not $\mathsf{nwd}$-ultrafilter. By the construction of $\mathcal{H}\upharpoonright W$, we have that $(\mathsf{nwd}(\tau_0)\upharpoonright W)^*\subseteq\mathcal{H}\upharpoonright W=(\mathsf{nwd}(\tau_W))^*$, so we have $\mathsf{nwd}(\tau_0)\upharpoonright W\subseteq\mathsf{nwd}(\tau_W)$. 
By the previous paragraph, $\mathsf{nwd}(\tau_0\upharpoonright W)=\mathsf{nwd}(\tau_0)\upharpoonright W$. Thus, $\mathsf{nwd}(\tau_0\upharpoonright W)^*\subseteq\mathcal{H}\upharpoonright W\subseteq\mathcal{U}$, which means that $\mathcal{U}$ is not a $\mathsf{nwd}$-ultrafilter.

We claim that $\mathcal{U}$ is basically generated. Define $\mathcal{B}_\mathcal{U}=\mathcal{U}\cap\{X\cap V\cap W:X\in\mathcal{F}\land V\in\tau_0\setminus\{\emptyset\}\}$. First, let us see that $\mathcal{B}_\mathcal{U}$ is a base for $\mathcal{U}$. Let $X\in\mathcal{U}$ be arbitrary. We have that $\tau_W-int(X)\in\mathcal{U}$ (since $X\setminus (\tau_W-int(X))\in\mathsf{nwd}(\tau_W)$). Consider $G=\{U\in\mathcal{B}_\mathcal{H}:U\subseteq X\}$, and let $\mathcal{A}=\{A_n:n\in\omega\}\subseteq G$ a maximal family of disjoint basic open sets. Then, for each $n\in\omega$, there are $B_n\in\mathcal{B}_0$ and $Z_n\in\mathcal{H}$ such that $A_n=Z_n\cap B_n$. Note that since $Z_n\in\mathcal{H}$ for all $n\in\omega$, we should have that $B_n\cap B_m=\emptyset$ for all different $n,m\in\omega$. Now, let $Z_\omega\in\mathcal{F}$ be a pseudointersection of $\{Z_n:n\in\omega\}$.
Note that we have that for each $n\in\omega$, $(Z_\omega\cap B_n)\setminus A_n$ is at most finite. Thus, $\left(Z_\omega\cap\bigcup_{n\in\omega}B_n\right)\setminus (\tau_W-int(X))\in\mathsf{nwd}(\tau_W)$, and actually, $N=\left( Z_\omega\cap\bigcup_{n\in\omega}B_n\right)\setminus (\tau_W-int(X))\in\mathsf{nwd}(\tau_0\upharpoonright W)$. Let $Z_\omega'\in\mathcal{F}$ be such that $Z_\omega'\subseteq Z_\omega\setminus N$. Now define $V=Z_\omega'\cap\bigcup_{n\in\omega}B_n$. By construction we have $V\subseteq \tau_W-int(X)$. Let us prove that $V\in\mathcal{U}$. Let $A_\omega=\bigcup_{n\in\omega}B_n$ and note that $N'=\tau_W-int(X)\setminus A_\omega\in\mathsf{nwd}(\tau_W)$. Then $\tau_W-int(X)\setminus N'\in\mathcal{U}$. But we also have $\tau_W-int(X)\setminus N'\subseteq A_\omega$, and therefore, it follows that $A_\omega\in\mathcal{U}$. Since $Z_\omega'\in\mathcal{F}\subseteq\mathcal{U}$, we also have, $V=A_\omega\cap Z_\omega'\in\mathcal{U}$. Finally, note that $A_\omega\cap Z_\omega'\in\mathcal{B}_\mathcal{U}$.

It remains to prove that $\mathcal{B}_\mathcal{U}$ witnesses that $\mathcal{U}$ is basically generated. Let now $\{X_n:n\in\omega\}\subseteq\mathcal{B}_\mathcal{U}$ be a sequence converging to some point $Y\in\mathcal{B}_\mathcal{U}$. For each $n\in\omega$, let $A_n\in\mathcal{F}$ and $V_n\in\tau_0$ such that $X_n=A_n\cap V_n$, and let $Y_*\in\mathcal{F}$ and $W_*\in\tau_0$ be such that $Y=Y_*\cap W_*$. Then, since $\mathcal{F}$ witnesses $\mathcal{U}$ is $\mathbb{Q}$-basically generated, we find $D\in[\omega]^\omega$ and $Z\in\mathcal{F}$ such that $Z\cap W_*\subseteq\bigcap_{n\in D}A_n\cap W_*$. Note that we have that $Z\cap W_*\in\mathcal{U}$: since $Y\in\mathcal{U}$ and $Y\subseteq W_*$, we have $W_*\in\mathcal{U}$; also, $Z\in\mathcal{F}\subseteq\mathcal{U}$, so $Z\cap W_*\in \mathcal{U}$.
This implies $\bigcap_{n\in D}A_n\cap W_*\in\mathcal{U}$. Thus, $\mathcal{U}$ is basically generated.
\end{proof}

A more satisfactory answer would be to know whether it is consistent that there is no $\mathsf{nwd}$-ultrafilter yet there is an ultrafilter which is not Tukey above $[\omega_1]^{<\omega}$.

\begin{question}
Is it consistent that there is no $\mathsf{nwd}$-ultrafilter but there is an ultrafilter not Tukey above $[\omega_1]^{<\omega}$?
\end{question}

\begin{question}
Is there a non-Tukey top ultrafilter in Shelah's model for no $\mathsf{nwd}$-ultrafilters?    
\end{question}

We don't know the answer to this questions. On the other hand, Theorem \ref{osvaldo_theorem_Z} says that any ultrafilter which is Kat\v{e}tov above the asymptotic density zero ideal $\mathcal{Z}$ is a Tukey top ultrafilter. So, Propositions \ref{not_the_same_1} and \ref{not_the_same_2} also show that the ideal $\mathsf{nwd}$ does not enjoy of this property.

\section{Suitable filters.}\label{suitable_filters}

This section introduces some terminology that will be used in sections \ref{one_tukey_type_start} to \ref{getting_all_together}. Section \ref{no_basically_generated} does not require everything that is presented here, and all that it needs is that whenever $\mathcal{U}$ is an ultrafilter, then Player I has no winning strategy in any of the games defined in 2.3, 2.12, 2.13 and 2.14 below; these are classical and well known results.

\begin{dfn}
Let $\mathcal{F}$ be a filter on $\omega$. An $\mathcal{F}$-partition is a sequence $\vec{E}=\langle E_n:n\in\omega\rangle$ such that:
\begin{enumerate}
    \item For all $n\in\omega$, $E_n\in\mathcal{F}^*$.
    \item For all different $n,m$, $E_n\cap E_m=\emptyset$.
    \item $dom(\vec{E})=\bigcup_{n\in\omega}E_n\in\mathcal{F}$.
    \item For all $n\in\omega$, $\min(E_n)<\min(E_{n+1})$.
\end{enumerate}
We call the enumeration given by (4) the canonical enumeration of $\vec{E}$.

Given two $\mathcal{F}$-partitions $\vec{E},\vec{F}$, define $\vec{F}\sqsubseteq\vec{E}$ if and only if:
\begin{enumerate}
    \item $dom(\vec{F})\subseteq dom(\vec{E})$.
    \item $\vec{E}\upharpoonright dom(\vec{F})$ is finer than $\vec{F}$, that is, each element from $\vec{F}$ is the union of elements from $\vec{E}$.
\end{enumerate}
The family of all $\mathcal{F}$-partitions is denoted by $\mathcal{P}_{\mathcal{F}}$.
\end{dfn}

\begin{dfn}
A filter $\mathcal{F}$ is $\kappa$-saturated if $\mathcal{P}(\omega)/\mathcal{F}^*$ has the $\kappa$-c.c.  If $\kappa=\omega_1$ then we just say $\mathcal{F}$ is saturated.
\end{dfn}

Let us recall the following game introduced by Shelah in \cite{con_i_u}.

\begin{dfn}
Let $\mathcal{F}$ be a filter on $\omega$. In the two players game $\mathcal{G}(\mathcal{F})$, Player I and Player II play elements from $\mathcal{F}$ as follows:

\begin{center}
\begin{tabular}[t]{l |c |c |c |c |c |r}
Player I & $A^1_0$ &      $ $      & $A^1_1\subseteq A^2_0$ & $ $    & $\ldots\ \ \ $  &\\
\hline
Player II & $ $          & $A^2_0\subseteq A^1_0$ & $ $         & $A^2_1\subseteq A^1_1$ & & $\ldots\ \ \ $\\
\end{tabular}
\end{center}

Player II wins if and only if $\bigcup_{n\in\omega}A^1_n\setminus A^2_n\in\mathcal{F}$.
\end{dfn}

It is well known that if $\mathcal{F}$ filter is $2^\omega$-saturated, then Player I has no winning strategy in the game $\mathcal{G}(\mathcal{F})$. The importance of saturation is that it allows several forcing notions to be proper and have all relevant properties necessary for us. However, in our applications we will find that is not possible to always get saturation properties\footnote{For example, if we have a $2^\omega$-saturated filter not saturated in the Sacks model, it is not possible to make sure that we get a reflection of such filter to intermediate steps of the iteration which is $2^\omega$-saturated.}, so a more general condition is necessary. In this section we formulate a generalization which is sufficient for our purposes. This property is intermediate between non-meagerness and $2^\omega$-saturation, and we will find that any $2^\omega$-saturated filter in the Sacks model reflects to one having the stated property on an $\omega_1$-club subset of $\omega_2$.

\begin{dfn}
Define $2^{<\omega}_*=2^{<\omega}\setminus\{\emptyset\}$, and for positive $k\in\omega$, $2^{\leq k}_*=\bigcup_{0<j\leq k}2^{j}$. We define a well ordering on $2^{<\omega}_*$ as follows: given different $\sigma,\tau\in 2^{<\omega}_*$ define $\sigma<\tau$ if and only if one of the following happens:
\begin{enumerate}
    \item $\vert\sigma\vert<\vert\tau\vert$.
    \item $\vert\sigma\vert=\vert\tau\vert$ and $\sigma\sqsubseteq_{lex}\tau$, where $\sqsubseteq_{lex}$ is the lexicographical order.
\end{enumerate}
This ordering induces a well order on $2_*^{<\omega}$ of type $\omega$, say $\{\sigma_n:n\in\omega\}$, which we call the canonical enumeration of $2^{<\omega}_*$.
\end{dfn}

\begin{dfn}\label{F_tree_def}
Let $\mathcal{F}$ be a filter on $\omega$. A  function $\tau:2^{<\omega}_*\to\mathcal{F}\times\mathcal{F}$ is an $\mathcal{F}$-tree whenever the following holds:
\begin{enumerate}
    \item For all $\sigma\in 2^{<\omega}_*$, $\tau(\sigma)=(\tau_0(\sigma),\tau_1(\sigma))$.
    \item For all $\sigma\in 2^{<\omega}_*$, $\tau_1(\sigma) \subseteq\tau_0(\sigma)$.

    \item For all $\sigma_n,\sigma_m\in2^{<\omega}_*$, if $n<m$, then $\tau_0(\sigma_m)\subseteq\tau_1(\sigma_n)$(not necessarily proper inclusion when $m=n+1$).
\end{enumerate}
\end{dfn}

\begin{dfn}
Let $\tau$ be an $\mathcal{F}$-tree. For $f\in 2^\omega$, define $\tau\text{-}br_0(f)$ and $\tau\text{-}br_1(f)$ as follows:
\begin{enumerate}
    \item $\tau\text{-}br_0(f)=\bigcup_{n>0}(\tau_1(f\upharpoonright n)\setminus\tau_0(f\upharpoonright(n+1)))$.
    \item $\tau\text{-}br_1(f)=\bigcup_{n>0}(\tau_0(f\upharpoonright n)\setminus \tau_1(f\upharpoonright n))$.
\end{enumerate}
\end{dfn}

\begin{dfn}
We say that a filter $\mathcal{F}$ is a suitable filter if for any $\mathcal{F}$-tree there is $f\in 2^\omega$ such that $\tau\text{-}br_0(f)\in\mathcal{F}$.
\end{dfn}

It may seem that the definition of suitable filter is somewhat artificial, however, as we will see below, they are rather a natural generalization of $2^\omega$-saturated filters, and this class is included in the class of filters for which Player I has no winning strategy in any instance of the game $\mathcal{G}(\mathcal{F})$. Our interest on suitable filters is two-folded. On the one side, they serve as intermediaries to reflect useful properties of $2^{\omega}$-saturated filters along an iteration (see Lemma \ref{suitable_reflection} below). On the other side, they will allow us to construct fusion sequences
when the game version of fusion is not available.

\begin{prp}
Any $2^\omega$-saturated filter is a suitable filter.    
\end{prp}

\begin{proof}
First note that for any $f\in 2^\omega$, $\tau\text{-}br_0(f)\notin \mathcal{F}$ if and only if $\tau\text{-}br_1(f)\in\mathcal{F}^+$. Let $\tau:2^{<\omega_*}\to\mathcal{F}\times\mathcal{F}$ be an $\mathcal{F}$-tree. It is not hard to see that for any different $f_0,f_1\in 2^\omega$, $\tau\text{-}br_1(f_0)\cap \tau\text{-}br_1(f_1)\in\mathcal{F}^*$. Thus, if for all $f\in 2^\omega$ we have $\tau\text{-}br_1(f)\in\mathcal{F}^+$, then $\mathcal{F}$ cannot be $2^\omega$-saturated.
\end{proof}

In the next lemma we consider a strategy for Player I as a tree $\Sigma\subseteq\mathcal{F}^{<\omega}$ such that even levels\footnote{We start to count levels from $0$, that is, the starting move corresponding to Player I is on level $0$.} correspond to what Player I plays and odd levels correspond to what Player II plays. So, even levels are uniquely determined by the initial sequence preceding them, while odd levels include all possibilities Player II is allowed to play. If we have that $\sigma\in\Sigma$ is a position in the game in which last movement was given by Player II, then $\Sigma(\sigma)$ denotes the answer of Player I according to $\Sigma$. $\Sigma(\emptyset)$ denotes the starting move of Player I.

\begin{lemma}\label{F_subtree}
Let $\mathcal{F}$ be a filter. Let $\Sigma$ be a strategy for Player I in the game $\mathcal{G}(\mathcal{F})$. Then $\Sigma$ ``contains" an $\mathcal{F}$-tree.
\end{lemma}

\begin{proof}
Before going into the proof, let us make precise what we mean by ``contains an $\mathcal{F}$-tree". The way we have defined $\mathcal{F}$-tree and the definition of strategy we have taken, clearly imply that there cannot be inclusion relation between these two objects, so let us explain. First, let $\tau$ be an $\mathcal{F}$-tree, and for each $f\in 2^\omega$, define a sequence $s_f\in\mathcal{F}^\omega$ as follows:
\begin{enumerate}
    \item $t_1^f=\langle\tau_0(f\upharpoonright 1),\tau_1(f\upharpoonright 1)\rangle$
    \item For $n>1$, $t_{n+1}^f={t_n^f}^\frown\langle\tau_0(f\upharpoonright n),\tau_1(f\upharpoonright n)\rangle$ 
    \item Let $s_f=\bigcup_{n\ge 1}t_n^f$.
\end{enumerate}
Let now $\Sigma$ be the strategy of Player I, and for each $f\in 2^\omega$, define $g_f=\Sigma(\emptyset)^\frown s_f$. Now define $G(\tau)=\{g_f\upharpoonright n: f\in 2^\omega\land n\in\omega\land n>0\}$. Then ``$\Sigma$ contains an $\mathcal{F}$-tree" means that $G(\tau)\subseteq \Sigma$, for some $\mathcal{F}$-tree $\tau$.

What remains is to construct the $\mathcal{F}$-tree. We proceed by recursion on $n$ along the canonical enumeration of $2^{<\omega}_*$, and on the side, we construct $\langle t_{\sigma_i}:i\in\omega\rangle\subseteq\mathcal{F}^{<\omega}$.
\begin{enumerate}
    \item We start by defining $\tau_0(\sigma_0)=\Sigma(\emptyset)=A_\emptyset$, $\tau_1(\sigma_0)=\Sigma(A_\emptyset,A_\emptyset)$ and $t_{\sigma_0}=(A_\emptyset,A_\emptyset,\tau_1(\sigma_0))$.
    \item Then make $\tau_0(\sigma_1)=\tau_1(\sigma_0)$, and $\tau_1(\sigma_1)=\Sigma(A_\emptyset,\tau_0(\sigma_1))$(which is the same as $\Sigma(A_\emptyset,\tau_1(\sigma_0))$). Then make $t_{\sigma_1}=(A_\emptyset,\tau_0(\sigma_1),\tau_1(\sigma_1))$.
    \item For $n>1$, suppose $\tau(\sigma_n)$ and $t_{\sigma_n}$ have been defined, so we have to define $\tau(\sigma_{n+1})$ and $t_{\sigma_{n+1}}$. Then make $\tau_0(\sigma_{n+1})=\tau_1(\sigma_n)$. To define $\tau_1(\sigma_{n+1})$, let $j<n+1$ be such that $\sigma_{n+1}\upharpoonright(\vert\sigma_{n+1}\vert-1)=\sigma_j$, and define $\tau_1(\sigma_{n+1})=\Sigma({t_{\sigma_j}}^\frown\tau_0(\sigma_{n+1}))$. Then define $t_{\sigma_{n+1}}={t_{\sigma_j}}^\frown\tau(\sigma_{n+1})=$ ${t_{\sigma_{j}}}^\frown\langle\tau_0(\sigma_{n+1}),\tau_1(\sigma_{n+1})\rangle$.
\end{enumerate}

\end{proof}

\begin{crl}\label{PI_no_ws}
If $\mathcal{F}$ is a suitable filter, then Player I has no winning strategy in $\mathcal{G}(\mathcal{F})$.
\end{crl}
\begin{proof}
Let $\Sigma$ be a strategy for Player I in the game $\mathcal{G}(\mathcal{F})$. By Lemma \ref{F_subtree}, there is an $\mathcal{F}$-tree $\tau$ such that $G(\tau)\subseteq\Sigma$. Since $\mathcal{F}$ is suitable, there is $f\in 2^\omega$ such that $\tau\text{-}br_0(f)\in\mathcal{F}$, and then $f$ induces a run on the game in which Player I follows $\Sigma$ but Player II wins as follows:
\begin{center}
\begin{tabular}[t]{l |c |c |c |c |c|c |c|r}
I & $A_\emptyset$ &      $ $      & $\tau_1(f\upharpoonright 1)$ & $ $    & $\tau_1(f\upharpoonright 2) $  & $ $ &$\tau_1(f\upharpoonright 3)$\\
\hline
II & $ $          & $\tau_0(f\upharpoonright 1)$ & $ $ &  $\tau_0(f\upharpoonright 2)$ & & $\tau_0(f\upharpoonright 3)$& $ $& $\ldots$\\
\end{tabular}
\end{center}
Where $A_\emptyset=\Sigma(\emptyset)$. 
\end{proof}

\begin{lemma}\label{suitable_reflection}
Let $\mathbb{S}_{\omega_2}$ be the countable support iteration of the Sacks forcing. Let $\mathcal{F}$ be an $\omega_2$-saturated filter on $\omega$. Then the set 
\begin{equation}
\{\alpha\in\omega_2:\dot{\mathcal{F}}\cap V[G_\alpha]\in V[G_\alpha]\text{ is a suitable filter in }V[G_\alpha]\}
\end{equation}
contains an $\omega_1$-club set (that is, a subset closed under increasing sequences of length $\omega_1$).
\end{lemma}

\begin{proof}
Let $\mathcal{C}\subseteq\omega_2$ be an $\omega_1$-club such that for all $\alpha\in\mathcal{C}$, $\mathbb{S}_\alpha\Vdash\dot{\mathcal{F}}\cap V[G_\alpha]\in V[G_\alpha]$. Let $\dot{\mathcal{F}}_\alpha$ be a $\mathbb{S}_\alpha$-name for $\dot{\mathcal{F}}\cap V[G_\alpha]$, and $\dot{\Gamma}_\alpha$ be a name for the collection of all $\dot{\mathcal{F}}_\alpha$-trees in $V[G_\alpha]$. Then $\dot{\Gamma}_\alpha$ is forced to have cardinality $\omega_1$, and since each $\dot{\mathcal{F}}_\alpha$-tree is countable, we can replace $\dot{\Gamma}_\alpha$ by a set of cardinality $\omega_1$ of countable names which are either evaluated as $\dot{\mathcal{F}}_\alpha$-trees or as the empty set. Since $\dot{\mathcal{F}}$ is forced to be suitable, for each $\dot{\tau}\in\dot{\Gamma}_\alpha$ there is an ordinal $\beta_{\dot{\tau}}\in\mathcal{C}$ such that $V[G_{\beta_{\dot{\tau}}}]$ contains a real $\dot{f}_{\dot{\tau}}$ which witness $\dot{\mathcal{F}}$ is suitable on the tree $\dot{\tau}$. Then we can find $\beta_\alpha\in\mathcal{C}$ such that for all $\dot{\tau}\in\dot{\Gamma}_\alpha$, $V[G_{\beta_\alpha}]$ contains a real witnessing $\dot{\mathcal{F}}$ is suitable for the $\mathcal{F}_\alpha$-tree $\dot{\tau}$. This defines an increasing function $h:\mathcal{C}\to\mathcal{C}$. The closure points of $h$(i.e., the set of $\alpha\in\mathcal{C}$ for which $h[\alpha]\subseteq\alpha$) is the required $\omega_1$-club. 
\end{proof}

Finally, we will need the following variations of the game $\mathcal{G}(\mathcal{F})$.
\begin{dfn}
Let $\mathcal{F}$ be a filter on $\omega$. Let $\mathcal{M}\prec H(\theta)$ be a countable elementary submodel such that $\mathcal{F}\in\mathcal{M}$. The game $\mathcal{G}_{\mathcal{M}}(\mathcal{F})$ between Player I and Player II, is defined as: both players play elements from $\mathcal{M}\cap\mathcal{F}$ as follows:

\begin{center}
\begin{tabular}[t]{l |c |c |c |c |c |r}
Player I & $A^1_0$ &      $ $      & $A^1_1\subseteq A^2_0$ & $ $    & $\ldots\ \ \ $  &\\
\hline
Player II & $ $          & $A^2_0\subseteq A^1_0$ & $ $         & $A^2_1\subseteq A^1_1$ & & $\ldots\ \ \ $\\
\end{tabular}
\end{center}

Player II wins if and only if $\bigcup_{n\in\omega}A^1_n\setminus A^2_n\in\mathcal{F}$.
\end{dfn}

\begin{dfn}
Let $\mathcal{F}$ be a filter on $\omega$. Let $\mathcal{M}\prec H(\theta)$ be a countable elementary submodel such that $\mathcal{F}\in\mathcal{M}$. The game $\mathcal{G}_{\mathcal{M}}^\mathcal{P}(\mathcal{F})$ between Player I and Player II, is defined as: both players play elements from $\mathcal{M}\cap\mathcal{P}_\mathcal{F}$ as follows:

\begin{center}
\begin{tabular}[t]{l |c |c |c |c |c |r}
Player I & $\vec{E}^1_0$ &      $ $      & $\vec{E}^1_1\sqsubseteq \vec{E}^2_0$ & $ $    & $\ldots\ \ \ $  &\\
\hline
Player II & $ $          & $\vec{E}^2_0\sqsubseteq \vec{E}^1_0$ & $ $         & $\vec{E}^2_1\sqsubseteq \vec{E}^1_1$ & & $\ldots\ \ \ $\\
\end{tabular}
\end{center}

Player II wins if and only if $\bigcup_{n\in\omega}dom(\vec{E}^1_n)\setminus dom(\vec{E}^2_n)\in\mathcal{F}$.

The game $\mathcal{G}^{\mathcal{P}}(\mathcal{F})$ is defined in a similar way as $\mathcal{G}_{\mathcal{M}}^{\mathcal{P}}(\mathcal{F})$, but we omit the elementary submodel and players are not restricted to play $\mathcal{F}$-partitions in any elementary submodel.
\end{dfn}

\begin{prp}
Let $\mathcal{F}$ be a suitable filter. Then Player I has no winning strategy in any of the games $\mathcal{G}_{\mathcal{M}}(\mathcal{F})$, $\mathcal{G}^{\mathcal{P}}(\mathcal{F})$ and  $\mathcal{G}_{\mathcal{M}}^{\mathcal{P}}(\mathcal{F})$.
\end{prp}

\begin{proof}
An easy adaptation of the proofs of Lemma \ref{F_subtree} and Corollary \ref{PI_no_ws}.
\end{proof}

\section{A model with no basically generated ultrafilters.}\label{no_basically_generated}

\begin{dfn}
Let $\vec{E}\in\mathcal{P}_{\mathcal{F}}$ be a $\mathcal{F}$-partition. For 
$A\subseteq\omega$ we define $\vec{E}*A=\{E^p_k:k\notin A\}$. In particular, for $n\in\omega$, $\vec{E}*n=\{E^p_k:k\ge n\}$.
\end{dfn}

\begin{dfn}
Let $\mathcal{U}$ be a filter. Define a forcing notion $\mathbb{P}_0(\mathcal{U})$ as follows: $p\in\mathbb{P}_0(\mathcal{U})$ if and only if:
\begin{enumerate}
    \item $p=(h_p,\vec{E}_p)$
    \item $\vec{E}_p=\langle E_n^p:n\in\omega\rangle$ is an $\mathcal{U}$-partition.
    \item $h_p:\omega\setminus dom(\vec{E}_p)\to 2$.
\end{enumerate}
The order on $\mathbb{P}_{0}(\mathcal{U})$ is defined as $q\leq p$ if and only if:
\begin{enumerate}
    \item $\vec{E}_q\sqsubseteq \vec{E}_p$.
    \item $h_p\subseteq h_q$
    \item If $E_k^p\subseteq dom(h_q)$, the for all $j\in E_k^p$, $h_q(j)=h_q(\min(E_k^p))$.
\end{enumerate}
\end{dfn}

\begin{dfn}
Let $\mathcal{U}$ be a filter and $\kappa$ and infinite cardinal. Let $\bigotimes_{\alpha\in\kappa}\mathbb{P}_0(\mathcal{U})$ be the countable support product of $\kappa$-many copies of $\mathbb{P}_{0}(\mathcal{U})$. Define the forcing $\mathbb{P}_{\kappa}(\mathcal{U})$ as $p\in\mathbb{P}_\kappa(\mathcal{U})$ if and only if:
\begin{enumerate}
    \item $p\in\bigotimes_{\alpha\in\kappa}\mathbb{P}_\alpha(\mathcal{U})$.
    \item $0\in dom(p)$.
    \item For all $\alpha\in dom(p)$, there is $F^p_\alpha\subseteq\omega$ such that $\vec{E}_{p(\alpha)}=\vec{E}_{p(0)}*F_\alpha^p$.
\end{enumerate}
The order is given by $q\leq p$ if and only if for all $\alpha\in dom(p)$, $q(\alpha)\leq p(\alpha)$.

For $p\in\mathbb{P}_\kappa(\mathcal{U})$, if $\alpha\in dom(0)\setminus\{0\}$, let $N(p,\alpha)$ be the unique subset of $\omega$ such that $\vec{E}_{p(\alpha)}=\vec{E}_{p(0)}*N(p,\alpha)$, equivalently, $(h_{p(\alpha)},\vec{E}_{p(0)}*N(p,\alpha))\in\mathbb{P}_0(\mathcal{U})$.
\end{dfn}

We make some conventions about notation here. Given an $\mathcal{F}$-partition $\vec{E}$, then:
\begin{enumerate}
    \item We think of $\vec{E}=\langle E_n:n\in\omega\rangle$ in two different ways, one is as a sequence and the other is like the set of its elements, $\{E_n:n\in\omega\}$. The context should be enough to avoid ambiguity.
    \item If $\{F_n:n\in\lambda\}$ is a disjoint family, where $\lambda\in\omega\cup\{\omega\}$, and such that for all $n\in\lambda$, $F_n\in\mathcal{F}^*$ and $F_n\cap dom(\vec{E})$, then, by $\{F_n:n\in\lambda\}\cup\vec{E}$ we mean the $\mathcal{F}$-partition which results from the reenumeration (that is, the canonical enumeration) of $\{F_n:n\in\lambda\}\cup\{E_n:n\in\omega\}$.
    \item If we write $\langle F_n:n\in\lambda\rangle\cup\vec{E}$ we mean the same partition as the previous clause.
\end{enumerate}
Sometimes we write $\{F_n:n\in\lambda\}\cup\vec{E}$ or $\langle F_n:n\in\lambda\rangle\cup\vec{E}$. Typically, we use this notation only for $\lambda\in\omega$.

The next definition needs a justification, which appears after the definition.

\begin{dfn}
Let $\dot{G}$ be a generic filter for $\mathbb{P}_\kappa(\mathcal{U})$. For $\alpha\in\kappa$, define $\dot{H}_\alpha:\omega\to 2$ as:
\begin{equation*}
    \dot{H}_\alpha=\bigcup\{h_{p(\alpha)}:p\in \dot{G}\}
\end{equation*}
Then, for $\alpha\in\kappa$ and $i\in 2$, define $\dot{D}_i^\alpha=\dot{H}_\alpha^{-1}[\{i\}]$.

Also, define $\dot{\varphi}:\omega\to 2^\kappa$ as:
\begin{equation*}
    \dot{\varphi}(n)(\alpha)=1\iff n\in \dot{D}^\alpha_1
\end{equation*}
\end{dfn}

It is not totally clear that for each $\alpha\in \kappa$ the function $\dot{H}_\alpha$ has $\omega$ as domain. First note that if $\alpha\in dom(p)$, then for any $k\in\omega$, the sets of conditions $q\leq p$ for which $k\in dom(h_q)$ is dense below $q$. Thus, it is enough to prove that given any $p\in\mathbb{P}_\kappa(\mathcal{U})$ and $A\in[\kappa]^\omega$
disjoint from $dom(p)$, there is an extension $q\leq p$ such that $A\subseteq dom(q)$. This follows easily by just making a copy of $p$ in $A$. More exactly, let $\langle\alpha_n:n\in\omega\rangle$ be an enumeration of $A$ and $\langle\beta_n:n\in\omega\rangle$ be an enumeration of $dom(p)$; then define $q:A\cup dom(p)\to\mathbb{P}_0(\mathcal{U})$ as $q(\beta_n)=p(\beta_n)$ and $q(\alpha_n)=p(\beta_n)$.\footnote{The reader may worry that this operation seems somewhat restrictive; the lemmas below introduced and some arguments used in the proofs should make clear how to construct more general conditions.}

The following is easy to proof.

\begin{lemma}
$\mathbb{P}_\kappa(\mathcal{F})$ forces $\dot{\varphi}[\omega]$ is a dense subset of $2^\kappa$.  
\end{lemma}

\begin{dfn}Let $p=(f_p,\vec{E}_p)\in\mathbb{P}_0(\mathcal{U})$ be a condition. Then
\begin{enumerate}
    \item For $n\in dom(\vec{E}_p)$, $n/\vec{E}_p$ denotes the element 
    ${E}_k^p$ from $\vec{E}_p$ such that $n\in E_k^p$.
    \item $A_p=\{\min(E_n^p):n\in\omega\}=\{a_n^p:n\in\omega\}$ is the set of canonical representa\-tives of $\vec{E}_p$, ordered by its increasing enumeration.
\end{enumerate}

\end{dfn}

\begin{dfn}
Let $p,q\in\mathbb{P}_0(\mathcal{U})$ be two conditions and $n\in\omega$. Define $q\leq_n p$ if and only if $q\leq p$ and for $i\leq n$, $a_i^q/\vec{E}_q=a_i^p/\vec{E}_p$.    
\end{dfn}

\begin{dfn}
Let $p=(h_p,\vec{E}_p)\in\mathbb{P}_0(\mathcal{U})$ be a condition and $n\in\omega$. Let $\{a_i^p:i\leq n\}$ be the first $n+1$ elements of $A_p$. Let $\nu \in 2^{n+1}$. Then define $p^{[\nu]}$ as the condition $q=(h_q,\vec{E}_q)\leq p$ such that:
\begin{enumerate}
    \item $dom(\vec{E}_q)=dom(\vec{E}_p*(n+1))=\bigcup_{k> n}E_k^p$.
    \item $\vec{E}_q=\vec{E}_p\upharpoonright dom(\vec{E}_q)$.
    \item For each $i\leq n$, $h_q\upharpoonright E_i^p\equiv \nu(i)$.
\end{enumerate}
\end{dfn}

\begin{dfn}
Let $\kappa$ be an infinite cardinal. We denote by $(\kappa)^{<\omega}$ the set of all finite inyective sequence of ordinals in $\kappa$ whose first element is 0.
\end{dfn}

\begin{dfn}
Let $p,q\in\mathbb{P}_\kappa(\mathcal{U})$ be conditions, $n\in\omega$ and $\vec{S}=\langle \alpha_0,\ldots,\alpha_{n}\rangle\in (\kappa)^{<\omega}$ such that $\vert \vec{S}\vert=n+1$ and $rng(\vec{S})\subseteq dom(p)$. Then define $q\leq_n^{\vec{S}}p$ if and only if:
\begin{enumerate}
    \item $q\leq p$
    \item For all $i\leq n$, $q(\alpha_i)\leq_{n-i}p(\alpha_i)$.
\end{enumerate}
\end{dfn}

\begin{dfn}
Let $p\in\mathbb{P}_\kappa(\mathcal{F})$ be a condition, $n\in\omega$ a natural number and $\vec{S}\in (\kappa)^{<\omega}$ such that $rng(\vec{S})\subseteq dom(p)$ and $\vert S\vert=n+1$, let $\vec{m}=\langle m_i:i\leq n\rangle\in \omega^{<\omega}$ is a sequence of positive natural numbers. Then for each $\nu\in \Pi_{j<n+1}2^{m_j}$ we define $p[\vec{S},\vec{\nu}]$ as follows:
\begin{enumerate}
    \item $dom(p{[\vec{S},\vec{\nu}]})=dom(p)$.
    \item For all $\alpha\notin rng(\vec{S})$, $p{[\vec{S},\vec{\nu}]}(\alpha)=p(\alpha)$.
    \item For each $j<n$, $p{[\vec{S},\vec{\nu}]}(\alpha_j)=p(\alpha_j)^{[\vec{\nu}(j)]}$.
\end{enumerate}
\end{dfn}

\begin{lemma}\label{n_S_extension}
Let $p\in\mathbb{P}_\kappa(\mathcal{U})$ be a condition and $\mathcal{A}\subseteq\mathbb{P}_\kappa(\mathcal{U})$ a maximal antichain. Then for any  $\vec{S}=\langle \alpha_0,\ldots,\alpha_n\rangle\in(\kappa)^{<\omega}$ such that $rng(\vec{S})\subseteq dom(p)$, there is $q\leq_n^{\vec{S}} p$ such that for all $\vec{\nu}\in\Pi_{j<n+1}2^{n+1-j}$, $q{[\vec{S},\vec{\nu}]}$ extends some condition in $\mathcal{A}$. 
\end{lemma}

\begin{proof}
Fix $p\in\mathbb{P}_\kappa(\mathcal{U})$, $n\in\omega$ and $\vec{S}=\langle\alpha_0,\ldots,\alpha_n\rangle\in(\kappa)^{<\omega}$. Let $\langle\vec{\nu}_i:i\leq m\rangle$ be an enumeration of $\Pi_{j<n+1}2^{n+1-j}$. We construct a sequence of conditions $\langle q_i:i\leq m\rangle$ such that:
\begin{enumerate}
    \item $q_0\leq_n^{\vec{S}}p$.
    \item For all $i<m$, $q_{i+1}\leq_n^{\vec{S}} q_i$.
    \item For all $i\leq m$, $q_i{[\vec{S},\vec{\nu}_i]}$ extends some condition in $\mathcal{A}$.
\end{enumerate}

Suppose we have defined such sequence. Note that for each $l\leq m$, $q_{m}[\vec{S},\vec{\nu}_l]\leq q_l[\vec{S},\vec{\nu}_l]$. Thus, for each $l\leq m$, $q_m[\vec{S},\vec{\nu}_l]$ extends some condition in $\mathcal{A}$.

We start by choosing $r_0\leq p{[\vec{S},\vec{\nu}_0]}$ extending some condition in $\mathcal{A}$, and then define $q_0$ as follows:
\begin{enumerate}
    \item If $\beta\notin rng(\vec{S})$, then $q_0(\beta)=r_0(\beta)$.
    \item For $j\leq n$, $q_0(\alpha_j)$ is defined as:
    \begin{enumerate}
        \item $\vec{E}_{q_0(\alpha_j)}=\vec{E}_{p(\alpha_j)}\upharpoonright(n-j+1)\cup\vec{E}_{r_0(\alpha_j)}$.
        \item $h_{q_0(\alpha_j)}= h_{r_0(\alpha_j)}\upharpoonright(\omega\setminus dom(\vec{E}_{q_0(\alpha_j)}))$.
    \end{enumerate}
\end{enumerate}

Suppose we have defined $q_l$ for some $l<m$. Then let $r_{l+1}\leq q_l[\vec{S},\vec{\nu}_{l+1}]$ be a condition extending one element from $\mathcal{A}$. Now define $q_{l+1}$ as follows:
\begin{enumerate}
    \item If $\beta\notin rng(\vec{S})$, then $q_{l+1}(\beta)=r_{l+1}(\beta)$.
    \item For $j\leq n$, $q_{l+1}(\alpha_j)$ is defined as:
    \begin{enumerate}
        \item $\vec{E}_{q_{l+1}(\alpha_j)}=\vec{E}_{p(\alpha_j)}\upharpoonright(n-j+1)\cup\vec{E}_{r_{l+1}(\alpha_j)}$.
        \item $h_{q_{l+1}(\alpha_j)}= h_{r_{l+1}(\alpha_j)}\upharpoonright(\omega\setminus dom(\vec{E}_{q_{l+1}(\alpha_j)}))$.
    \end{enumerate}
\end{enumerate}

This finishes the construction

\end{proof}

\begin{dfn}\label{extension_to_compatible_partitions}
Let $q\in\mathbb{P}_0(\mathcal{U})$ be a condition. Let $n\in\omega$ and $\vec{E}\sqsubseteq \vec{E}_{q}$ such that $dom(\vec{E})\cap E^{q}_k=\emptyset$ for each $k\leq n$. Then define the condition $q^{(n,\vec{E})}$ as follows:
\begin{enumerate}
    \item $\vec{E}_{q^{(n,\vec{E})}}=\{E_k^{q}:k\leq n\}\cup\vec{E}$.
    \item For $l\in dom(h_{q^{(n,\vec{E})}})\setminus dom(h_q)$, define $h_{q^{(n,\vec{E})}}(l)=0$
\end{enumerate}

Given $\vec{E}\in\mathcal{P}_\mathcal{U}$ such that $\vec{E}\sqsubseteq\vec{E}_q$, we also define the condition $q^{(*,\vec{E})}$ as:
\begin{enumerate}
    \item [(1)] $\vec{E}_{q^{(*,\vec{E})}}=\vec{E}$.
    \item [(2)] For $l\in dom(h_{q^{(*,\vec{E})}})\setminus dom(h_q)$, define $h_{q^{(*,\vec{E})}}(l)=0$.
\end{enumerate}
\end{dfn}

\begin{dfn}
Let $p\in\mathbb{P}_\kappa(\mathcal{F})$ be a condition. We say that $p$ is a simple condition if there is $\langle\alpha_j:j\in\omega\rangle$ enumeration of  $dom(p)$ such that $\alpha_0=0$, and for all $j\in\omega$, $\vec{E}_{p(\alpha_j)}=\vec{E}_{p(0)}*j$. In this case we refer to $\langle\alpha_j:j\in\omega\rangle$ as the canonical enumeration of $dom(p)$.
\end{dfn}

\begin{lemma}\label{simple_conditions_dense}
The set of simple conditions is dense.
\end{lemma}

\begin{proof}
Let $p\in\mathbb{P}_\kappa(\mathcal{F})$ be a condition. Let $\langle\alpha_n:n\in\omega\rangle$ be an enumeration of $dom(p)$ such that $\alpha_0=0$. Construct a sequence $\langle A_n:n\in\omega\rangle$ of subsets of $\omega$ such that:
\begin{enumerate}
    \item $\forall n\in\omega$, $A_n\varsubsetneq
 A_{n+1}$.
    \item $\forall n\in\omega$, $N(\alpha_{n+1},p)\subseteq A_n$.
    \item $\forall n\in\omega$, $\bigcup_{k\in A_n}E_k^p\in\mathcal{F}^*$.
    \item $\omega=\bigcup_{n\in\omega}A_n$.
\end{enumerate}
Now define $E_0=\bigcup_{k\in A_0}E_k^p$, and for $n>0$, $E_n=\bigcup_{k\in A_n\setminus A_{n-1}}E_k^p$. Let $\vec{E}=\langle E_n:n\in\omega\rangle$, and note that $\vec{E}$ is an $\mathcal{F}$-partition such that $dom(\vec{E})=dom(\vec{E}_{p(0)})$. Also note that for each $n\in\omega$, $\vec{E}*n\sqsubseteq \vec{E}_{p(0)}*N(\alpha_n,p)=\vec{E}_{p(\alpha_n)}$.

Now define a condtion $q\leq p$ as follows:
\begin{enumerate}
    \item $dom(q)=dom(p)$.
    \item For all $n\in\omega$, $q(\alpha_n)=p(\alpha_n)^{(*,\vec{E}*n)}$.
\end{enumerate}
It is clear from the construction that $q$ is a simple condition and $q\leq p$.
\end{proof}

\begin{lemma}\label{n_S_simple_extension}
Let $p\in\mathbb{P}_\kappa(\mathcal{U})$ be a simple condition and $\langle \alpha_i:i\in\omega\rangle$ be the canonical enumeration of $dom(p)$. Fix $n\in\omega$ and let $\vec{S}=\langle \alpha_i:i\leq n\rangle$. Let $\mathcal{A}$ be a maximal antichain. Then, there is a simple condition $q\leq_n^{\vec{S}} p$ such that for all $\vec{\nu}\in\Pi_{j<n+1}2^{n+1-j}$, $q{[\vec{S},\vec{\nu}]}$ extends some condition in $\mathcal{A}$
\end{lemma}

\begin{proof}
Let $p,n,\vec{S}$ as stated. First, apply Lemma \ref{n_S_extension} to obtain a condition $q'\leq_{n}^{\vec{S}}p$ such that for all $\vec{\nu}\in\Pi_{j<n+1}2^{n+1-j}$, $q'[\vec{S},\vec{\nu}]$ extends some condition in $\mathcal{A}$. Now apply a similar construction to that of Lemma \ref{simple_conditions_dense} to construct a simple condition $q\leq_n^{\vec{S}}q'$.
\end{proof}

\begin{thm}\label{pp_iterated}
Any countable support iteration of proper forcings having the Sacks property has the Sacks property.    
\end{thm}

The construction in the next lemma makes use of a bookkeeping device. Let $\{P_n:n\in\omega\}$ be a partition (which is intended to be fixed from now on) of $\omega$ into infinite sets such that for all $n\in\omega$, $n+2\subseteq\bigcup_{k\leq n}P_k$, so $\{0,1\}\subseteq P_0$. When recursively constructing a decreasing sequence of conditions $\langle p_n:n\in\omega\rangle$, we associate a bookkeeping device $\langle\varphi_n:n\in\omega\rangle$ in the following way:
\begin{enumerate}
    \item For all $n\in\omega$, $\varphi_n:\bigcup_{k\leq n}P_k\to dom(p_n) $ is a bijection.
    \item For all $n\in\omega$, $\varphi_{n+1}\setminus \varphi_n$ is a bijection between $dom(\varphi_{n+1})\setminus dom(\varphi_n)$ and $dom(p_{n+1})\setminus dom(p_n)$.
\end{enumerate}

Note that we assume that $dom(p_{n+1})\setminus dom(p_n)$ is infinite for all $n\in\omega$.

\begin{thm}\label{P_k_has_sacks}
Let $\mathcal{F}$ be a suitable filter. Then $\mathbb{P}_\kappa(\mathcal{F})$ is proper and has the Sacks property.
\end{thm}

\begin{proof}
We prove both properness and Sacks property at the same time. Let $\mathcal{M}\prec H(\theta)$ be a countable elementary submodel, where $\theta$ is big enough, and such that $\mathcal{F},\kappa,\mathbb{P}_\kappa(\mathcal{F})\in\mathcal{M}$. Let $\dot{f}\in\mathcal{M}$ and a simple condition $p\in\mathbb{P}_\kappa(\mathcal{F})\cap \mathcal{M}$ be such that $p\Vdash\dot{f}\in\omega^\omega$. Let $\langle\dot{\alpha}_n:n\in\omega\rangle$ be an enumeration of all the $\mathbb{P}_\kappa(\mathcal{F})$-names in $\mathcal{M}$ for ordinals. For each $n\in\omega$, let $\mathcal{A}_n\in\mathcal{M}$ be a maximal antichain
which decides the value of $\dot{f}(n)$, and $\mathcal{B}_n\in\mathcal{M}$ a maximal antichain which decides the value of $\dot{\alpha}_n$.

We give an strategy to Player I in the game $\mathcal{G}_{\mathcal{M}}^\mathcal{P}(\mathcal{F})$, on the side we construct sequences $\langle p_n, \vartheta_n,\vec{S}_n:n\in\omega\rangle$ and a bookkeeping $\langle\varphi_n:n\in\omega\rangle$ such that:
\begin{enumerate}
    
    \item $p_0\leq p$.
    \item For all $n\in\omega$, $p_n\in\mathcal{M}$ is a simple condition.
    \item For each $n\in\omega$ and for all $\vec{\nu}\in \Pi_{k\leq n}2^{n-k+1}$, the condition $p_{n}[\vec{S}_n,\vec{\nu}]$ extends one condition in $\mathcal{A}_n$ and one condition in $\mathcal{B}_n$, both inside $\mathcal{M}$.
    \item For all $n\in\omega$, $p_{n+1}\leq^{\vec{S}_n}_n p_n$.
    \item $\bigcup_{n\in\omega}dom(p_n)=\bigcup_{n\in\omega}rng(\vec{S}_n)$.
    \item $\varphi_0:P_0\to dom(p_0)$ is a bijection such that $\varphi_0(0)=0$.
    \item $\vec{S}_0=\langle 0\rangle$ and $\vec{S}_1=\langle0,\varphi_0(1)\rangle$.
    \item For all $n\in\omega$, $dom(p_{n+1})\setminus dom(p_n)$ is infinite.
    \item For all $n\in\omega$, once $p_{n+1}$ is defined, $\varphi_{n+1}:\bigcup_{k\leq n+1}P_k\to dom(p_{n+1})$ is a bijection extending $\varphi_n$.
    \item For all $n\in\omega$, $\beta_n=\varphi_n(n)$ and $\vec{S}_n=\langle\beta_0,\ldots,\beta_n\rangle$.
    \item For $n>0$, once $\varphi_{n}$ is defined, define $\vec{S}_{n+1}={\vec{S}_{n}}^\frown\varphi_{n}(n+1)$.
    \item $\langle\vartheta_n:n\in\omega\rangle$ is an increasing sequence of natural numbers.
    \item For all $n\in\omega$, $\vec{E}_{p_n(0)}*\vartheta_n\sqsubseteq\vec{E}_{p_n(\beta_{n+1})}$.
    \item At round $n\in\omega$ of the game $\mathcal{G}_{\mathcal{M}}(\mathcal{F})$, Player I plays $\vec{E}_n^0=\vec{E}_{p_{n-1}(0)}*\vartheta_{n-1}$(define $p_{-1}=p$ and $\vartheta_{-1}=0$).
    \item If at round $n$ of the game $\mathcal{G}_{\mathcal{M}}(\mathcal{F})$ Players I and II have played $\vec{E}_n^0$ and $\vec{E}_n^1$, respectively, then $\vec{E}^{p_n(0)}_n\supseteq dom(\vec{E}_n^0)\setminus dom(\vec{E}_n^1)$.
\end{enumerate}

Clearly, if we succeed to construct such sequence having a lower bound $p_\omega$, clause (3) above implies $p_\omega$ is an $(\mathcal{M},\mathbb{P}_\kappa(\mathcal{F}))$-generic condition and also forces that $\dot{f}(n)$ is contained in a finite set of cardinality $h(n)=\Pi_{k\leq n}2^{n-k+1}$. Since $\mathcal{M}$ was arbitrary and we can capture any name for a real inside a countable elementary submodel, the Sacks property follows.

We leave it to the reader to check that all the objects constructed below, whenever it is required, can be found inside $\mathcal{M}$ (since everything Players I and II play should be in $\mathcal{M}$). We start with the definition of $p_0$. We can assume $p$ is a simple condition. Then we make Player I to play $\vec{E}_0^0=\vec{E}_{p(0)}*1$. Then Player II answer with $\vec{E}_0^1\sqsubseteq\vec{E}_0^0$. Now, define an $\mathcal{F}$-partition $\vec{F}_0$ as follows:
\begin{enumerate}
    \item $F^0_0=E_{0}^{p}\cup (dom(\vec{E}_0^0)\setminus dom(\vec{E}_0^1))$.
    \item $\vec{F}_0=\{F_0^0\}\cup\vec{E}^1_0$.
\end{enumerate}
Let $\langle\alpha_k^p:k\in\omega\rangle$ be the canonical enumeration of $dom(p)$. Note that for all $k\in\omega$, 
\begin{equation*}
    \vec{F}_0*k\sqsubseteq\vec{E}_{p(\alpha_k^p)}=\vec{E}_{p(0)}*k
\end{equation*}
So we can define a condition $r_0$ as follows:
\begin{enumerate}
    \item $\vec{E}_{r_0(0)}=\vec{F}_0$.
    \item $dom(r_0)=dom(p)$.
    \item For all $k\in\omega$, $r_0(\alpha_k^p)=p(\alpha_k^p)^{(*,\vec{F}_0*k)}$.
\end{enumerate}
Thus, we have $E_0^{r_0}=F_0^0=E_{0}^{p}\cup (dom(\vec{E}_0^0)\setminus dom(\vec{E}_0^1))$. Now, by two consecutive applications of Lemma \ref{n_S_simple_extension} with $\vec{S}_0=\langle 0\rangle$ and $n=0$, we find a simple condition $p_0\leq_0^{\vec{S}_0} r_0$ such that for all $s\in 2^1$, the condition $p_0[\vec{S}_0,s]$ extends some condition in $\mathcal{A}_0$ and some condition in $\mathcal{B}_0$. Finally define the corresponding function $\varphi_0$ and make $\vec{S}_1={\vec{S}_0}^\frown\varphi_0(1)=\langle\beta_0,\beta_1\rangle$. Now, let $\vartheta_0\in\omega$ be big enough so $\vec{E}_{p_0(0)}*\vartheta_0\sqsubseteq \vec{E}_{p_0(\beta_1)}$. Note that $E_{\vartheta_0}^{p_0}\in\vec{E}_{p_0(\beta_1)}$.

Suppose $p_n$, $\vec{S}_{n}$, $\vec{S}_{n+1}$, $\vartheta_n$ have been defined, so $\vec{S}_{n+1}=\langle\beta_0,\ldots,\beta_{n+1}\rangle$, and $\vec{E}_{p_n(0)}*\vartheta_n\sqsubseteq\vec{E}_{p_n(\beta_{n+1})}$. Then we let Player I to play $\vec{E}_{n+1}^0=\vec{E}_{p_n(0)}*(\vartheta_n+1)$, and Player II answers with $\vec{E}_{n+1}^1\sqsubseteq\vec{E}_{n+1}^0$. Define an $\mathcal{F}$-partition as follows:
\begin{enumerate}
    \item $F_0^{n+1}=E_{\vartheta_n}^{p_n}\cup(dom(\vec{E}_{n+1}^0)\setminus dom(\vec{E}_{n+1}^1))$.
    \item $\vec{F}_{n+1}=\{F_0^{n+1}\}\cup\vec{E}_{n+1}^1$.
\end{enumerate}

Let $\langle\alpha_k^{n}:k\in\omega\rangle$ be the canonical enumeration of $dom(p_n)$. Note that we have $\vec{S}_{n}=\langle \alpha_k^n:k\leq n\rangle$. Let $b_{n+1}\in\omega$ be such that $\beta_{n+1}=\alpha_{b_{n+1}}^n$; note that $b_{n+1}\leq\vartheta_{n}$. Note that for all $l>b_{n+1}$, $\vec{F}_{n+1}*(l-b_{n+1})\sqsubseteq\vec{E}_{p_n(\alpha_l^n)}$. We also have that for each $l\in[n+1,b_{n+1}]$, $\vec{F}_{n+1}\sqsubseteq\vec{E}_{p_n(\alpha_{l}^n)}$. Let $\langle\gamma_n:n\in\omega\rangle$ be an enumeration of $dom(p_n)\setminus\{\beta_0,\ldots,\beta_n\}$ and such that:
\begin{enumerate}
    \item $\gamma_0=\beta_{n+1}$.
    \item $\{\alpha_l^n:l\in[n+1,b_{n+1})\}=\{\gamma_{l}:l\in[1,b_{n+1}-n+1)\}$.
    \item $\{\alpha_{l}^n:l>b_{n+1}\}=\{\gamma_l:l>b_{n+1}-n+1\}$.
\end{enumerate}

Note that for each $l\in\omega$, $\vec{F}_{n+1}*l\sqsubseteq \vec{E}_{p_n(\gamma_l)}$. Now define a condition $r_{n+1}$ as follows:
\begin{enumerate}
    \item $\vec{E}_{r_{n+1}(0)}=(\vec{E}_{p_n}\upharpoonright (n+1))\cup \vec{F}_{n+1}$.
    \item $dom(r_{n+1})=dom(p_n)$.
    
    \item For all $j\leq n$, $r_{n+1}(\beta_j)=p_n(\beta_j)^{(n-j,\vec{F}_{n+1})}$.
    
    \item For $l\in\omega$, $r_{n+1}(\gamma_l)=p(\gamma_l)^{(*,\vec{F}_{n+1}*l)}$.
\end{enumerate}

Note that we have $r_{n+1}(\beta_{n+1})=r_{n+1}(\gamma_0)=p_n(\beta_{n+1})^{(*,\vec{F}_{n+1})}$. Moreover, $r_{n+1}$ is a simple condition such that $r_{n+1}\leq_n^{\vec{S}_n}p_n$. Then we have that:
\begin{enumerate}
    \item For each $j\leq n$, $E^{r_{n+1}}_j=E_j^{p_n}$.
    \item $E^{r_{n+1}}_{n+1}=F_{0}^{n+1}=\vec{E}_{\vartheta_n}^{p_n}\cup(dom(\vec{E}_{n+1}^0)\setminus dom(\vec{E}_{n+1}^1))$.
\end{enumerate}
Now, by two consecutive applications of Lemma \ref{n_S_simple_extension}, we can find a simple 
 condition $p_{n+1}\leq_{n+1}^{S_{n+1}}r_{n+1}$ such that for all $\vec{\nu}\in\Pi_{j\leq n+1}2^{n+1-j+1}$, the condition $p_{n+1}[\vec{S}_{n+1},\vec{\nu}]$ extends some condition in $\mathcal{A}_{n+1}$ and also extends some condition in $\mathcal{B}_{n+1}$. Note that we have $p_{n+1}\leq_n^{\vec{S}_n}p_n$. Now make the corresponding definition of $\varphi_{n+2}$. Also define ${\vec{S}_{n+2}=\vec{S}_{n+1}}^\frown \beta_{n+2}$, where $\beta_{n+2}=\varphi_{n+1}(n+2)$. Now, let $\vartheta_{n+1}\in\omega$ be big enough so $\vec{E}_{p_{n+1}(0)}*\vartheta_{n+1}\sqsubseteq\vec{E}_{p_{n+1}(\beta_{n+2})}$. This finishes the construction.

Since the previous strategy given to Player I is not a winning strategy, there is a run of the game in which Player II wins. Let $\langle (p_n,\vec{S}_n):n\in\omega\rangle$ be the sequence constructed along such game. Then we have,
\begin{enumerate}
    \item $\bigcup_{n\in\omega}dom(p_n)=\bigcup_{n\in\omega}rng(\vec{S}_n)$.
    \item $\langle\beta_n:n\in\omega\rangle=\bigcup_{n\in\omega}\vec{S}_n$
\end{enumerate}
Then, define a condition $p_\omega$ as follows:
\begin{enumerate}
    \item $dom(p_\omega)=\{\beta_n:n\in\omega\}$.
    \item $\vec{E}_{p_\omega(0)}=\langle E_n^{p_n}:n\in\omega\rangle=\langle F_0^n:n\in\omega\rangle$.
    \item For all $n\in\omega$, $h_{p_{\omega}(\beta_n)}=\bigcup_{j\in\omega}h_{p_j(\beta_n)}$.
    \item For all $n\in\omega$, $\vec{E}_{p_\omega(\beta_n)}=\vec{E}_{p_\omega(0)}*n$.
\end{enumerate}

It is not hard to see that $p_\omega$ is a condition, and for all $n\in\omega$, $p_\omega\leq_n^{\vec{S}_n}p_n$. Thus, $p_\omega$ is the required condition.
\end{proof}

The next lemma is easy to prove.

\begin{lemma}
Let $p\in\mathbb{P}_\kappa(\mathcal{F})$ be a condition. Then, for each $n\in\omega$ and $\alpha\in\kappa$, $p\Vdash\dot{\varphi}(n)(\alpha)=1$ if and only if $h_{p(\alpha)}(n)=1$.    
\end{lemma}

\begin{lemma}\label{nwd_codes_ground_model_avoided}
Let $\mathcal{F}$ be a suitable filter. Then $\mathbb{P}_\kappa(\mathcal{F})$ forces that for all $\vec{\nu}\in\mathcal{C}_\kappa\cap V$, $\dot{\varphi}^{-1}[\dot{X}_{\vec\nu}]\in\mathcal{F}^*$.
\end{lemma}

\begin{proof}
Let $p\in\mathbb{P}_\kappa(\mathcal{F})$ be a condition and $\vec{\nu}=\langle\nu_n:n\in\omega\rangle\in\mathcal{C}_\kappa$. We can assume that $p$ is a simple condition and that for all $n\in\omega$, $dom(\nu_n)\subseteq dom(p)$. Let $\langle\alpha_n:n\in\omega\rangle$ be the canonical enumeration of $dom(p)$. First note that  for all $k\in\omega$, there are $m\ge k$ and $l\in\omega$ such that $dom(\nu_l)\subseteq\{\alpha_j:j\in[k,m)\}$. We give a strategy to Player I in the game $\mathcal{G}^\mathcal{P}(\mathcal{F})$, and on the side we are going to construct sequences $\langle p_n:n\in\omega\rangle,\langle m_n:n\in\omega\rangle,\langle k_n:n\in\omega\rangle,\langle \vec{E}_n^0,\vec{E}_n^1, E_n^{-1}:n\in\omega\rangle,\langle b_n:n\in\omega\rangle$.

We start by defining $m_0=k_0=0$. Then let $m_1\in\omega$ be such that $dom(\nu_{k_0})\subseteq\{\alpha_j:j\in [m_0,m_1)\}$, and let $b_0\in\omega$ be big enough so for all $j\in[m_0,m_1)$, 
\begin{equation}
\vec{E}_{p(0)}*b_0\sqsubseteq\vec{E}_{p(\alpha_j)}    
\end{equation}
Then let Player I play $\vec{E}_0^0=\vec{E}_{p(0)}*(b_0+1)$. Then Player II answers with $\vec{E}_{0}^1\sqsubseteq\vec{E}_{0}^0$. Define $E_0^{-1}$ as:
\begin{equation*}
E_0^{-1}=E^{p(0)}_{b_0}\cup(dom(\vec{E}_0^0)\setminus dom(\vec{E}_0^1))    
\end{equation*}
Suppose $m_n,m_{n+1},k_n,\vec{E}_n^0,\vec{E}_n^1$ and $E_n^{-1}$ have been defined. Then let $k_{n+1},m_{n+2}\in\omega$ be big enough so $dom(\nu_{k_{n+1}})\subseteq\{\alpha_j:j\in[m_{n+1},m_{n+2})\}$. Note that since we assume $p$ to be a simple condition, there is $b_{n+1}$ such that for all $j\in [m_{n+1},m_{n+2})$,
\begin{equation*}
    \vec{E}_{n}^1*b_{n+1}\subseteq \vec{E}_{p(\alpha_j)}
\end{equation*}
Then we let Player I to play $\vec{E}_{n+1}^0=\vec{E}_{n+1}^{0}*(b_{n+1}+1)$. Then Player II answers with $\vec{E}_{n+1}^1\sqsubseteq\vec{E}_{n+1}^0$. Define $E_{n+1}^{-1}$ as follows:
\begin{equation*}
    E_{n+1}^{-1}=E_{b_{n+1}}^1\cup(dom(\vec{E}_{n+1}^0)\setminus dom(\vec{E}_{n+1}^1))
\end{equation*}

Let $\{E_n^{-1}:n\in\omega\}$ be the sequence constructed along a run of the game in which Player II wins the game. Thus, $\vec{E}_\omega=\{E_{n}^{-1}:n\in\omega\}$ is an $\mathcal{F}$-partition having the following property:
\begin{enumerate}
    \item [$\divideontimes$] For all $j\in\omega$, if $j\in [m_n,m_{n+1})$, then $\vec{E}_\omega* n\sqsubseteq\vec{E}_{p(\alpha_j)}$.
\end{enumerate}
So we can define a condition $p_\omega\leq p$ as follows:
\begin{enumerate}
    \item $dom(p_\omega)=dom(p)$.
    \item For all $j\in\omega$, is $j\in [m_n,m_{n+1})$, $p_\omega(\alpha_j)=p(\alpha_j)^{(*,\vec{E}_\omega*n)}$.
\end{enumerate}
In particular we have that $p_\omega(0)=\vec{E}_\omega$. Note that $p_\omega$ has the following properties:
\begin{enumerate}
    \item For all $j\in\omega$, if $j\in [m_n,m_{n+1})$, then $\vec{E}_{0}^{p_\omega(\alpha_j)}=E_n^{p_\omega(0)}=E_n^{-1}$.
    \item For all $n\in\omega$, $dom(\nu_{k_n})\subseteq \{\alpha_j:j\in [m_n,m_{n+1})\}$.
\end{enumerate}
Thus, we can define another condition $q_\omega\leq p_\omega$ as follows:
\begin{enumerate}
    \item $dom(q_\omega)=dom(p_\omega)$.
    \item For each $\alpha\in dom(q_\omega)\setminus \bigcup_{n\in\omega}dom(\nu_{k_n})$, $q_\omega(\alpha)=p_\omega(\alpha)$.
    \item For each $n\in\omega$, if $\alpha\in dom(\nu_{k_n})$, then:
        \begin{enumerate}
            \item $\vec{E}_{q_\omega(\alpha)}=\vec{E}_{\omega}*(n+1)$.
            \item $h_{q_\omega(\alpha)}\upharpoonright E_n^{-1}\equiv\nu_{k_n}(\alpha)$.
        \end{enumerate}
\end{enumerate}
It is easy to see that for each $l\in E_n^{-1}$, $q_\omega\Vdash \nu_{k_n}\subseteq\dot{\varphi}(l)$. Thus, $q_\omega$ forces $\dot{\varphi}[dom(\vec{E}_\omega)]\cap\dot{X}_{\vec{\nu}}=\emptyset$. Since $dom(\vec{E}_\omega)\in\mathcal{F}$, this implies that $q_\omega$ forces $\dot{\varphi}^{-1}[\dot{X}_{\vec{\nu}}]\in\mathcal{F}^*$.
\end{proof}

Let us recall the following definition from \cite{free_sequences}.

\begin{dfn}
Let $\mathbb{P}$ be a forcing. We say that $\mathbb{P}$ is Cohen preserving, if for any $p\in\mathbb{P}$ and any $\mathbb{P}$-name $\dot{\mathbb{P}}$ such that $p\Vdash\dot{D}\subseteq 2^{<\omega}$ is open dense, there are $q\leq p$ and an open dense set $E\subseteq 2^{<\omega}$ such that $q\Vdash E\subseteq\dot{D}$.
\end{dfn}

Clearly, being Cohen preserving is equivalent to preserving $\{X_{\vec{\nu}}:\vec{\nu}\in\mathcal{C}_{\omega}\}$ (of course, the family of evaluations of the sets of the for $\dot{X}_{\vec{\nu}}$) as a base for the ideal $\mathsf{nwd}$. It turns out that under the additional assumption of properness, Cohen preserving forcings also preserve $\{X_{\vec{\nu}}:\vec{\nu}\in\mathcal{C}_{\kappa}\}$ as a ba se for $\mathsf{nwd}_\divideontimes(\kappa)$.

\begin{lemma}\label{cohen_preserving_kappa}
Let $\mathbb{P}$ be a Cohen preserving proper forcing. Then $\mathbb{P}$, for any $p\in\mathbb{P}$ and $\dot{Y}$ such that $p\Vdash\dot{Y}\in\mathsf{nwd}_\divideontimes(\kappa)$, there are $q\leq p$ and $\vec{\nu}\in\mathcal{C}_\kappa\cap V$ such that $q\Vdash\dot{Y}\subseteq\dot{X}_{\vec{\nu}}$.
\end{lemma}

\begin{proof}
Let $\vec{\underaccent{\tilde}{\nu}}$ be a $\mathbb{P}$-name and $p\in\mathbb{P}$ such that $p\Vdash\vec{\underaccent{\tilde}{\nu}}\in\dot{\mathcal{C}}_\kappa$. Let $\dot{x}$ be a name for $\bigcup_{n\in\omega}dom(\dot{\nu}_n)$. Note that $\dot{x}$ is forced to be an infinite subset of $\kappa$ having order type $\omega$. Extend $p$ to a condition $q$ to find two ordinals $\alpha_0,\alpha_1\in\kappa$ such that $q\Vdash\min(\dot{x})=\alpha_0$ and $q\Vdash\sup(\dot{x})=\check{\alpha}_1$. Note that $\mathsf{cof}(\alpha)=\omega$. Let $\langle\beta_n:n\in\omega\rangle$ be a strictly increasing sequence of ordinals cofinal in $\alpha$ and such that $\beta_0=\alpha_0$. Since $q\Vdash ot(\dot{x})=\omega$, we have that $q\Vdash(\forall n\in\omega)(\vert\dot{x}\cap[\beta_n,\beta_{n+1}))\vert<\omega$. Let $\dot{f}:\omega\to[\kappa]^{<\omega}$ defined by $\dot{f}(n)=\dot{x}\cap[\beta_n,\beta_{n+1})$. By properness, we can extend $q$ to a condition $q_0$, and find $A\in[\kappa]^{\omega}$ such that $q_0\Vdash(\forall n\in\omega)(\dot{f}(n)\subseteq \check{A})$. Equivalently, $q_0\Vdash(\forall n\in\omega)(\dot{f}(n)\in[A]^{<\omega})$. Now, since $\mathbb{P}$ is Cohen preserving, $\mathbb{P}$ is $\omega^\omega$-bounding as well, which implies that we can find a condition $q_1$, for each $n\in\omega$, a finite set $F_n\subseteq[A]^{<\omega}$ such that $q_1\Vdash\dot{f}(n)\in F_n$. Now, for each $n\in\omega$, define $G_n=[\beta_n,\beta_{n+1})\cap\bigcup F_n$, and let $G=\bigcup_{n\in\omega}G_n$. Note that $ot(G)=\omega$ and $G$ is cofinal in $\alpha_1$. More over, $q_1\Vdash(\forall n\in\omega)(\dot{f}(n)\subseteq G)$. Thus, $q_1\Vdash(\forall n\in\omega)(dom(\nu_n)\subseteq G)$. Let $\phi:\omega\to G$ be the increasing enumeration of $G$. Then $\tilde{\nu}$ induces a $\mathbb{P}$-name $\vec{\underaccent{\tilde}{\mu}}$ in a natural way such that $q_1\Vdash\vec{\underaccent{\tilde}{\mu}}\in\dot{\mathcal{C}}_\omega$. Now, since $\mathbb{P}$ is Cohen preserving, we can extend $q_1$ to a condition $q_2$ and find $\vec{\mu}_0\in\mathcal{C}_\omega\cap V$ such that $q_2\Vdash\dot{X}_{\vec{\underaccent{\tilde}{\mu}}}\subseteq\dot{X}_{\vec{{\mu}}_0}$. Now we can translate $\vec{\mu}_0$ to an element $\vec{\nu}_0\in\mathcal{C}_\kappa$ such that $q_2\Vdash\dot{X}_{\vec{\underaccent{\tilde}{\nu}}}\subseteq\dot{X}_{\vec{\underaccent{\tilde}{\nu}}_0}$.
\end{proof}

\begin{lemma}[see \cite{FS}]
If $\mathbb{P}$ has the Sacks property, then $\mathbb{P}$ is Cohen preserving.
\end{lemma}

\begin{crl}
If $\mathcal{F}$ is suitable, $\mathbb{P}_\kappa(\mathcal{F})$ is Cohen preserving.
\end{crl}

For the next lemma recall that $B^\alpha_i$ is the pre-basic clopen set $\{x\in2^\kappa:x(\alpha)=i\}$.

\begin{lemma}\label{lemma_323}
Let $\mathcal{U}$ be a suitable filter. Let $\dot{\mathbb{Q}}$ be a $\mathbb{P}_\kappa(\mathcal{U})$-name for a Cohen preserving proper forcing. Let $\dot{Y}$ be a $\mathbb{P}_\kappa(\mathcal{U})*\dot{\mathbb{Q}}$-name and $(p,\dot{q})\in\mathbb{P}_\kappa(\mathcal{U})*\dot{\mathbb{Q}}$ such that $(p,\dot{q})$ forces $\dot{Y}\in\mathsf{nwd}_\divideontimes(\kappa)$. Then there are $(p_0,\dot{q}_0)\leq (p,\dot{q})$ and $X\in\mathcal{U}$ such that $(p_0,\dot{q}_0)$ forces $X\cap\dot{\varphi}^{-1}[\dot{Y}]=\emptyset$.
\end{lemma}

\begin{proof}
Let $(p,\dot{q})\in\mathbb{P}_\kappa(\mathcal{U})$ and $\dot{Y}$ be as stated. By Lemma \ref{lemma_323}, there are $(p_0,\dot{q}_0)\leq (p,\dot{q})$ and $\vec{\nu}\in\mathcal{C}_\kappa$ such that:
\begin{equation*}
    (p_0,\dot{q}_0)\Vdash\dot{Y}\subseteq\dot{X}_{\vec{\nu}}
\end{equation*}
By Lemma \ref{nwd_codes_ground_model_avoided}, we find a condition $(p_1,\dot{q}_1)\leq(p_0,\dot{q}_0)$ and $A\in\mathcal{F}$ such that 
\begin{equation*}
    (p_1,\dot{q}_1)\Vdash\dot{\varphi}^{-1}[\dot{X}_{\vec{\nu}}]\cap A=\emptyset
\end{equation*}
It follows that 
\begin{equation*}
    (p_1,\dot{q}_1)\Vdash\dot{\varphi}^{-1}[\dot{Y}]\cap A=\emptyset
\end{equation*}
\end{proof}

\begin{crl}\label{two_steps_sacks_prop_Tukey_above}
Let $\mathcal{F}$ be a suitable filter. Let $\dot{\mathbb{Q}}$ be a $\mathbb{P}_\kappa(\mathcal{F})$-name for a Cohen preserving proper forcing. Then $\mathbb{P}_\kappa(\mathcal{F})*\dot{\mathbb{Q}}$ forces that any ultrafilter extending $\mathcal{F}$ is Tukey above $[\kappa]^{<\omega}$.    
\end{crl}

\begin{proof}
Let $G$ be a $\mathbb{P}_\kappa(\mathcal{F})*\dot{\mathbb{Q}}$-generic filter over $V$. In $V[G]$, let $\mathcal{U}$ be an ultrafilter on $\omega$ extending $\mathcal{F}$. The previous lemma implies that $\mathsf{nwd}_\divideontimes(\kappa)\leq_K\mathcal{U}^*$, as witnessed by $\dot{\varphi}$. Since $\tau_\kappa\subseteq\mathsf{nwd}_\divideontimes(\kappa)$, it follows that $\tau_\kappa\leq_K\mathcal{U}^*$, so $[\kappa]^{<\omega}\leq_T\mathcal{U}$ follows.
\end{proof}

Finally, recall that any forcing with the Sacks property is Cohen preserving.
\begin{lemma}[see \cite{FS}]
If $\mathbb{P}$ has the Sacks property, then $\mathbb{P}$ is Cohen preserving.
\end{lemma}

\begin{thm}
It is relatively consistent that for any ultrafilter $\mathcal{U}$, it holds that $[\omega_1]^{<\omega}\leq_T\mathcal{U}$. 
\end{thm}

\begin{proof}
Let $V$ be a model of $\mathsf{ZFC}+\mathsf{GCH}+\diamondsuit(S)$, where $S=\{\alpha\in\omega_2:cof(\alpha)=\omega_1\}$, and let $\langle A_\alpha:\alpha\in S\rangle$ be a $\diamondsuit(S)$-guessing sequence. Let $\mathbb{P}=\langle\mathbb{P}_\alpha,\dot{\mathbb{Q}}_\alpha:\alpha\in\omega_2\rangle$ be a countable support iteration defined as follows:
\begin{enumerate}
    \item $\mathbb{P}_0=\mathbb{P}_{\omega_1}(\mathcal{U})$, where $\mathcal{U}$ is an ultrafilter on $\omega$ from $V$.
    \item If $\alpha\notin S$, then $\mathbb{P}_\alpha\Vdash\dot{\mathbb{Q}}_\alpha=\{1\}$.
    \item If $\alpha\in S$, and $A_\alpha$ codifies a $\mathbb{P}_\alpha$-name for an ultrafilter on $\omega$, let $\dot{\mathcal{U}}_\alpha$ be a $\mathbb{P}_\alpha$-name for such ultrafilter, then make $\mathbb{P}_\alpha\Vdash\dot{\mathbb{Q}}_\alpha=\dot{\mathbb{P}}_{\omega_1}(\dot{\mathcal{U}}_\alpha)$.
    \item Otherwise, $\mathbb{P}_\alpha\Vdash\dot{\mathbb{Q}}_\alpha=\{1\}$.
\end{enumerate}
By \ref{pp_iterated} and \ref{P_k_has_sacks} we have that $\mathbb{P}$ has the Sacks property, and for all $\alpha\in\omega_2$, $\mathbb{P}_\alpha$ forces $[\mathbb{P}:\mathbb{P}_\alpha]$ has the Sacks property.

Now, let $\dot{\mathcal{F}}$ be a $\mathbb{P}$-name for an ultrafilter, and let $C=\{\alpha\in\omega_2:\dot{\mathcal{F}}\cap V[G_\alpha]\in V[G_\alpha]\text{ is an ultrafilter in }V[G_\alpha]\}$. The set $C$ contains an $\omega_1$-club subset of $\omega_2$, so there is $\alpha\in S\cap C$ such that $A_\alpha$ codifies $\dot{\mathcal{F}}\cap V[G_\alpha]$, so $\mathbb{P}_\alpha\Vdash\dot{\mathbb{Q}}_\alpha=\dot{\mathbb{P}}_{\omega_1}(\dot{\mathcal{U}}_\alpha)=\dot{\mathbb{P}}_{\omega_1}(\dot{\mathcal{F}}\cap V[G_\alpha])$, and we have that $\mathbb{P}=\mathbb{P}_{\alpha}*\dot{\mathbb{P}}_{\omega_1}(\dot{\mathcal{U}}_\alpha)*[\mathbb{P}:\mathbb{P}_{\alpha+1}]$ forces that any ultrafilter extending $\dot{\mathcal{F}}\cap V[G_\alpha]$ is Tukey above $[\omega_1]^{<\omega}$, thus, we have that $\mathbb{P}\Vdash[\omega_1]^{<\omega}\leq_T\dot{\mathcal{F}}$.
\end{proof}

\section{Towards only one Tukey type of ultrafilters.}\label{one_tukey_type_start}

Throughout this section we assume that $V$ is a model of $\mathsf{ZFC}+\mathsf{GCH}+\diamondsuit(S)$, where $S=\{\alpha\in\omega_2:cof(\alpha)=\omega_1\}$. We define the forcing for the main result and prove it has the $\omega_2$-c.c. After the next definition we point out two important remarks about the definition.

\begin{dfn}\label{Q_k_definition}
Let $\kappa>\omega_1$ be a regular cardinal. Define $S=\{\alpha\in\omega_2:cof(\alpha)=\omega_1\}$ and let $\langle A_\alpha:\alpha\in S\rangle$ be a $\diamondsuit(S)$-guessing sequence. Let $\mathbb{S}_{\omega_2}$ be the countable support iteration of the Sacks forcing. Let $D_0$ be the set of ordinals $\alpha\in S$ such that $A_\alpha$ codifies a $\mathbb{S}_\alpha$-name for a suitable filter. For each $\alpha\in D_0$, let $\dot{\mathcal{U}}_\alpha$ the $\mathbb{S}_\alpha$-name $A_\alpha$ codifies. Let $\{\alpha_\beta:\beta\in\omega_2\}$ be the increasing enumeration of $D_0$. Define a forcing notion $\mathbb{Q}_\kappa$ as follows:
\begin{enumerate}
    \item $(p,q)\in\mathbb{Q}_\kappa$ if and only if:
    \begin{enumerate}
        \item $p\in\mathbb{S}_{\omega_2}$.
        \item $q$ is a partial function with domain a countable subset of $\omega_2$, and for each $\beta\in dom(q)$, $q(\beta)$ is a $\mathbb{S}_{\alpha_\beta}$-name.
        \item For all $\beta\in dom(q)$, $p\upharpoonright \alpha_\beta\Vdash q(\beta)\in\dot{\mathbb{P}}_\kappa(\dot{\mathcal{U}}_{\alpha_\beta})^{V[\dot{G}_{\alpha_\beta}]}$, where $G_{\alpha_\beta}$ is the $\mathbb{S}_{\alpha_\beta}$-generic filter over $V$.
    \end{enumerate}
    \item For $(p_1,q_1),(p_0,q_0)\in\mathbb{Q}_\kappa$, define $(p_1,q_1)\leq(p_0,q_0)$ if and only if:
    \begin{enumerate}
        \item $p_1\leq p_0$.
        \item For all $\gamma\in dom(q_0)$, $p_1\upharpoonright\alpha_\gamma\Vdash q_1(\gamma)\leq q_0(\gamma)$.
    \end{enumerate}
\end{enumerate}
Given a condition $(p,q)\in\mathbb{Q}_\kappa$, we also define:
\begin{enumerate}
    \item $supp(p)=\{\alpha\in \omega_2:\alpha\in dom(p)\}$.
    \item $supp(q)=\{\alpha\in\omega_2:\alpha\in dom(q)\}$.
    \item $supp(p,q)=(supp(p),supp(q))$.
\end{enumerate}
(The context should be enough to avoid confusions of which support we refer to).
\end{dfn}

\noindent{\textbf{Remark.}} We emphasize two points of major importance in Definition \ref{Q_k_definition}. First, clause (1)(b): for each $\beta\in dom(q)$, $q(\beta)$ is a $\mathbb{S}_{\alpha_\beta}$-name. Note that for $\gamma>\alpha_\beta$ there may be $\mathbb{S}_{\gamma}$-names which are forced to be elements from $\dot{\mathbb{P}}_\kappa(\dot{\mathcal{U}}_{\alpha_\beta})^{V[\dot{G}_{\alpha_\beta}]}$, yet, they may not be $\mathbb{S}_{\alpha_\beta}$-names. By no means $q(\beta)$ is allowed to be one of such $\mathbb{S}_\gamma$-names: for all $\beta\in dom(q)$, $q(\beta)$ is strictly a $\mathbb{S}_{\alpha_\beta}$-name. Note that this implies that for each $\beta>0$, $q(\beta)$ does not depend $q(\gamma)$ for any $\gamma<\beta$; $q(\beta)$ depends only on $p\upharpoonright\alpha_\beta$. Second, clause (1)(c), 
\begin{center}
For any $\beta\in dom(q)$, $p\upharpoonright\alpha_\beta$ forces $q(\beta)\in\dot{\mathbb{P}}_\kappa(\dot{\mathcal{U}}_{\alpha_\beta})^{V[\dot{G}_{\alpha_\beta}]}$
\end{center}
here, $\dot{\mathbb{P}}_\kappa(\dot{\mathcal{U}}_{\alpha_\beta})^{V[\dot{G}_{\alpha_\beta}]}$ lives in the $\mathbb{S}_{\alpha_\beta}$-extension. Note that for any $\gamma>\alpha_\beta$, $V[G_\gamma]$ has its own definition of $\dot{\mathbb{P}}_\kappa(\dot{\mathcal{U}}_{\alpha_\beta})$, say $\dot{\mathbb{P}}_\kappa(\dot{\mathcal{U}}_{\alpha_\beta})^{V[\dot{G}_{\gamma}]}$. Naturally, $\dot{\mathbb{P}}_\kappa(\dot{\mathcal{U}}_{\alpha_\beta})^{V[\dot{G}_{\alpha_\beta}]}\subseteq \dot{\mathbb{P}}_\kappa(\dot{\mathcal{U}}_{\alpha_\beta})^{V[\dot{G}_{\gamma}]}$, however, note that  $\dot{\mathbb{P}}_\kappa(\dot{\mathcal{U}}_{\alpha_\beta})^{V[\dot{G}_{\alpha_\beta}]}$ is not even a regular suborder of $\dot{\mathbb{P}}_\kappa(\dot{\mathcal{U}}_{\alpha_\beta})^{V[\dot{G}_{\gamma}]}$.

The next lemma is a place where the definition of $\mathbb{P}_\kappa(\mathcal{F})$ allowing conditions with redundant information simplifies things, instead of looking just at simple conditions.

\begin{lemma}
Let $(p,q)\in\mathbb{Q}_\kappa$ be a condition. Then there is $(p_0,q_0)\leq(p,q)$ and $a\in[\kappa]^\omega$ such that  for all $\beta\in supp(q)$, $p\upharpoonright\alpha_\beta\Vdash dom(q(\beta))=a$.
\end{lemma}

\begin{proof}

Let $(p,q)\in\mathbb{Q}_\kappa$ be a condition. Let $\mathcal{M}\prec H(\theta)$ be a countable elementary submodel, for some big enough $\theta$, such that $(p,q),\mathbb{S}_{\omega_2},\omega_2,\omega_1,\kappa\in\mathcal{M}$. Note that $supp(q)\subseteq \mathcal{M}$, and for each $\gamma\in supp(q)$, there is, inside $\mathcal{M}$, a $\mathbb{S}_{\alpha_\gamma}$-name for $dom(q(\gamma))$. Let $p_0\leq q$ be a $(\mathbb{S}_{\omega_2},\mathcal{M})$-generic condition. Note that for each $\gamma\in supp(q)$, $p_0\upharpoonright\alpha_\gamma$ is $(\mathbb{S}_{\alpha_\gamma},\mathcal{M})$-generic, which implies that for any $\gamma\in supp(q)$, $$p_0\upharpoonright\alpha_\gamma\Vdash dom(q(\gamma))\subseteq\mathcal{M}\cap\mathsf{Ord}$$ Define $a=\mathcal{M}\cap \mathsf{Ord}$, and $q_0$  as:
\begin{enumerate}
    \item $supp(q_0)=supp(q)$.
    \item For each $\gamma\in supp(q)$, $p_0\upharpoonright\alpha_{\gamma}$ forces:
    \begin{enumerate}
        \item $dom(q_0(\gamma))=a$. 
        \item $q_0(\gamma)\upharpoonright dom(q(\gamma))=q(\gamma)$.
        \item For each $\beta\in a\setminus dom(q(\gamma))$, $q_0(\gamma)(\beta)=q(\gamma)(0)$
    \end{enumerate}
\end{enumerate}
It follows from the construction that $(p_0,q_0)$ is a condition in $\mathbb{Q}_\kappa$, and for each $\gamma\in supp(q_0)$, $p_0\upharpoonright\alpha_\gamma\Vdash dom(q_0(\gamma))=a$.
\end{proof}

\begin{lemma}
$\mathbb{Q}_{\kappa}$ has the $\omega_2$-c.c.
\end{lemma}

\begin{proof}
Let $\{(p_\alpha,q_\alpha):\alpha\in\omega_2\}$ be any subset of $\mathbb{Q}_\kappa$ of cardinality $\omega_2$. By the previous lemma we can assume that for each $\alpha\in\omega_2$, there is $A_\alpha$ such that for all $\gamma\in supp(q_\alpha)$, $p_\alpha\upharpoonright\alpha_\gamma\Vdash dom(q_\alpha(\gamma))=A_\alpha$. By three consecutive applications of the $\Delta$-system lemma, we find $G\in[\omega_2]^{\omega_2}$, and $R_0\in [\omega_2]^{\leq\omega},R_1\in[\omega_2]^{\leq\omega},R_2\in [\kappa]^{\leq\omega}$ such that:
\begin{enumerate}
    \item $\{supp(p_\alpha):\alpha\in G\}$ is a $\Delta$-system with root $R_0$.
    \item $\{supp(q_\alpha):\alpha\in G\}$ is a $\Delta$-system with root $R_1$.
    \item $\{A_\alpha:\alpha\in G\}$ is a $\Delta$-system with root $R_2$.
    \item For each $\alpha,\beta\in G$ such that $\alpha<\beta$, 
    \begin{enumerate}
        \item $sup(R_0)<\min(supp(p_\alpha)\setminus R_0)$.
        \item $sup(R_1)<\min(supp(q_\alpha)\setminus R_1)$.
        \item $sup(supp(p_\alpha)\setminus R_0)<\min(supp(p_\beta)\setminus R_0)$.
        \item $sup(supp(q_\alpha)\setminus R_1)<\min(supp(q_\beta)\setminus R_1)$.
    \end{enumerate}
    \item Any two conditions in $\{p_\alpha:\alpha\in G\}$ are compatible.
\end{enumerate}
Define $\lambda=\sup\{\alpha_\gamma:\gamma\in R_1\}$. Let us look for a moment at 
\begin{multline*}
\mathbb{Q}_\kappa\upharpoonright(\lambda,R_1,R_2)=\{(p,q)\in\mathbb{Q}_\kappa: supp(p)\subseteq\lambda\land\\ supp(q)\subseteq R_1\land\\ \forall (\gamma\in supp(q))(p\upharpoonright\alpha_\gamma\Vdash dom(q(\gamma)\subseteq R_2))\}    
\end{multline*}
and give it the order corresponding to the restriction of the order defined on $\mathbb{Q}_\kappa$. By standard considerations of iterated forcing applied to the present case (for example, making use of hereditarily countable names and counting arguments), it is possible to find a dense subset of $\mathbb{Q}_\kappa\upharpoonright(\lambda,R_1,R_2)$ of cardinality $\omega_1$. Let $D$ be one of such sets. 

Now, for each $\beta\in G$, let $(\rho_\beta,\tau_\beta)\in D$ be such that $(\rho_\beta,\tau_\beta)\leq (p_\beta\upharpoonright\lambda,q_\beta\upharpoonright(R_1,R_2))$,
where $q_\beta\upharpoonright(R_1,R_2)$ is defined as:
\begin{enumerate}
    \item $supp(q_\beta\upharpoonright(R_1,R_2))= R_1$.
    \item For all $\gamma\in R_1$, $p\upharpoonright\alpha_\gamma\Vdash q_\beta\upharpoonright(R_1,R_2)(\gamma)=q_\beta(\gamma)\upharpoonright R_2$.
\end{enumerate}

Since $D$ has cardinality $\omega_1$ and $\vert G\vert=\omega_2$, there are different $\beta_0,\beta_1\in G$ such that $(\rho_{\beta_0},\tau_{\beta_0})=(\rho_{\beta_1},\tau_{\beta_1})$. Then we should have that $(p_{\beta_0},q_{\beta_0})$ and $(p_{\beta_1},q_{\beta_1})$ are compatible.
\end{proof}

\section{Cohen proper forcing.}\label{cohen_properness_notion}
In this section and the next one we develop a general framework from which the forcing $\mathbb{Q}_\kappa$ is a special case. Here we define the notion of Cohen proper forcing and derive some equivalent formulations enclosed in Theorem \ref{cohen_proper_equivalences}. The reader that wish to jump to the iteration and preservation theorem can take definitions \ref{restricted_iteration_definition}, \ref{divideontimes_symbol} and \ref{cohen_proper_definition} below, and clause (2) from Theorem \ref{cohen_proper_equivalences} as a black-box.

The notion stated in the next definition was named strongly proper by Shelah. However, there is also a strongly proper forcing notion due to Mitchell. It seems that Mitchell's notion of strongly proper forcing is more well known, so, to avoid any possible confusion and fully distinguish the two notions, we have decided to rename Shelah's notion as ``$\sigma$-proper''. More explicitly:

\begin{dfn}\label{strongly_proper_def}
Let $\mathbb{P}$ be a forcing. We say that $\mathbb{P}$ is $\sigma$-proper if for any $\mathcal{M}\prec H(\theta)$ (for big enough $\theta$) such that $\mathbb{P}\in\mathcal{M}$, any $p\in\mathbb{P}\cap\mathcal{M}$ and any family $\vec{\mathcal{D}}=\{D_n:n\in\omega\}$ of open dense subsets of $\mathbb{P}\cap\mathcal{M}$, there is a condition $q\leq p$ such that each set $D_n$ is predense below $q$. We say that $q$ is a $(\mathbb{P},\mathcal{M},\vec{\mathcal{D}})$-generic condition.
\end{dfn}

It is well known that countable support iterations of $\sigma$-proper forcings are $\sigma$-proper (strongly proper, under Shelah's naming) (see \cite{proper}).

The following lemma is easy to prove, so we omit the argument.

\begin{lemma}\label{omega_cohen_preserving}
If $\mathbb{P}$ is Cohen preserving, then for any countable family $\{\dot{D}_n:n\in\omega\}$ of $\mathbb{P}$-names for open dense subsets of $\mathbb{C}$ and $p\in\mathbb{P}$, there are $\{E_n:n\in\omega\}$ family of open dense subsets of $\mathbb{C}$ and $q\leq p$ such that for each $n\in\omega$, $q\Vdash E_n\subseteq\dot{D}_n$.
\end{lemma}

Next we fix some notation to be used along the rest of the paper.
\begin{enumerate}
    \item If $f$ is a function and $A$ is a set, then $f*A=\{(a,b)\in f:a\notin A\}$(when taking about iterations $*$ keeps its usual meaning, the context should be enough to differentiate between the two meanings).
    \item $\mathbb{P}\lessdot\mathbb{Q}$ means that $\mathbb{P}$ is a regular suborder of $\mathbb{Q}$.
    \item If $\mathbb{P}\lessdot\mathbb{Q}$, $[\mathbb{Q}:\mathbb{P}]$ denotes the forcing quotient of $\mathbb{Q}$ given a $\mathbb{P}$-generic filter:
    \begin{equation*}
        \mathbb{P}\Vdash[\mathbb{Q}:\mathbb{P}]=\{q\in\mathbb{Q}:(\forall p\in G)(p\parallel q)\}
    \end{equation*}
    where $p\parallel q$ means $p$ and $q$ are compatible.
    \item Given a $\mathbb{P}$-name $\dot{x}$ and a $\mathbb{P}$-generic filter $G$, the evaluation of $\dot{x}$ by $G$ is denoted by $\dot{x}_G$.
\end{enumerate}

\begin{dfn}\label{restricted_iteration_definition}
Let $\mathbb{P,Q}$ be forcings such that $\mathbb{P}\lessdot\mathbb{Q}$, and let $\pi:\mathbb{Q}\to\mathbb{P}$ be a regular projection projection, $\dot{\mathbb{R}}$ a $\mathbb{P}$-name for a forcing and $\mathbb{P}*\dot{\mathbb{R}}$ the two step iteration of $\mathbb{P}$ and $\dot{\mathbb{R}}$. Define the $\mathbb{P}$-restricted iteration of $\mathbb{Q}$ and $\dot{\mathbb{R}}$ as: $(q,\dot{r})\in\mathbb{Q}*_{\mathbb{P}}\dot{\mathbb{R}}$ if and only if:
\begin{equation*}
    (\pi(q),\dot{r})\in\mathbb{P}*\dot{\mathbb{R}}.
\end{equation*}
and the order is given by $(q_1,\dot{r}_1)\leq(q_0,\dot{r}_0)$ if and only if:
\begin{enumerate}
    \item $q_1\leq q_0$.
    \item $\pi(q_1)\Vdash\dot{r}_1\leq\dot{r}_0$.
\end{enumerate}
\end{dfn}

\begin{dfn}\label{divideontimes_symbol}
Let $\mathbb{P}$ and $\mathbb{Q}$ be forcings such that $\mathbb{P}\lessdot\mathbb{Q}$ and let $\pi:\mathbb{Q}\to\mathbb{P}$ be a projection. Let $\dot{\mathbb{R}}$ be a $\mathbb{Q}$-name for a forcing and consider $\mathbb{Q}*_{\mathbb{P}}\dot{\mathbb{R}}$. For $X\subseteq\mathbb{Q}*_{\mathbb{P}}\dot{\mathbb{R}}$, define the set $$X^\divideontimes=\{(\dot{r},q):(q,\dot{r})\in X\}$$

Note that $X^\divideontimes$ is a $\mathbb{Q}$-name for a subset of $\dot{\mathbb{R}}$.
\end{dfn}

\begin{dfn}\label{cohen_proper_definition}\quad
Let $\mathbb{P}$ and $\mathbb{Q}$ be complete Boolean algebras such that $\mathbb{P}\lessdot\mathbb{Q}$, and let $\pi$ be the projection from $\mathbb{Q}$ to $\mathbb{P}$. We say that $(\mathbb{P},\mathbb{Q})$ is Cohen proper if for any countable $\mathcal{M}\prec H(\theta)$ such that $\mathbb{P},\mathbb{Q}\in\mathcal{M}$, any $\mathbb{P}$-name $\dot{\mathbb{R}}\in\mathcal{M}$ for a forcing, any $q\in\mathcal{M}\cap\mathbb{Q}$ and $\langle D_n:n\in\omega\rangle$ such that for each $n\in\omega$:
\begin{enumerate}
    \item $D_n\subseteq\mathbb{Q}*_{\mathbb{P}}\dot{\mathbb{R}}\cap\mathcal{M}$.
    \item $D_n$ is open dense in $\mathbb{Q}*_{\mathbb{P}}\dot{\mathbb{R}}\cap\mathcal{M}$.
\end{enumerate}
there are $q_0\leq q$ and $\langle \dot{E}_n:n\in\omega\rangle$ such that for each $n\in\omega$:
\begin{enumerate}
    \item $\dot{E}_n$ is a $\mathbb{P}$-name.
    \item $\pi(q_0)\Vdash \dot{E}_n\text{ is an open dense subset of $\dot{\mathbb{R}}\cap\mathcal{M}[\pi[\dot{G}]]$}$.
    \item $q_0\Vdash D_n^\divideontimes\subseteq\dot{\mathbb{R}}\cap\mathcal{M}[\dot{G}]\text{ is open dense}$.
    \item $q_0\Vdash \dot{E}_n\subseteq {D_n^\divideontimes}$.
\end{enumerate}
where $G$ is $\mathbb{Q}$-generic.

We say that $q_0$ is a $(\mathbb{P},\mathbb{Q},\mathcal{M})$-generic condition for $\{D_n:n\in\omega\}$ and $\dot{\mathbb{R}}$.
\end{dfn}

\begin{lemma}\label{strongly_proper_forcing_open_sets}
Let $\mathbb{Q}$ be a $\sigma$-proper forcing, $\mathbb{P}\lessdot\mathbb{Q}$ a regular suborder, and $\mathcal{M}\prec H(\theta)$ countable such that $\mathbb{P},\mathbb{Q}\in\mathcal{M}$. Let $\dot{\mathbb{R}}\in\mathcal{M}$ be a $\mathbb{P}$-name for a forcing, so $\mathbb{Q}*_{\mathbb{P}}\dot{\mathbb{R}}\in\mathcal{M}$, and $\{D_n:n\in\omega\}$ a family of open dense subsets of $\mathbb{Q}*_{\mathbb{P}}\dot{\mathbb{R}}\cap\mathcal{M}$. Fix $q\in\mathbb{Q}\cap\mathcal{M}$. Then there is $q_0\leq q$ which forces that for each $n\in\omega$, $D_n^\divideontimes$ is an open dense subset of $\dot{\mathbb{R}}\cap\mathcal{M}[\dot{G}]$.
\end{lemma}

\begin{proof}
To simplify notation, given a condition $a\in \mathbb{Q}$, we write $a^{\perp}$ to denote the set of conditions in $\mathbb{Q}$ which are incompatible with $a$. We need to define several open dense subsets of $\mathbb{Q}\cap\mathcal{M}$:
\begin{enumerate}
\item For each $(p,\dot{r})\in\mathbb{Q}*_{\mathbb{P}}\dot{\mathbb{R}}\cap\mathcal{M}$ and $n\in\omega$, define 
    \begin{equation*}
    \overline{D}(n,p,\dot{r})=\{a\in\mathbb{Q}\cap\mathcal{M}:(\exists (t,\dot{s})\in D_n)((t,\dot{s})\leq(p,\dot{r}) \land t=a)\}    
    \end{equation*}
    Then let $D(n,p,\dot{r})=\overline{D}(n,p,\dot{r})\cup(p^\perp\cap\mathcal{M})$
\item Now, for each $\mathbb{Q}$-name $\dot{r}\in\mathcal{M}$, define
    \begin{equation*}
D(\dot{r})=\{p\in\mathbb{Q}\cap\mathcal{M}:p\Vdash \dot{r}\in\dot{\mathbb{R}}\vee p\Vdash\dot{r}\notin\dot{\mathbb{R}}\}            \end{equation*}

\item For $\mathbb{P}$-names $\dot{r},\dot{s}\in\mathcal{M}$. For a condition $p\in\mathbb{Q}$, let us say that $p$ decides the relation $\dot{s}<\dot{r}$ it either, $\pi(p)\Vdash \dot{r}<\dot{s}$, or $\pi\Vdash\dot{r}>\dot{r}$ or $\pi(p)\Vdash\dot{r}=\dot{s}$. define $D_0(\dot{s},\dot{r})$ and $D_1(\dot{s},\dot{r})$ as:
    \begin{equation*}
    D_0(\dot{s},\dot{r})=
    \{p\in\mathbb{Q}\cap\mathcal{M}:(\pi(p)\Vdash\dot{r},\dot{s}\in\dot{\mathbb{R}})\land
    (p\text{ decides }\dot{r}<\dot{s}))\}
    \end{equation*}
    
    \begin{equation*}
    D_1(\dot{s},\dot{r})=\{p\in\mathbb{Q}\cap\mathcal{M}:(\pi(p)\Vdash\dot{r}\notin \dot{\mathbb{R}})\vee (\pi(p)\Vdash \dot{s}\notin\dot{\mathbb{R}})\}
    \end{equation*}
Then let $D(\dot{r},\dot{s})=D_0(\dot{r},\dot{s})\cup D_1(\dot{r},\dot{s})$.

\item Now, for a $\mathbb{Q}$-name $\dot{r}\in\mathcal{M}$ and $p\in\mathbb{Q}\cap\mathcal{M}$ such that $p\Vdash\dot{r}\in\dot{\mathbb{R}}$, define
    \begin{equation*}
    \overline{D}_1(p,\dot{r})=\{p_1\in\mathbb{Q}\cap\mathcal{M}:(\exists(p_1,\dot{r}_1)\in\mathbb{Q}*_{\mathbb{P}}\dot{\mathbb{R}}\cap\mathcal{M})(p_1\leq p)\land(p_1\Vdash\dot{r}=\dot{r}_1)\}
    \end{equation*}
    and let $D_1(p,\dot{r})=\overline{D}_1(p,\dot{r})\cup(p^\perp\cap\mathcal{ B})$.
\end{enumerate}

Let $\mathcal{E}$ be the collection of all sets of the types (1)-(4). Let $q_0\leq q$ be a $(\mathbb{Q},\mathcal{M},\mathcal{E})$-generic condition. We claim that $q_0$ forces each $D_n^\divideontimes$ to be an open dense subset of $\dot{\mathbb{R}}\cap\mathcal{M}[\dot{G}]$. Let $G$ be a generic filter such that $q_0\in G$ and let us work in $V[G]$.  Note that $\dot{\mathbb{R}}_G=(\mathbb{Q}*_{\mathbb{P}}\dot{\mathbb{R}})^\divideontimes_G$.

First let us see that $(D_n)^\divideontimes_G$ is dense for each $n\in\omega$. Pick $r\in\dot{\mathbb{R}}_G\cap\mathcal{M}[G]$. Then there is $\dot{r}\in\mathcal{M}$ such that $r=\dot{r}_{G}$. Since $D(\dot{r})\cap G\neq\emptyset$, there is $p_0\in\mathbb{Q}\cap\mathcal{M}\cap G$ such that $p_0\Vdash \dot{r}\in\dot{\mathbb{R}}$. Now, $D_1(p_0,\dot{r})\cap G\neq\emptyset$, and $p_0\in G$, so there is $(p_1,\dot{r}_1)\in\mathbb{Q}*_{\mathbb{P}}\dot{\mathbb{R}}\cap\mathcal{M}$ such that $p_1\Vdash\dot{r}=\dot{r}_1$ and $p_1\leq p_0$. Now, $D(n,(p_1,\dot{r}_1))\cap G\neq\emptyset$, so pick a condition $p_2\in D(n,(p_1,\dot{r}_1))\cap G$. Then we should have $p_2\leq p_1$ and there is $\dot{r}_2\in\mathcal{M}$ such that $(p_2,\dot{r}_2)\in D_n\cap\mathcal{M}$ and $(p_2,\dot{r}_2)\leq(p_1,\dot{r}_1)$, which implies $(\dot{r}_2)_G\leq (\dot{r}_1)=\dot{r}_G=r$. By definition of $D_n^\divideontimes$, we also have $(\dot{r}_2)_G\in (D_n^\divideontimes)_G$. Thus $(D_n)^\divideontimes_G$ is dense.

Let us now see that $(D_n)^\divideontimes$ is open. Fix $r\in(D_n)^\divideontimes_G$ and let $s\in\dot{\mathbb{R}}_G\cap\mathcal{M}[G]$ be such that $s\leq r$. We have to prove that $s\in (D_n)^\divideontimes_G$. By definition of $(D_n)^\divideontimes_G$, there is $(p_0,\dot{r})\in D_n$ such that $p_0\in G$ and $r=\dot{r}_G$. Let $\dot{s}\in\mathcal{M}$ be such that $s=\dot{s}_G$. Then $D(\dot{s})\cap G\cap\mathcal{M}\neq\emptyset$. Fix $p_1\in D(\dot{s})\cap G\cap\mathcal{M}$ such that $p_1\leq p_0$. Then we have $p_1\Vdash \dot{s}\in\dot{\mathbb{R}}$. Now, we have $D_1(p_1,\dot{s})\cap G\cap\mathcal{M}\neq\emptyset$, so we can pick $p_2\in D_1(p_1,\dot{s})\cap G\cap\mathcal{M}$. Note that we have $p_2\leq p_1$, and there is $\dot{s}_2$ such that $(p_2,\dot{s}_2)\in\mathbb{Q}*_{\mathbb{P}}\dot{\mathbb{R}}\cap\mathcal{M}$ and $p_2\Vdash\dot{s}=\dot{s}_2$. Now, since $s\leq r$, there is $p_3\in G\cap D(\dot{r},\dot{s}_2)\cap\mathcal{M}$ such that $p_3\leq p_2$ and $\pi(p_3)\Vdash\dot{s}_2\leq\dot{r}$. Note that $(p_3,\dot{s}_2)\in \mathbb{Q}*_{\mathbb{P}}\dot{\mathbb{R}}$. Moreover, we also have $(p_3,\dot{s}_2)\leq(p_0,\dot{r})$. Since $(p_0,\dot{r})\in D_n$ and $D_n$ is open, we have $(p_3,\dot{s}_2)\in D_n$, which implies $(\dot{s}_2)_G\in D_n^\divideontimes$; since $p_3\Vdash \dot{s}=\dot{s}_2$, we have $s\in D_n^\divideontimes$. Thus, $D_n^\divideontimes$ is open.
\end{proof}

\begin{prp}\label{varsigma_proper_stringly_proper_cohen_preserving}
If $\mathbb{Q}$ is $\sigma$-proper and $\mathbb{P}$ forces $[\mathbb{Q}:\mathbb{P}]$ is Cohen preserving, then $(\mathbb{P},\mathbb{Q})$ is Cohen proper.
\end{prp}
\begin{proof}
Let $\mathcal{M}\prec H(\theta)$ be countable and such that $(\mathbb{P},\mathbb{Q})\in\mathcal{M}$ and pick $q\in\mathbb{Q}\cap\mathcal{M}$. Fix $\dot{\mathbb{R}}\in\mathcal{M}$ a $\mathbb{P}$-name for a partial order. Let $\vec{D}$ be a countable family of open dense subsets of $\mathbb{Q}*_{\mathbb{P}}\dot{\mathbb{R}}\cap\mathcal{M}$. By the previous Lemma \ref{strongly_proper_forcing_open_sets}, we can find a $(\mathbb{Q},\mathcal{M})$-generic condition $q'\leq q$ forcing each set $D_n^\divideontimes$ to be an open dense subset of $\dot{\mathbb{R}}\cap\mathcal{M}[\dot{\Gamma}]$, fix $q_0$ one of such conditions. Let $\Gamma_0$ be a $\mathbb{P}$-generic filter such that $\pi(q_0)\in\Gamma_0$. Then, in $V[\Gamma_0]$, $q_0\in[\mathbb{Q}:\mathbb{P}]$ and forces each $D_n^\divideontimes$ to be an open dense subset of $\dot{\mathbb{R}}\cap\mathcal{M}[\Gamma_0]$. Since $[\mathbb{Q}:\mathbb{P}]$ Cohen preserving, by Lemma \ref{omega_cohen_preserving}, in $V[\Gamma_0]$ we can extend $q_0$ to a condition $q_1$ and find $\{E_n:n\in\omega\}$ open dense subsets of $\dot{\mathbb{R}}\cap\mathcal{M}[\Gamma_0]$ such that for each $n\in\omega$, $q_1\Vdash E_n\subseteq D_n^\divideontimes$. Going back to $V$ we get a $\mathbb{P}$-name $\dot{q}_1$ for $q_1$ and the corresponding family of $\mathbb{P}$-names $\{\dot{E}_n:n\in\omega\}$ for the open dense subsets of $\dot{\mathbb{R}}\cap\mathcal{M}[\Gamma_0]$. Then we have $(\pi(q_0),\dot{q}_1)$ forces, for each $n\in\omega$:
\begin{enumerate}
    \item $D_n^\divideontimes$ is an open dense subset of $\dot{\mathbb{R}}\cap\mathcal{M}[\dot{\Gamma}_0]$.
    \item $\dot{E}_n\subseteq D_n^\divideontimes$.
\end{enumerate}
and $\pi(q_0)$ forces for each $n\in\omega$,
\begin{enumerate}
    \item $\dot{E}_n$ lives in $V[\dot{\Gamma}_0]$ and is an open dense subset of $\dot{\mathbb{R}}\cap\mathcal{M}[\Gamma_0]$.
\end{enumerate}
Now we only have to find $q_2\in\mathbb{Q}$ such that $(\pi(q_2),q_2)\leq(\pi(q_0),\dot{q_1})$. Then $q_2$ and $\{\dot{E}_n:n\in\omega\}$ work.
\end{proof}

\begin{thm}\label{cohen_proper_equivalences}
Let $\mathbb{P}$ and $\mathbb{Q}$ be forcings such that $\mathbb{P}\lessdot\mathbb{Q}$ and $\pi:\mathbb{Q}\to\mathbb{P}$ a regular projection. The following are equivalent:
\begin{enumerate}
    \item $(\mathbb{P},\mathbb{Q})$ is Cohen proper.
    \item $\mathbb{Q}$ is $\sigma$-proper and $\mathbb{P}\Vdash[\mathbb{Q}:\mathbb{P}]\text{ is Cohen preserving}$.
    \item For any $\mathcal{M}\prec H(\theta)$ countable such that $\mathbb{Q},\mathbb{P}\in\mathcal{M}$, any open dense $D\subseteq(\mathbb{Q}\times\mathbb{C})\cap\mathcal{M}$ and any condition $q\in\mathbb{Q}\cap\mathcal{M}$, there are $q_0\leq q$ and a $\mathbb{P}$-name $\dot{E}$ such that,
    \begin{enumerate}
    \item $\pi(q_0)\Vdash\dot{E}\subseteq\mathbb{C}\text{ is open dense}$. 
    \item $q_0\Vdash\dot{E}\subseteq{D}^\divideontimes$.
    \end{enumerate}
    \item For any $\mathcal{M}\prec H(\theta)$ countable such that $\mathbb{Q},\mathbb{P}\in\mathcal{M}$, any $\mathbb{P}$-name $\dot{\mathbb{R}}\in\mathcal{M}$, any open dense $D\subseteq(\mathbb{Q}*_{\mathbb{P}}\dot{\mathbb{R}})\cap\mathcal{M}$ and any condition $q\in\mathbb{Q}\cap\mathcal{M}$, there are $q_0\leq q$ and a $\mathbb{P}$-name $\dot{E}$ such that,
    \begin{enumerate}
    \item $\pi(q_0)\Vdash\dot{E}\subseteq\dot{\mathbb{R}}\cap\mathcal{M}[{\pi[\dot\Gamma}]]\text{ is open dense}$. 
    \item $q_0\Vdash\dot{E}\subseteq{D}^\divideontimes$.
    \end{enumerate}

\end{enumerate}
\end{thm}

\begin{proof}
The implications $(4)\Longrightarrow(3)$ and $(1)\Longrightarrow(4)$ are easy, and the implication $(2)\Longrightarrow(1)$ is Proposition \ref{varsigma_proper_stringly_proper_cohen_preserving}. So we are left with implication $(3)\Longrightarrow(2)$.

$(3)\Longrightarrow(2)$.  Let $\dot{D}$ be a $\mathbb{Q}$-name for an open dense subset of $\mathbb{C}$. For each $n\in\omega$, let $s_n={\vec{1}^n}^\frown 0$. Now let $\mathcal{M}\prec H(\theta)$ be countable such that $\mathbb{P},\mathbb{Q}, \dot{D}\in\mathcal{M}$. Let $\mathcal{F}=\{F_n:n\in\omega\}$ be a family of open dense subsests of $\mathbb{Q}\cap\mathcal{M}$ and fix a condition $q\in\mathbb{Q}\cap\mathcal{M}$. We will find an extension $q_0\leq q$ and a $\dot{E}$ $\mathbb{P}$-name such that:
\begin{enumerate}
    \item[i)] $q_0$ is $(\mathcal{F},\mathbb{Q},\mathcal{M})$-generic.
    \item[ii)] $\pi(q_0)\Vdash\dot{E}\subseteq\mathbb{C}\text{ is open dense}$.
    \item[iii)] $q_0\Vdash\dot{E}\subseteq\dot{D}$.
\end{enumerate}
Let us first define $D_0=\{(q,s)\in\mathbb{Q}\times\mathbb{C}:q\Vdash s\in\dot{D}\}\cap\mathcal{M}$. It is clear that $D_0$ is an open dense subset of $(\mathbb{Q}\times\mathbb{C})\cap\mathcal{M}$. Define the set $H_0$ as follows:
\begin{equation*}
    H_0=\{(q, s_0^\frown t):(q,t)\in D_0\}
\end{equation*}

Now, for each positive $n\in\omega$, define a set $H_n$ as follows:
\begin{equation*}
    (p,t)\in H_n\Longleftrightarrow (p,t)\in D_0\land p\in F_{n-1}\land s_n\subseteq t
\end{equation*}
and let $H=\bigcup_{n\in\omega}H_n$. Let us see that $H$ is an open dense subset of $\mathbb{Q}\times\mathbb{C}\cap\mathcal{M}$. Let $(r,t)\in \mathbb{Q}\times\mathbb{C}\cap\mathcal{M}$ be arbitrary. We can assume $s_n\subseteq t$ for some $n\in\omega$. Since $D_0$ is open dense, there is $(p_0,t_0)\in D_0$ such that $(p_0,t_0)\leq (r,t)$. The set $F_n$ is open dense, there is $p_1\in F_n$ such that $p_1\leq p_0$, and since $D_0$ is open, we get $(p_1,t_0)\in D_0$, so we have $(p_1,t_0)\in D_0$, $p_1\in F_n$ and $s_n\subseteq t_0$, which implies $(p_1,t_0)\in H_n\subseteq H$, so $H$ is dense. To see that $H$ is open just note that each $H_n$ is open. Our hypothesis (3) implies that there is a condition $q_0\leq q$ and a $\mathbb{P}$-name $\dot{E}$ such that (3)(a) and (3)(b) hold for $H$. Then we have:
\begin{enumerate}
    \item $q_0\Vdash\dot{E}\subseteq H^\divideontimes$.
    \item $\pi(q_0)\Vdash\dot{E}\subseteq\mathbb{C}\text{ is open dense}$.
\end{enumerate}
Since $q_0$ forces $H^\divideontimes$ to contain a dense open set, we have that each $F_n$ is predense below $q_0$, since otherwise there is $n\in\omega$ such that $H^\divideontimes$ is not dense below $s_n$. On the other hand, $H_0^\divideontimes$ is open dense below $s_0$, and since $H_0$ is isomorphic to $D_0$, we have that $D_0^\divideontimes$ is open dense in $2^{<\omega}$. Finally note that $q_0\Vdash D_0^\divideontimes\subseteq\dot{D}_{\Gamma}$: let $G$ be a generic filter such that $q_0\in G$ and pick $t\in D_0^\divideontimes$, so there is $p\in G$ such that $(p,t)\in D_0$, and by definition of $D_0$ we have $p\Vdash t\in\dot{D}$, so it follows $t\in \dot{D}_{G}$.
\end{proof}

\begin{prp}\quad
\begin{enumerate}
    \item If $\mathcal{U}$ is a suitable filter, then $(\mathbb{S},\mathbb{S}\times\mathbb{P}_\kappa(\mathcal{U}))$ is Cohen proper.
    \item If $\mathbb{P}$ is a $\sigma$-proper forcing, and $\mathbb{P}$ forces $\dot{\mathbb{Q}}$ to be a $\sigma$-proper forcing and Cohen preserving, then $(\mathbb{P},\mathbb{P}*\dot{\mathbb{Q}})$ is Cohen proper.
    \item If $\mathbb{P}$ is $\sigma$-proper then $(\mathbb{P},\mathbb{P}\times\mathbb{Q}^\kappa)$ is Cohen proper, where $\mathbb{Q}^\kappa$ may be the countable support product of $\kappa$-many copies of the Sacks forcing, Silver forcing, Miller lite forcing, etc.
\end{enumerate}

\end{prp}

\section{Iteration of Cohen forcing.}\label{iteration_cohen_properness}

Here we define a way to iterate Cohen proper forcings which seems very similar to countable support iterations and allows us to get the continuum arbitrarily large.

\begin{dfn}
Let $\alpha$ be an ordinal. $\mathbf{P}=\langle\mathbb{P}_\beta,(\dot{\mathbb{Q}}_\beta^0,\dot{\mathbb{Q}}_\beta^1,\dot{\pi}_\beta):\beta<\alpha\rangle$ is a $\varsigma$-iteration if and only if:
\begin{enumerate}
    \item $\mathbf{P}_0^\alpha=\langle\mathbb{P}_\beta,\dot{\mathbb{Q}}_\beta^0:\beta<\alpha\rangle$ is a countable support iteration.
    \item For each $\beta<\alpha$, $\dot{\mathbb{Q}}_\beta^0$ and $\dot{\mathbb{Q}}_\beta^1$ are  $\mathbb{P}_\beta$-names for complete boolean algebras.
    \item For each $\beta<\alpha$, $\mathbb{P}_\beta\Vdash \dot{\mathbb{Q}}_\beta^0\lessdot\dot{\mathbb{Q}}_\beta^1$, and $\dot{\pi}_\beta:\dot{\mathbb{Q}}_\beta^1\to\dot{\mathbb{Q}}_\beta^0$ is a $\mathbb{P}_\beta$-name for the regular projection.
    \item $(p,q)\in\mathbf{P}$ if and only if:
    \begin{itemize}
        \item $p\in\mathbf{P}_0^\alpha$.
        \item $q$ is a countable function such that $dom(q)\subseteq\alpha$.
        \item For each $\beta\in dom(q)$, $p\upharpoonright\beta\Vdash p(\beta)\Vdash q(\beta)\in[\dot{\mathbb{Q}}_\beta^1:\dot{\mathbb{Q}}_\beta^0]$
    \end{itemize}
\end{enumerate}
For $(p_0,q_0),(p_1,q_1)\in\mathbf{P}$, define $(p_1,q_1)\leq(p_0,q_0)$ if and only if $p_1\leq_{\mathbf{P}_0} p_0$ and for each $\beta\in dom(q_0)$, $p_1\upharpoonright\beta\Vdash (p_1(\beta),q_1(\beta))\leq (p_0(\beta),q_0(\beta))$.

For $\beta<\alpha$, the restriction of $\mathbf{P}$ to $\beta$, denoted by $\mathbf{P}^\beta$, is the $\varsigma$-iteration $\langle\mathbb{P}_\gamma,(\dot{\mathbb{Q}}_\gamma^0,\dot{\mathbb{Q}}_\gamma^1):\gamma<\beta\rangle$, and $\mathbf{P}^{\beta}_0$ is the countable support iteration $\langle\mathbb{P}_\gamma,\dot{\mathbb{Q}}_\gamma^0:\gamma<\beta\rangle$.

In the language of partial orders, $(2)$ above is replaced by $(2')$ below, and we omit the $\mathbb{P}_\beta$-name $\dot{\pi}_\beta$ (everything else remains the same),
\begin{enumerate}
    \item [(2')] For each $\beta<\alpha$, $\dot{\mathbb{Q}}_\beta^0$ and $\dot{\mathbb{Q}}_\beta^1$ are  $\mathbb{P}_\beta$-names for partial orders.
\end{enumerate}
\end{dfn}

\begin{dfn}
Let $\mathbf{P}=\langle\mathbb{P}_\beta,(\dot{\mathbb{Q}}_\beta^0,\dot{\mathbb{Q}}_\beta^1,\dot{\pi}_\beta):\beta<\alpha\rangle$ be a $\varsigma$-iteration. We say that $\mathbf{P}$ is a Cohen proper $\varsigma$-iteration, if:
\begin{enumerate}
    \item $(\mathbb{Q}_0^0,\mathbb{Q}_0^1)$ is Cohen proper.
    \item For all $\beta<\alpha$, $\mathbb{P}_\beta\Vdash (\dot{\mathbb{Q}}_\beta^0,\dot{\mathbb{Q}}_\beta^1)$ is Cohen proper.
\end{enumerate}
\end{dfn}

\begin{dfn}
Along a $\varsigma$-iteration there are many different generic filters arising, and we name some of them we will need later. For $\beta<\alpha$, 
\begin{enumerate}
    \item[0)] $\Gamma_0=\{(p(0),q(0)):(p,q)\in\Gamma\}$.
    \item[1)] $\Gamma_{<\beta}=\{(p\upharpoonright\beta,q\upharpoonright\beta):(p,q)\in\Gamma\}$.
    \item[2)] $\Gamma_{<\beta}^0=\{p:(p,q)\in\Gamma_{<\beta}\text{ for some $q$}\}$.
    \item[3)] $\Gamma_0^0=\Gamma_{<1}^0=\{p(0):(p,q)\in\Gamma\text{ for some $q$}\}$.
    \item[4)] $\Gamma_{<\beta}^1=\{\langle q(\beta)[\Gamma_{\beta+1}^0]:\beta\in dom(q)\rangle:(\exists p)(p,q)\in\Gamma_{<\beta}\}$.
\end{enumerate}
\end{dfn}

\begin{dfn}
Let $\mathbf{P}=\langle\mathbb{P}_\beta,(\dot{\mathbb{Q}}_\beta^0,\dot{\mathbb{Q}}_\beta^1):\beta<\alpha\rangle$ be a $\varsigma$-iteration and $\beta_0\leq\alpha$ some fixed ordinal. Define $\mathbf{P}^{(*\beta_0)}$ as:
\begin{equation*}
\mathbf{P}^{(*\beta_0)}=\{(p,q*\beta_0):(p,q)\in\mathbf{P}\}
\end{equation*}
here, the operation $q*\beta_0$ keeps the meaning of previous sections and means $$q*\beta_0=\{(\gamma,\dot{q}_\gamma)\in q:\gamma\ge\beta_0\}$$ The order is given  by $(p_1,q_1)\leq(p_0,q_0)$ if and only if $p_1\leq_{\mathbf{P}_0} p_0$ and for each $\gamma\in dom(q_0)$, $p_1\upharpoonright\gamma\Vdash (p_1(\gamma),q_1(\gamma))\leq (p_0(\gamma),q_0(\gamma))$.

Note that for $\beta_0=\alpha$, for any $(p,q)\in\mathbf{P}^{(*\beta_0)}$ we have $q=\emptyset$.
If $\Gamma_{<\beta}^0$ is a $\mathbf{P}^\beta_0$-generic filter, define in $V[\Gamma_{<\beta}^0]$ the quotient forcing $[\mathbf{P}^{(*\beta)}:\mathbf{P}^\beta_0]$ and give it the order induced by $\mathbf{P}^{(*\beta)}$. 
\end{dfn}

It can be easily seen that $\mathbf{P}$ can be represented as a $\mathbf{P}^\beta_0$-restricted iteration of $\mathbf{P}^{\beta}$ and $[\mathbf{P}^{(*\beta)}:\mathbf{P}^\beta_0]$. Indeed, let $$[\mathbf{P}^{(*\beta)}:\mathbf{P}^\beta_0]=\{{((p,q){\check{\phantom{x}}} },p\upharpoonright\beta):(p,q)\in\mathbf{P}^{(*\beta)}\}$$
which is the natural name for the poset $[\mathbf{P}^{(*\beta)}:\mathbf{P}^\beta_0]$. Now let $$\mathbf{P}^\beta*_{\mathbf{P}^\beta_0}[\mathbf{P}^{(*\beta)}:\mathbf{P}^\beta_0]=\{((p,q),{(p_0,q_0)\check{\phantom{x}}}):(p,q)\in\mathbf{P}^\beta\land((p_0,q_0)\in\mathbf{P}^{(*\beta)})\land(p\leq p_0\upharpoonright\beta)\}$$
Note that this poset is a $\mathbf{P}^{\beta}_0$-restricted iteration equivalent to $\mathbf{P}$: there is a natural dense embedding $\varphi_\beta$ from $\mathbf{P}$ into $\mathbf{P}^\beta*_{\mathbf{P}^\beta_0}[\mathbf{P}^{(*\beta)}:\mathbf{P}^\beta_0]$ above, to say, $$(p,q)\xrightarrow{\varphi_\beta}((p\upharpoonright\beta,q\upharpoonright\beta),(p,q*\beta)\check{\phantom{x}})$$
For any $((p,q),(p_0,q_0)\check{\phantom{x}})\in\mathbf{P}^\beta*_{\mathbf{P}^\beta_0}[\mathbf{P}^{(*\beta)}:\mathbf{P}^\beta_0]$, we have $(p\upharpoonright\beta^\frown (p_0*\beta),q^\frown q_0)\in\mathbf{P}$, and its $\varphi_\beta$-image is below $((p,q),(p_0,q_0)\check{\phantom{x}})$. The factorization $\mathbf{P}^\beta*_{\mathbf{P}_0^\beta}[\mathbf{P}^{(*\beta)}:\mathbf{P}^\beta_0]$ will be useful later. Note that if $D\subseteq\mathbf{P}$ is an open dense set, then $\varphi[D]$ gives us a dense subset of $\mathbf{P}^\beta*_{\mathbf{P}^\beta_0}[\mathbf{P}^{(*\beta)}:\mathbf{P}^\beta_0]$, instead of an open dense subset. This remark comes to say that in the iteration theorem we will prove in the next section, we need to make a convention on how to deal with the just stated remark. If $D\subseteq\mathbf{P}$ is an open dense set and $\beta<\alpha$, let us define the set $\overline{D}^\beta$ as follows: 
\begin{equation*}
    \overline{D}^\beta=\{p\in\mathbf{P}^\beta*_{\mathbf{P}^\beta_0}[\mathbf{P}^{(*\beta)}:\mathbf{P}^\beta_0]:(\exists p_0\in D)(p\leq\varphi_\beta(p_0)\}
\end{equation*}
Similarly, if $\mathcal{M}\prec H(\theta)$ is countable and $\beta,\mathbf{P}\in\mathcal{M}$, for any $D\subseteq\mathbf{P}\cap\mathcal{M}$ open dense, define $\overline{D}^{\beta}$ as
\begin{equation*}
    \overline{D}^\beta=\{p\in\mathbf{P}^\beta*_{\mathbf{P}^\beta_0}[\mathbf{P}^{(*\beta)}:\mathbf{P}^\beta_0]\cap\mathcal{M}:(\exists p_0\in D)(p\leq\varphi_\beta(p_0)\}
\end{equation*}
Clearly, if $D\subseteq\mathbf{P}\cap\mathcal{M}$ is open dense, then $\overline{D}^\beta$ is an open dense subset of $\mathbf{P}^\beta*_{\mathbf{P}^\beta_0}[\mathbf{P}^{(*\beta)}:\mathbf{P}^\beta_0]\cap\mathcal{M}$.
This operation will be used only for sets of the form $D\subseteq\mathbf{P}\cap\mathcal{M}$, so there will be no danger of confusion about which of the two versions above is being used (the second one, always). We need one lemma that illustrates the reason of introducing these objects, and the proof of Theorem \ref{preservation_of_cohen_properness} will show how we make use of them. When $D\subseteq\mathbf{P}\cap\mathcal{M}$ and $\Gamma$ is a $\mathbf{P}^\beta$-generic filter, define $$D^\divideontimes_{\Gamma}=\{(p,q)\in D:(p\upharpoonright\beta,q\upharpoonright\beta)\in \Gamma\}=[D:\mathbf{P}^\beta]$$

\begin{lemma}
Let $\mathbf{P}$ be a $\varsigma$-iteration of length $\alpha$ and $\mathcal{M}\prec H(\theta)$ countable such that $\mathbf{P}\in\mathcal{M}$ and pick $\beta\in\alpha\cap\mathcal{M}$. Let $D\subseteq\mathbf{P}\cap\mathcal{M}$ be an open dense set and assume $(p,q)\in\mathbf{P}^\beta$ is $(\mathbf{P}^\beta_0,\mathbf{P}^\beta,\mathcal{M})$-generic for $[\mathbf{P}^{(*\beta)}:\mathbf{P}^\beta_0]$ and $\overline{D}^\beta$, and let $\dot{E}$ be a $\mathbf{P}^\beta_0$-name witnessing this. Then, 
\begin{equation*}
(p,q)\Vdash(\overline{D}^\beta)^\divideontimes=\{(r,s*\beta):(r,s)\in D^\divideontimes_{\dot{\Gamma}}\}
\end{equation*}
and
\begin{equation*}
    (p,q)\Vdash\dot{E}\subseteq\{(r,s*\beta):(r,s)\in D^\divideontimes_{\dot{\Gamma}}\}
\end{equation*}
Therefore, $(p,q)$ forces that for any $(r,s)\in (\overline{D}^\beta)^\divideontimes_\Gamma$, there is $\tau$ such that $(r,\tau^\frown s)\in D$. In particular, this holds for all elements of $\dot{E}_{\Gamma}$.
\end{lemma}

\begin{proof}
It is enough to prove the first forcing relation, as the second follows from the first one, and the choice of $(p,q)$ and $\dot{E}$. Let $\Gamma$ be a $\mathbf{P}^\beta$-generic filter such that $(p,q)\in\Gamma$. Pick $(r,s)\in(\overline{D}^\beta)^\divideontimes_\Gamma$. By definition of $\overline{D}^\beta$, there is $(r_0,s_0)\in D$ such that:
\begin{enumerate}
    \item $(r_0\upharpoonright\beta,s_0\upharpoonright\beta)\in \Gamma$.
    \item $(r,s)\leq(r_0,s_0*\beta)$.
\end{enumerate}
Since $(\overline{D}^\beta)^\divideontimes_{\Gamma}\subseteq[\mathbf{P}^{(*\beta)}:\mathbf{P}^\beta_0]\cap\mathcal{M}$, we have that $r\upharpoonright\beta\in\Gamma_{<\beta}^0$. Note that $(r,s_0\upharpoonright\beta^\frown s)$ is actually a condition in $\mathbf{P}\cap\mathcal{M}$ and $(r,s_0\upharpoonright\beta^\frown s)\leq(r_0,s_0)\in D$, so $(r,s_0\upharpoonright\beta^\frown s)\in D$. Since we also have $(r\upharpoonright\beta,s_0\upharpoonright\beta)\in \Gamma$, we conclude that $(r,s_0\upharpoonright\beta^\frown s)\in D^\divideontimes_{\Gamma}$. It follows that $(r,s)\in\{(\rho,\eta*\beta):(\rho,\eta)\in D_{\Gamma}^\divideontimes\}$.  The other contention follows from the fact that $\varphi_\beta[D]\subseteq\overline{D}^\beta$ and the definition of the operation $\divideontimes$.
\end{proof}

\noindent\textbf{Convention.} Finally, before passing to the next section, whenever we mention a factorization of the form $\mathbf{P}^\beta*_{\mathbf{P}^\beta_0}[\mathbf{P}^{(*\beta)}:\mathbf{P}^\beta_0]$, we assume this factorization of $\mathbf{P}$ has been done accordingly to the previous paragraphs.

\section{A preservation theorem.}\label{A_preservation_theorem}

We need a few things before going into the preservation theorem. Assume $\mathbf{P}$ is a $\varsigma$-iteration, $\mathcal{M}\prec H(\theta)$ is countable such that $\mathbf{P}\in\mathcal{M}$ and pick a $\mathbf{P}_0^\alpha$-name $\dot{\mathbb{R}}\in\mathcal{M}$. Given a family $\mathcal{D}=\{D_n:n\in\omega\}$ of open dense subsets of $(\mathbf{P}*_{\mathbf{P}_0^\alpha}\dot{\mathbb{R}})\cap\mathcal{M}$, we say that $\mathcal{D}$ is $\mathbf{P}$-sufficient\footnote{It would be more correct to write something like $(\mathbf{P},\mathcal{M})$-sufficient, since this notion depends on $\mathcal{M}$ as well, but to keep simplicity we write $\mathbf{P}$-sufficient. Since this notion is only used in the proof of the next theorem and $\mathcal{M}$ is fixed, there is not danger of confusion.} if any $(\mathbf{P}^\alpha_0,\mathbf{P},\mathcal{M})$-generic condition for $\{D_n:n\in\omega\}$ and $\dot{\mathbb{R}}$, is also a $(\mathbf{P},\mathcal{M})$-generic condition. If $\{D_n:n\in\omega\}$ is the family of dense open subsets of $\mathbf{P}$ living in $\mathcal{M}$, and $E\subseteq(\mathbf{P}*_{\mathbf{P}_0^\alpha}\dot{\mathbb{R}})\cap\mathcal{M}$ is open dense, it is easy to see that the family $\{H_n:n\in\omega\}$ defined as $$H_n=\{(p^\frown\dot{r},q)\in E:(p,q)\in \bigcap_{j\leq n}D_j\}$$ is a $\mathbf{P}$-sufficient family of open dense sets. Moreover, for any $\beta\in\alpha\cap\mathcal{M}$, it turns out that $\{\overline{H}_n^\beta:n\in\omega\}$ is $\mathbf{P}^\beta$-sufficient\footnote{Let $D\in\mathcal{M}$ be an open dense subset of $\mathbf{P}^\beta$ and define $D'=\{(p,q)\in D_0: (p\upharpoonright\beta,q\upharpoonright\beta)\in D\}$, this is an open dense subset of $\mathbf{P}$ and lives in $\mathcal{M}$, so there is $k\in\omega$ such that $D_{k}=D'$. Then we have that for any $((p\upharpoonright\beta,q\upharpoonright\beta),(p^\frown\dot{r},q*\beta)\check{\phantom{x}})\in \varphi_\beta[H_k]$, $(p\upharpoonright\beta,q\upharpoonright\beta)\in D$, which implies the same about $\overline{H}_k^\beta$.}. We will denote this family by $\mathcal{H}_\beta(E,\mathbf{P},\dot{\mathbb{R}})$, that is $\mathcal{H}_\beta(E,\mathbf{P},\dot{\mathbb{R}})=\{\overline{H}_n^\beta:n\in\omega\}$ and assume this enumeration is fixed.

\begin{thm}\label{preservation_of_cohen_properness}
Let $\mathbf{P}=\langle \mathbb{P}_\beta,(\dot{\mathbb{Q}}_0^\beta,\dot{\mathbb{Q}}_1^\beta):\beta<\alpha\rangle$ be a $\varsigma$-iteration of Cohen proper forcings. Let $\mathcal{M}\prec H(\theta)$ be countable such that $\mathbf{P}\in\mathcal{M}$ and $\dot{\mathbb{R}}\in\mathcal{M}$ a $\mathbf{P}_0^\alpha$-name for a forcing. Fix $\beta_0\in\alpha\cap\mathcal{M}$. Let $D\subseteq(\mathbf{P}*_{\mathbf{P}^\alpha_0}\dot{\mathbb{R}})\cap\mathcal{M}$ be an open dense set, a condition $(p_0,q_0)\in\mathbf{P}^{\beta_0}$ and assume that $(p_0,q_0)$ is $(\mathbf{P}^{\beta_0}_0,\mathbf{P}^{\beta_0},\mathcal{M})$-generic for $[\mathbf{P}^{(*\beta_0)}*_{\mathbf{P}_0^\alpha}\dot{\mathbb{R}}:\mathbf{P}^{\beta_0}_0]$ and $\mathcal{H}_{\beta_0}(D,\mathbf{P},\dot{\mathbb{R}})$. Furthermore, assume that $\tilde{\eta}$ is a $\mathbf{P}^{\beta_0}_0$-name such that $p_0\Vdash\tilde{\eta}=(\tilde{\eta}_0,\tilde{\eta}_1)\in[\mathbf{P}^{(*\beta_0)}:\mathbf{P}^{\beta_0}_0]\cap\mathcal{M}$. Then there is a condition $(\overline{p},\overline{q})\in\mathbf{P}$ and a ${\mathbf{P}}_0^\alpha$-name $\dot{E}$ such that:
\begin{enumerate}
    \item $p_0=\overline{p}\upharpoonright\beta_0$ and $q_0=\overline{q}\upharpoonright\beta_0$.
    \item $(\overline{p},\overline{q})$ is $(\mathbf{P}_0^\alpha,\mathbf{P},\mathcal{M})$-generic for $[\mathbf{P}^{(*\alpha)}*_{\mathbf{P}_0^\alpha}\dot{\mathbb{R}}:\mathbf{P}_0^\alpha](\equiv\dot{\mathbb{R}})$ and $\mathcal{H}_\alpha(D,\mathbf{P},\dot{\mathbb{R}})$.
    \item $\overline{p}\Vdash\dot{E}\subseteq\dot{\mathbb{R}}\cap\mathcal{M}[\dot{\Gamma}_{<\alpha}^0]\text{ is open dense}$.
    \item $(\overline{p},\overline{q})\Vdash \dot{E}\subseteq ({D})^\divideontimes=\{\dot{r}_{\Gamma}:(\exists (p,q)\in\mathbf{P})((p^\frown\dot{r},q)\in D)\}$.
    \item There is a $\mathbf{P}^{\beta_0}$-name $\tilde{\tau}$ such that $(\overline{p},\overline{q})\Vdash(\tilde{\eta}_0,\tilde{\tau}^\frown\tilde{\eta}_1)\in\dot{\Gamma}$.
\end{enumerate}
\end{thm}

\begin{proof}
The proof is by induction on the length of $\mathbf{P}$. 

Let us deal with the successor step first and assume $\alpha=\gamma+1$ and the theorem is true for $\gamma$. Then 1) we can extend $(p_0,q_0)$ to a condition $(\overline{p}_0,\overline{q}_0)\in\mathbf{P}^{\gamma}$ which is $(\mathbf{P}^\gamma_0,\mathbf{P}^\gamma,\mathcal{M})$-generic for $[\mathbf{P}^{(*\gamma)}*_{\mathbf{P}_0^\alpha}\dot{\mathbb{R}}:\mathbf{P}^\gamma_0]$ and $\mathcal{H}_\gamma(D,\mathbf{P},\dot{\mathbb{R}})$; 2) there is a $\mathbf{P}_\gamma$-name $\dot{E}^0$ such that $(\overline{p}_0,\overline{q}_0)$ forces $\dot{E}^0\subseteq(\overline{D}^\gamma)^\divideontimes$, and 3) for each $n\in\omega$, there is a $\mathbf{P}^\gamma$-name for an open dense set $\dot{F}_n^0\subseteq[\mathbf{P}^{(*\gamma)}*_{\mathbf{P}_0^\alpha}\dot{\mathbb{R}}:\mathbf{P}^\gamma_0]\cap\mathcal{M}$ such that $(\overline{p}_0,\overline{q}_0)$ forces $\dot{F}_n^0\subseteq (\overline{H}_n^\gamma)^\divideontimes$. Let $\Gamma$ be a $\mathbf{P}^\gamma$-generic filter such that $(\overline{p}_0,\overline{q}_0)\in\Gamma$; 4) there is ${\tau}_0$ such that $\tilde{\eta}_{\Gamma_{<\gamma}^0}=({\eta}_0,\tau_0^\frown{\eta}_1)\in\mathcal{M}[\Gamma_{<\gamma}^0]$ and $(\eta_0\upharpoonright\gamma,(\tau_0^\frown\eta_1\upharpoonright\gamma))\in\Gamma$, which implies $(\eta_0,\eta_1*\gamma)\in[\mathbf{P}^{(*\gamma)}*_{\mathbf{P}_0^\alpha}\dot{\mathbb{R}}:\mathbf{P}_0^\gamma]\cap\mathcal{M}$. Going to $V[\Gamma_{<\gamma}^0]$ we have that $[\mathbf{P}^{(*\gamma)}:\mathbf{P}_0^\gamma]$ is forcing equivalent to $\mathbb{Q}_0^\gamma*[\dot{\mathbb{Q}}_1^\gamma:\dot{\mathbb{Q}}_0^\gamma]$, and $(\dot{\mathbb{Q}}_0^\gamma,\dot{\mathbb{Q}}_1^\gamma)$ is Cohen proper. Thus, we can find a condition $(p_1,q_1)\leq(\eta_0,\eta_1*\gamma)$ which is $([\mathbf{P}_0^{\gamma+1}:\mathbf{P}^{\gamma}_0],[\mathbf{P}^{(*\gamma)}:\mathbf{P}^\gamma_0],\mathcal{M}[\Gamma_{<\gamma}^0])$-generic for $\dot{\mathbb{R}}$ and the family $$\{\dot{E}_{\Gamma_{<\gamma}^0}^0\}\cup\{F_n^0:n\in\omega\}$$ 
let $\dot{E}^1$, $\{\dot{F}_n^1:n\in\omega\}$ be a family of $[\mathbf{P}^{\gamma+1}_0:\mathbf{P}^\gamma_0]$-names witnessing $(p_1,q_1)$ is $([\mathbf{P}_0^{\gamma+1}:\mathbf{P}^{\gamma}_0],[\mathbf{P}^{(*\gamma)}:\mathbf{P}^\gamma_0],\mathcal{M}[\Gamma_{<\gamma}^0])$-generic for the given family and $\dot{\mathbb{R}}$. Going to $V[\Gamma]$, there is $\tau_1$ such that $(p_1\upharpoonright\beta_0,\tau_1)\in\Gamma$; note that we can actually get $\tau_1$ such that $(p_1\upharpoonright\beta_0,\tau_1)\leq(\eta_0\upharpoonright\beta_0,\tau_0)$, in particular we also have $(p_1,\tau_1^\frown q_1)\leq(\eta_0,\tau_0^\frown \eta_1)$.
Going back to $V$, $\tilde{\tau}_1$ be a $\mathbf{P}^\gamma$-name for $\tau_1$, let $\tilde{p}_1$ and $\tilde{q}_1$ be $\mathbf{P}^\gamma_0$-names for $p_1$ and $q_1$, respectively, and $\tilde{\nu},\tilde{\eta}$ $\mathbf{P}^\gamma_0$-names for $p_1(\gamma)$ and $q_1(\gamma)$, respectively, and define $\overline{p}=\overline{p}_0^\frown\tilde{\nu}$ and $\overline{q}=\overline{q}_0^\frown\tilde{\eta}$. Let $\tilde{E}^1$ be a translation of $\dot{E}^1$ to a $\mathbf{P}^\alpha_0$-name, and $\tilde{F}_n^1$ the translation of $\dot{F}_n^1$ to a $\mathbf{P}^\alpha_0$-name. We claim that $(\overline{p},\overline{q})$ and $\tilde{E}^1,\{\tilde{F}_n^1:n\in\omega\}$ make the work. Let now $\Gamma$ be a $\mathbf{P}$-generic filter such that $(\overline{p},\overline{q})\in\Gamma$. It is not hard to see that $(\overline{p},\overline{q})\Vdash (\tilde{p}_1,\tilde{\tau}_1^\frown\tilde{q}_1)\in\Gamma$, which implies that $\tilde{E}_{\Gamma_{<\alpha}^0}^1\subseteq\dot{\mathbb{R}}_{\Gamma_{<\alpha}^0}\cap\mathcal{M}[\Gamma_{<\alpha}^0]$ is an open dense set, and $\tilde{E}_{\Gamma_{<\alpha}^0}^1\subseteq(\dot{E}^0_{\Gamma_{<\gamma}^0})^{\divideontimes}_{_{\Gamma_{<\alpha}^{(*\gamma)}}}$, and the choice of $\dot{E}^0$ implies in turn that $\tilde{E}_{\Gamma_{<\alpha}^0}^1\subseteq({D})^\divideontimes_{\Gamma}$: pick any $r\in\tilde{E}_{\Gamma_{<\alpha}^0}^1$, since $r\in (\dot{E}_{\Gamma_{<\gamma}^0}^0)^\divideontimes_{\Gamma_{<\alpha}^{(*\gamma)}}$, there is $(p,q)\in[\mathbf{P}^{(*\gamma)}:\mathbf{P}^\gamma_0]$ such that $(p^\frown\dot{r},q)\in \dot{E}_{\Gamma_{<\gamma}^0}^0$ and $\dot{r}_{\Gamma_{<\alpha}^0}=r$; now, since $\dot{E}^0_{\Gamma_{<\gamma}^0}\subseteq({D})^\divideontimes_{\Gamma_{<\gamma}}$, there is $\nu$ such that $(p^\frown\dot{r},\nu^\frown q)\in D$, which implies $\dot{r}_{\Gamma_{<\alpha}^0}\in({D})^\divideontimes_{\Gamma}$; $\tilde{E}^1_{\Gamma_{<\alpha}^0}\subseteq({D})^\divideontimes_{\Gamma}$ follows. The same argument works to prove that $\tilde{F}_n^1$ is forced to be contained in $(\overline{H_n}^{\alpha})^\divideontimes$. On the other hand, we have $(\overline{p},\overline{q})\Vdash(\tilde{p}_1,\tilde{\tau}_1^\frown\tilde{q}_1)\leq(\tilde{\eta}_0,\tilde{\tau}_0^\frown\tilde{\eta}_1)$ and $(\overline{p},\overline{q})\Vdash(\tilde{p}_1,\tilde{\tau}_1^\frown\tilde{q}_1)$, which implies that $(\overline{p},\overline{q})\Vdash (\tilde{\eta_0},\tilde{\tau}_0^\frown \tilde{\eta}_1)\in\dot{\Gamma}$, and clause (5) above also holds.

Now we assume that $\alpha$ is a limit ordinal. Assume all the hypothesis and let $\{\beta_n:n\in\omega\}\subseteq\mathcal{M}\cap\alpha$ be a strictly increasing sequence cofinal in $\mathcal{M}\cap\alpha$, where $\beta_0$ is the ordinal from the hypothesis. Working in $\mathcal{M}$, let $\tilde{\lambda}$ be a $\mathbf{P}_0$-name for $\vert\dot{\mathbb{R}}\vert_{\dot{\Gamma}_{<\alpha}^0}$, and let $\tilde{\varphi}_0:\tilde{\lambda}\to\dot{\mathbb{R}}$ be a $\mathbf{P}_0$-name for a bijective function from $\tilde{\lambda}$ to $\dot{\mathbb{R}}_{\Gamma_{<\alpha}^0}$. Now let $\tilde{\varphi}_1:\dot{\mathbb{R}}_{\dot{\Gamma}_{<\alpha}^0}\to dom(\dot{\mathbb{R}})$ a function such that for each $r\in\dot{\mathbb{R}}_{\dot{\Gamma}_{<\alpha}^0}$, $r=\tilde{\varphi}_1(r)_{\dot{\Gamma}_{<\alpha}^0}$. Finally, let $\tilde{\varphi}_2=\tilde{\varphi}_1\circ\tilde{\varphi}_0$. We will use these functions to keep track of finite portions of the evaluation of $\dot{\mathbb{R}}\cap\mathcal{M}$, which we need to construct the set $\tilde{F}_0$ below stated.

To prove clauses (3) and (4), we need to find $(\overline{p},\overline{q})$ and construct a $\mathbf{P}_0^\alpha$-name  $\tilde{F}$ for an open dense subset of $\dot{\mathbb{R}}\cap\mathcal{M}$ such that
\begin{equation*}
    (\overline{p},\overline{q})\Vdash\tilde{F}_0\subseteq (D)^\divideontimes
\end{equation*}
We will focus on the construction of this set and then see that this can be used to prove clause (2). Note that actually it is enough to find a $\mathbf{P}_0^\alpha$-name $\tilde{F}_0$ for a dense subset (we don't need to worry about it being open) of $\dot{\mathbb{R}}_{\dot{\Gamma}_{<\alpha}^0}\cap\mathcal{M}[\Gamma_{<\alpha}^0]$ for which the previous forcing relation holds.

Let $\Gamma_{<\beta_0}$ be a $\mathbf{P}^{\beta_0}$-generic filter such that $(p_0,q_0)\in\Gamma_{<\beta_0}$. Since $(p_0,q_0)$ is $(\mathbf{P}_0^{\beta_0},\mathbf{P}^\beta,\mathcal{M})$-generic for $\mathcal{H}_{\beta_0}(D,\mathbf{P},\dot{\mathbb{R}})$ and $[\mathbf{P}*_{\mathbf{P}_0^\alpha}\dot{\mathbb{R}}:\mathbf{P}_0^{\beta_0}]$, there is a $\mathbf{P}^{\beta_0}_0$-name $\dot{E}^{0}$ for a dense subset of $[\mathbf{P}^{(*\beta_0)}*_{\mathbf{P}_0^\alpha}\dot{\mathbb{R}}:\mathbf{P}^{\beta_0}_0]\cap\mathcal{M}[\Gamma_{<\beta_0}^0]$ such that,
\begin{enumerate}
    \item[$(1_0)$] In $V[\Gamma_{<\beta_0}]$, $\dot{E}^{0}_{\Gamma_{<\beta_0}^0}\subseteq(\overline{H_0}^{\beta_0})^\divideontimes_{\Gamma_{<\beta_0}}$.
    \item[$(2_0)$] In $V[\Gamma_{<\beta_0}^0]$, $\dot{E}^{0}_{\Gamma_{<\beta_0}^0}$ is an open dense subset of $[\mathbf{P}^{(*\beta_0)}*_{\mathbf{P}_0^\alpha}\dot{\mathbb{R}}:\mathbf{P}^{\beta_0}_0]\cap\mathcal{M}[\Gamma_{<\beta_0}^0]$.
    \item[$(3_0)$] In $V[\Gamma_{<\beta_0}]$, for each $(t^\frown\dot{r},s)\in\dot{E}^{0}_{\Gamma_{<\beta_0}^0}$, there is $\varrho$ from $V$ such that $(t^\frown\dot{r},\varrho^\frown s)\in H_0$ and $(t\upharpoonright\beta_0,\varrho)\in\Gamma_{<\beta_0}$.

    \end{enumerate}
    and we also have,
    \begin{enumerate}
    \item[$(4_0)$] In $V[\Gamma_{<\beta_0}^0]$, $\tilde{\eta}_{\Gamma_{<\beta_0}^0}=(\eta_0,\eta_1)\in[\mathbf{P}^{(*\beta_0)}:\mathbf{P}^{\beta_0}_0]\cap\mathcal{M}[\Gamma_{<\beta_0}^0]$.
    \item[$(5_0)$] In $V[\Gamma_{<\beta_0}]$, there is $\nu$ such that $(\eta_0\upharpoonright\beta_0,\nu)\in\Gamma_{<\beta_0}\cap\mathcal{M}[\dot{\Gamma}_{<\beta_0}]$.
\end{enumerate}

In $V[\Gamma_{<\beta_0}^0]$, extend $\eta_0$ to a condition $\tilde{p}_0'\in[\mathbf{P}_0^\alpha:\mathbf{P}_0^{\beta_0}]\cap\mathcal{M}$ which decides $\tilde{\lambda}$, say $\tilde{p}_0'\Vdash \tilde{\lambda}=\kappa_0$ for a fixed $\kappa_0\in\mathcal{M}\cap\mathsf{Ord}$. Let $\{\gamma_n:n\in\omega\}$ be an enumeration of $\mathcal{M}\cap\kappa_0$. Now extend $\tilde{p}_0'$ to a condition $\tilde{p}_0''\in[\mathbf{P}_0^\alpha:\mathbf{P}_0^{\beta_0}]\cap\mathcal{M}$ which decides the value of $\tilde{\varphi}_2(\gamma_0)$, that is,
\begin{equation*}
    \tilde{p}_0''\Vdash \tilde{\varphi}_0(\gamma_0)=\tilde{\varphi}_2(\gamma_0)_{\dot{\Gamma}_{<\alpha}^0}
\end{equation*}
Let $\tilde{r}_0=\tilde{\varphi}_2(\gamma_0)$. Note that we can assume $\tilde{r}_0\in\mathcal{M}$. Now consider the condition $(\tilde{p}_0''^\frown\tilde{r}_0,\tilde{\eta}_1)\in[\mathbf{P}^{(*\beta_0)}*_{\mathbf{P}_0^\alpha}\dot{\mathbb{R}}:\mathbf{P}_0^{\beta_0}]\cap\mathcal{M}[\Gamma_{<\beta_0}^0]$. Since $\dot{E}^{0}_{\Gamma_{<\beta_0}^0}$ is an open dense subset of $[\mathbf{P}^{(*\beta_0)}*_{\mathbf{P}_0^\alpha}\dot{\mathbb{R}}:\mathbf{P}_0^{\beta_0}]\cap\mathcal{M}$ we can find $(a^\frown\dot{\rho},b)\in\dot{E}^0_{\Gamma_{<\beta_0}^0}$, such that $(a^\frown\dot{\rho},b)\leq(\tilde{p}_0''^\frown\tilde{r}_0,\tilde{\eta}_1)$. Let $(\tilde{p}_0^\frown\tilde{\rho}_0,\tilde{q}_0)$ be a $\mathbf{P}_0^{\beta_0}$-name for $(a^\frown\dot{\rho},b)$. Let us now go to $V[\Gamma_{<\beta_0}]$. Since $(\tilde{p}_0^\frown\tilde{\rho}_0,\tilde{q}_0)\in\dot{E}^{0}_{\Gamma_{<\beta_0}^0}\subseteq(\overline{H_0}^{\beta_0})^\divideontimes_{\Gamma_{<\beta_0}}$, we find $\tau_0$ such that:
\begin{enumerate}
    \item[$(A_0)$] $(\tilde{p}_0^\frown\tilde{\rho}_0,\tau_0^\frown\tilde{q}_0)\in H_0$.
    \item[$(B_0)$] $(\tilde{p}_0\upharpoonright\beta_0,\tau_0)\in\Gamma_{<\beta_0}$.
    \item[$(C_0)$] $(\tilde{p}_0,\tau_0^\frown\tilde{q}_0)\leq ({\eta}_0,{\nu}^\frown{\eta}_1)$.
\end{enumerate}
Now go to $V$ and let $\tilde{\tau}_0$ be a $\mathbf{P}^{\beta_0}$-name for $\tau_0$. We apply the induction hypothesis as follows:
\begin{enumerate}
    \item[$i_0)$] $\beta_0$ as it is and $\beta_1$ taking the place of $\alpha$.
    \item[$i_0)$] $\mathbf{P}^{\beta_1}$ taking the place of $\mathbf{P}$.
    \item[$iii_0)$] $[\mathbf{P}^{(*\beta_1)}*_{\mathbf{P}_0^\alpha}\dot{\mathbb{R}}:\mathbf{P}_0^{\beta_1}]$ taking the place of $\dot{\mathbb{R}}$.
    \item[$iv_0)$] $(p_0,q_0)$ as it is.
    \item[$v_0)$] $(\tilde{p}_0\upharpoonright\beta_1,\tilde{q}_0\upharpoonright\beta_1)$ taking the place of $\tilde{\eta}$.
    \item[$vi_0)$] The family $\mathcal{H}_{\beta_1}(D,\mathbf{P}^\beta,\dot{\mathbb{R}})$ taking the place of $\mathcal{H}_\alpha(D,\mathbf{P},\dot{\mathbb{R}})$.
\end{enumerate}
Thus, we find $(p_1,q_1)\in\mathbf{P}^{\beta_1}$ which is $(\mathbf{P}_0^{\beta_1},\mathbf{P}^{\beta_1},\mathcal{M})$-generic for $\mathcal{H}_{\beta_1}(D,\mathbf{P},\dot{\mathbb{R}})$ and $[\mathbf{P}^{(*\beta_1)}*_{\mathbf{P}_0^\alpha}\dot{\mathbb{R}}:\mathbf{P}^{\beta_1}_0]$. In particular, there is a $\mathbf{P}_0^{\beta_1}$-name $\dot{E}^1$ such that:
\begin{enumerate}
    \item[$\text{I}_0$)] $p_1\Vdash \dot{E}^1\subseteq[\mathbf{P}^{(*\beta_1)}*_{\mathbf{P}_0^\alpha}\dot{\mathbb{R}}:\mathbf{P}_0^{\beta_1}]\cap\mathcal{M}$ is an open dense set.
    \item[$\text{II}_0$)] $(p_1,q_1)\Vdash\dot{E}^1\subseteq(\overline{H_1}^{\beta_1})^\divideontimes$.
    \end{enumerate}
and we also have,
    \begin{enumerate}
    \item[$\text{III}_0$)] $(p_0,q_0)=(p_1\upharpoonright\beta_0,q_1\upharpoonright\beta_0)$.
     \item[$\text{IV}_0$)] $p_1\Vdash (\tilde{p}_0^\frown\tilde{\rho}_0,\tilde{q}_0*\beta_1)\in[\mathbf{P}^{(*\beta_1)}*_{\mathbf{P}_0^\alpha}\dot{\mathbb{R}}:\mathbf{P}_0^{\beta_1}]\cap\mathcal{M}$.
     
    \item[$\text{V}_0$)] There is a $\mathbf{P}^{\beta_0}$-name $\tilde{\tau}_0$ such that $$(p_1,q_1)\Vdash (\tilde{p}_0\upharpoonright\beta_1,\tilde{\tau}_0^\frown\tilde{q}_0\upharpoonright\beta_1)\in\Gamma_{<\beta_1}\cap\mathcal{M}$$ and $$(p_1,q_1)\Vdash(\tilde{p}_0^\frown\tilde{\rho}_0,\tilde{\tau}_0^\frown\tilde{q}_0)\in H_0\subseteq D$$

\end{enumerate}
Moreover, 
\begin{enumerate}
    \item [$\text{VI}_0$)] $p_1\Vdash\tilde{p}_0\Vdash\tilde{\rho}_0\leq\tilde{\varphi}_2(\gamma_0)$.
\end{enumerate}
This finishes the first step of the induction.

Suppose we are at step $n>0$ of the induction and we have the following objects:
\begin{enumerate}
    \item $(p_n,q_n)\in\mathbf{P}^{\beta_n}$ which is $(\mathbf{P}_0^{\beta_n},\mathbf{P}^{\beta_n},\mathcal{M})$-generic for $\mathcal{H}_{\beta_n}(D,\mathbf{P},\dot{\mathbb{R}})$ and $[\mathbf{P}^{(*\beta_n)}*_{\mathbf{P}_0^\alpha}\dot{\mathbb{R}}:\mathbf{P}_0^{\beta_n}]$.
    \item A $\mathbf{P}^{\beta_n}_0$-name $\dot{E}^n$.
    \item A $\mathbf{P}_0^{\beta_n}$-name $({\tilde{p}_{n-1}}^\frown\tilde{\rho}_{n-1},\tilde{q}_{n-1})$ and a $\mathbf{P}^{\beta_{n-1}}$-name $\tilde{\tau}_{n-1}$.
\end{enumerate}
such that:
\begin{enumerate}
    \item[$(1_n)$] $(p_n,q_n)\Vdash\dot{E}^n\subseteq (\overline{H_{n}}^{\beta_n})^\divideontimes$.
    \item[$(2_n)$] $p_n$ forces $\dot{E}^n$ to be an open dense subset of $[\mathbf{P}^{(*\beta_n)}*_{\mathbf{P}_0^\alpha}\dot{\mathbb{R}}:\mathbf{P}_0^{\beta_n}]\cap\mathcal{M}$.
    \item[$(3_n)$] $(p_n,q_n)$ forces that for each $(t^\frown\dot{r},s)\in\dot{E}^n$ there is $\varrho$ from $V$ such that $(t^\frown\dot{r},\varrho^\frown s)\in H_n$.
    \item[$(4_n)$] $p_n\Vdash ({\tilde{p}_{n-1}}^\frown\tilde{\rho}_{n-1},\tilde{q}_{n-1}*\beta_n)\in[\mathbf{P}^{(*\beta_n)}*_{\mathbf{P}_0^\alpha}\dot{\mathbb{R}}:\mathbf{P}_0^{\beta_n}]\cap\mathcal{M}$.
    \item[$(5_n)$] $(p_n,q_n)\Vdash(\tilde{p}_{n-1}\upharpoonright\beta_n,\tilde{\tau}_{n-1}^\frown\tilde{q}_{n-1}\upharpoonright\beta_{n})\in\dot{\Gamma}_{<\beta_n}\cap\mathcal{M}$.

    \item[$(6_n)$] $(p_n,q_n)\Vdash({\tilde{p}_{n-1}}^\frown\tilde{\rho}_{n-1},\tilde{\tau}_{n-1}^\frown\tilde{q}_{n-1})\in H_{n-1}$.
    \item[$(7_n)$] $(p_n,q_n)\Vdash (\tilde{p}_{n-1},\tilde{\tau}_{n-1}^\frown\tilde{q}_{n-1})\leq (\tilde{p}_{n-2},\tilde{\tau}_{n-2}^\frown\tilde{q}_{n-2})$.
    \item[$(8_n)$] $p_n\Vdash\tilde{p}_{n-1}\Vdash \tilde{\rho}_{n-1}\leq\tilde{\varphi}_2(\gamma_{n-1})$.
\end{enumerate}

Let $\Gamma_{<\beta_{n}}$ be a $\mathbf{P}^{\beta_n}$-generic filter such that $(p_n,q_n)\in\Gamma_{<\beta_n}$ and let us work first in $V[\Gamma_{<\beta_n}^0]$. Then we have, $({\tilde{p}_{n-1}}^\frown\tilde{\rho}_{n-1},\tilde{q}_{n-1}*\beta_n)\in[\mathbf{P}^{(*\beta_n)}*_{\mathbf{P}_0^\alpha}\dot{\mathbb{R}}:\mathbf{P}_0^{\beta_n}]\cap\mathcal{M}$. Extend $\tilde{p}_{n-1}$ to a condition $\tilde{p}_n'\in[\mathbf{P}_0^\alpha:\mathbf{P}_0^{\beta_n}]\cap\mathcal{M}$ which decides the value of $\tilde{\varphi}_2(\gamma_{n})$, that is, let $\tilde{r}_n\in dom(\dot{\mathbb{R}})$ be such that
\begin{equation*}
\tilde{p}_n'\Vdash\tilde{\varphi}_0(\gamma_{n})=\tilde{\varphi}_2(\gamma_{n})_{\Gamma_{<\beta_n}^0}=(\tilde{r}_n)_{\Gamma_{<\beta_n}^0}
\end{equation*}
Note that we can get $\tilde{r}_n\in\mathcal{M}$. Consider the condition $({\tilde{p}_n'}^\frown\tilde{r}_n,\tilde{q}_{n-1}*\beta_n)\in[\mathbf{P}^{(*\beta_n)}*_{\mathbf{P}_0^\alpha}\dot{\mathbb{R}}:\mathbf{P}_0^{\beta_n}]\cap\mathcal{M}$. Since $\dot{E}^n_{\Gamma_{<\beta_n}^0}$ is an open dense subset of $[\mathbf{P}^{(*\beta_n)}*_{\mathbf{P}_0^\alpha}\dot{\mathbb{R}}:\mathbf{P}_0^{\beta_n}]\cap\mathcal{M}$, we can find $(p^\frown\dot{\rho},q)\in \dot{E}^n_{\Gamma_{<\beta_n}^0}$ such that $(p^\frown\dot{\rho},q)\leq({\tilde{p}_n'}^\frown\tilde{r}_n,\tilde{q}_{n-1}*\beta_n)$. Let $({\tilde{p}_n}^\frown\tilde{\rho}_n,\tilde{q}_n)$ be a $\mathbf{P}_0^{\beta_n}$-name for $(p^\frown\rho,q)$. Now, going to $V[\Gamma_{<\beta_n}]$, since $({\tilde{p}_n}^\frown\tilde{\rho}_n,\tilde{q}_n)\in\dot{E}^n_{\Gamma_{<\beta_n}^0}\subseteq(\overline{H_{n}}^{\beta_n})^\divideontimes$, we find $\tau$ such that:
\begin{enumerate}
     \item[$(A_n)$] $({\tilde{p}_n}^\frown\tilde{\rho}_n,\tau^\frown\tilde{q}_n)\in H_{n}$.
    \item[$(B_n)$] $(\tilde{p}_n\upharpoonright\beta_n,\tau)\in\Gamma_{<\beta_n}$.
    \item[$(C_n)$] $(\tilde{p}_n,\tau^\frown\tilde{q}_n)\leq (\tilde{p}_{n-1},\tilde{\tau}_{n-1}^\frown\tilde{q}_{n-1})$.
\end{enumerate}
let $\tilde{\tau}_n$ be a $\mathbf{P}^{\beta_n}$-name for $\tau$. Now we apply the induction hypothesis as follows to find $(p_{n+1},q_{n+1})$:
\begin{enumerate}
    \item[$i_n)$] $\beta_n$ taking the place of $\beta_0$ and $\beta_{n+1}$ taking the place of $\alpha$.
    \item[$ii_n)$] $\mathbf{P}^{\beta_{n+1}}$ taking the place of $\mathbf{P}$.
    \item[$iii_n)$] $[\mathbf{P}^{(*\beta_{n+1})}*_{\mathbf{P}_0^\alpha}\dot{\mathbb{R}}:\mathbf{P}_0^{\beta_{n+1}}]$ taking the place of $\dot{\mathbb{R}}$.
    \item[$iv_n)$] $(p_n,q_n)$ taking the place of $(p_0,q_0)$.
    \item[$v_n)$] $(\tilde{p}_n\upharpoonright\beta_{n+1},\tilde{q}_n\upharpoonright\beta_{n+1})$ taking the place of $\tilde{\eta}$.
    \item[$vi_n)$] $\mathcal{H}_{\beta_{n+1}}(D,\mathbf{P},\dot{\mathbb{R}})$ taking the place of $\mathcal{H}_\alpha(D,\mathbf{P},\dot{\mathbb{R}})$.
\end{enumerate}
Therefore, we find $(p_{n+1},q_{n+1})\in\mathbf{P}^{\beta_{n+1}}$ which is $(\mathbf{P}_0^{\beta_{n+1}},\mathbf{P}^{\beta_{n+1}},\mathcal{M})$-generic for $\mathcal{H}_{\beta_{n+1}}(D,\mathbf{P},\dot{\mathbb{R}})$ and $[\mathbf{P}^{(*\beta_{n+1})}*_{\mathbf{P}_0^\alpha}\dot{\mathbb{R}}:\mathbf{P}^{\beta_{n+1}}_0]$, so in particular there is a $\mathbf{P}_0^{\beta_{n+1}}$-name $\dot{E}^{n+1}$ such that:
\begin{enumerate}
    \item[$\text{I}_n$)] $p_{n+1}\Vdash \dot{E}^{n+1}\subseteq[\mathbf{P}^{(*\beta_{n+1})}*_{\mathbf{P}_0^\alpha}\dot{\mathbb{R}}:\mathbf{P}_0^{\beta_{n+1}}]\cap\mathcal{M}$ is an open dense set.
    \item[$\text{II}_n$)] $(p_{n+1},q_{n+1})\Vdash\dot{E}^{n+1}\subseteq(\overline{H_{n+1}}^{\beta_{n+1}})^\divideontimes$.
   \end{enumerate}
and we also have,
    \begin{enumerate}
    \item[$\text{III}_n$)] $(p_n,q_n)=(p_{n+1}\upharpoonright\beta_n,q_{n+1}\upharpoonright\beta_n)$.
    \item[$\text{IV}_n$)] $p_{n+1}\Vdash ({\tilde{p}_n}^\frown\tilde{\rho}_{n},\tilde{q}_n*\beta_{n+1})\in[\mathbf{P}^{(*\beta_{n+1})}*_{\mathbf{P}_0^\alpha}\dot{\mathbb{R}}:\mathbf{P}_0^{\beta_{n+1}}]\cap\mathcal{M}$. 
    
    \item[$\text{V}_n$)] There is a $\mathbf{P}^{\beta_{n}}$-name $\tilde{\tau}_n$ such that $$(p_{n+1},q_{n+1})\Vdash (\tilde{p}_n\upharpoonright\beta_{n+1},\tilde{\tau}_n^\frown\tilde{q}_n\upharpoonright\beta_{n+1})\in\Gamma_{<\beta_{n+1}}\cap\mathcal{M}$$ and $$(p_{n+1},q_{n+1})\Vdash({\tilde{p}_n}^\frown\tilde{\rho}_n,\tilde{\tau}_n^\frown\tilde{q}_n)\in H_n\subseteq 
    D$$
\end{enumerate}
and additionally,
\begin{enumerate}
    \item[$\text{VI}_n$)] $(p_{n+1},q_{n+1})\Vdash (\tilde{p}_n,\tilde{\tau}_n^\frown \tilde{q}_n)\leq(\tilde{p}_{n-1},\tilde{\tau}_{n-1}^\frown \tilde{q}_{n-1})$.
    \item[$\text{VII}_n$)] $p_{n+1}\Vdash\tilde{r}_n=\tilde{\varphi}_2(\gamma_{n})$.
    \item[$\text{VIII}_n$)] $p_{n+1}\Vdash ({\tilde{p}_n}^\frown\tilde{\rho}_n,\tilde{q}_n)\leq({\tilde{p}_n}^\frown\tilde{r}_n,\tilde{q}_n)$.
    \item[$\text{IX}_n$)] $p_{n+1}\Vdash ({\tilde{p}_n}^\frown\tilde{\rho}_n,\tilde{q}_n)\in\dot{E}^n$.
    
\end{enumerate}
This finishes the inductive construction.
Assume the construction has been done and define
\begin{enumerate}
    \item $\overline{p}=\bigcup_{n\in\omega}p_n$.
    \item $\overline{q}=\bigcup_{n\in\omega}q_n$.
    \item $\tilde{F}_0=\{(\tilde{\rho}_k,\overline{p}):k\in \omega\}$.
    \item $\tilde{F}_1^{l}=\{(\tilde{p}_k^\frown\tilde{\rho}_k,\overline{p}):k\in\omega\land k\ge l\}$, for each $l\in\omega$.
\end{enumerate}

We claim that $(\overline{p},\overline{q})$ is $(\mathbf{P}_0,\mathbf{P},\mathcal{M})$-generic for $\mathcal{H}_\alpha(D,\mathbf{P},\dot{\mathbb{R}})$.

First, let us prove that $(\overline{p},\overline{q})\Vdash (\tilde{p}_n,\tilde{\tau}_n^\frown\tilde{q}_n)\in\dot{\Gamma}$, where $\Gamma$ is the $\mathbf{P}$-generic filter. By construction, for each $k\in\omega$,
\begin{equation*}
    (\overline{p},\overline{q})\Vdash (\tilde{p}_k\upharpoonright\beta_{k+1},\tilde{\tau}_k^\frown\tilde{q}_k\upharpoonright\beta_{k+1})\in\Gamma_{<\beta_{k+1}}
\end{equation*}
and for each $m\ge k$,
\begin{equation*}
    (\overline{p},\overline{q})\Vdash (\tilde{p}_m,\tilde{\tau}_m^\frown\tilde{q}_m)\leq(\tilde{p}_k,\tilde{\tau}_k^\frown\tilde{q}_k)
\end{equation*}
Therefore, for all $m> k$,
\begin{equation*}
    (\overline{p},\overline{q})\Vdash (\tilde{p}_k\upharpoonright\beta_{m},\tilde{\tau}_k^\frown\tilde{q}_k\upharpoonright\beta_{m})\in\Gamma_{<\beta_{m}}
\end{equation*}

Now extend $(\overline{p},\overline{q})$ to a condition $(p',q')$ which determines $(\tilde{p}_k,\tilde{q}_k)$ and $\tilde{\tau}_k$ to be $(p,q)$ and $\tau$, respectively. Then we also have that for each $m>k$
\begin{equation*}
    (p',q')\Vdash (p\upharpoonright\beta_m,\tau^\frown q\upharpoonright\beta_m)\in\dot{\Gamma}_{<\beta_m}
\end{equation*}
which implies $(p'\upharpoonright\beta_m,q'\upharpoonright\beta_m)\leq(p\upharpoonright\beta_m,\tau^\frown q\upharpoonright\beta_m)$, for each $m> k$. Finally note that $supp(p),supp(\tau^\frown q)\subseteq\bigcup_{l\in\omega}\beta_l$, since $(p,\tau^\frown q)\in\mathcal{M}$, so we can conclude that $(p',q')\leq(p,\tau^\frown q)$. Clearly this happens for any condition $(p',q')$ which determines $(\tilde{p}_k,\tilde{q}_k)$ and $\tilde{\tau}_k$, so we get $(\overline{p},\overline{q})\Vdash(\tilde{p}_k,\tilde{\tau}_k^\frown\tilde{q}_k)\in\dot{\Gamma}$.

On the other hand, by definition of $H_k$, we have $(\tilde{p}_k,\tilde{\tau}_k^\frown\tilde{q}_k)\in D_k$ (recall the construction at the beginning of this section). Therefore, $(\overline{p},\overline{q})$ is $(\mathbf{P},\mathcal{M})$-generic, and this implies that $(\overline{p},\overline{q})$ forces $\dot{\mathbb{R}}_{\Gamma_{<\alpha}^0}\cap\mathcal{M}[\Gamma_{<\alpha}^0]=\dot{\mathbb{R}}_{\Gamma_{<\alpha}^0}\cap\mathcal{M}[\Gamma]$. This will be assumed in the rest of the argument.

Let us now see that $\tilde{F}_0$ is forced by $(\overline{p},\overline{q})$ to be a dense subset of $\dot{\mathbb{R}}\cap\mathcal{M}[\dot{\Gamma}_{<\alpha}^0]$. Let $\Gamma$ be a $\mathbf{P}$-generic filter such that $(\overline{p},\overline{q})\in\Gamma$. Pick $r\in\dot{\mathbb{R}}_{\Gamma_{<\alpha}^0}\cap\mathcal{M}[\Gamma_{<\alpha}^0]$. Then there is $j\in\omega$ such that $r=\varphi_0(\gamma_j)$. Thus, by construction we have,
\begin{enumerate}
    \item $p_{j+1}\Vdash r=\tilde{r}_{j}=\tilde{\varphi}_2(\gamma_{j})$.
    \item $p_{j+1}\Vdash ({\tilde{p}_j}^\frown\tilde{\rho}_j,\tilde{q}_j)\leq({\tilde{p}_j}^\frown\tilde{r}_j,\tilde{q}_j)$.
    \item $p_{j+1}\Vdash({\tilde{p}_j}^\frown\tilde{\rho}_j,\tilde{q}_j)\in\dot{E}^j$.
\end{enumerate}
From $(\overline{p},\overline{q})\Vdash(\tilde{p}_j,\tilde{\tau}_j^\frown\tilde{q}_j)\in\Gamma$ we get $\overline{p}\Vdash\tilde{p}_j\in\Gamma_{<\alpha}^0$, so we conclude $(\tilde{\rho}_j)_{\Gamma_{<\alpha}^0}\leq(\tilde{r}_j)_{\Gamma_{<\alpha}^0}=r$. Finally note that $(\tilde{\rho}_j)_{\Gamma_{<\alpha}^0}\in(\tilde{F}_0)_{\Gamma_{<\alpha}^0}$. Thus, $(\tilde{F}_0)_{\Gamma_{<\alpha}^0}$ is a dense subset of $\dot{\mathbb{R}}_{\Gamma_{<\alpha}^0}\cap\mathcal{M}[\Gamma_{<\alpha}^0]$.

Let us prove that $(\overline{p},\overline{q})\Vdash\tilde{F}_0\subseteq (D)^\divideontimes$. Pick $r\in(\tilde{F}_0)_{\Gamma_{<\alpha}^0}$. By definition of $\tilde{F}_0$, there is $l\in \omega$ such that $r=(\tilde{\rho}_l)_{\Gamma_{<\alpha}^0}$. Then we have $(\tilde{p}_l,\tilde{\tau}_l^\frown\tilde{q}_l)\in\Gamma$, and $({\tilde{p}_l}^\frown\tilde{\rho}_l,\tilde{\tau}_l^\frown\tilde{q}_l)\in D$, from which it follows that $r=(\tilde{\rho}_l)_{\Gamma_{<\alpha}^0}\in (D)^\divideontimes_{\Gamma}$. Now we verify clause (2) from the conclusions by proving that for each $l\in\omega$, $(\tilde{F}_1^l)_{\Gamma_{<\alpha}^0}$ is predense in $[\mathbf{P}_0^\alpha*\dot{\mathbb{R}}:\mathbf{P}_0^\alpha]\cap\mathcal{M}[\Gamma]$ and $(\tilde{F}_1^l)_{\Gamma_{<\alpha}^0}\subseteq(\overline{H_l}^\alpha)^\divideontimes$. Assume $\Gamma$ is as before and pick an arbitrary condition $p^\frown\dot{r}\in[\mathbf{P}_0^\alpha*\dot{\mathbb{R}}:\mathbf{P}^\alpha_0]\cap\mathcal{M}[\Gamma]$, so we have $\dot{r}_{\Gamma_{<\alpha}^0}\in\dot{\mathbb{R}}_{\Gamma_{<\alpha}^0}\cap\mathcal{M}[\Gamma_{<\alpha}^0]$ and there is $n\ge l$ such that $(\tilde{\rho}_{n})_{\Gamma_{<\alpha}^0}\leq\dot{r}_{\Gamma_{<\alpha}^0}$, and actually, $(\tilde{p}_n^\frown\tilde{\rho}_n,\tilde{\tau}_n^\frown\tilde{q}_n)\in H_l$ (since $H_n\subseteq H_l$), which implies that $\tilde{p}_n^\frown\tilde{\rho}_n\in(\overline{H_l}^\alpha)^\divideontimes$. Since $\tilde{p}_n,p\in\Gamma_{<\alpha}^0$, we have that $\tilde{p}_n^\frown\tilde{\rho}_n$ and $p^\frown\dot{r}$ are compatible.

All it remains is to prove clause $(5)$: this follows from the facts $(\overline{p},\overline{q})\Vdash(\tilde{p}_j,\tilde{\tau}_j^\frown\tilde{q}_j)\in\dot{\Gamma}$, $(p_1,q_1)\Vdash(\tilde{p}_0,\tilde{\tau}_0^\frown\tilde{q}_0)\leq(\tilde{\eta}_0,\tilde{\nu}^\frown\tilde{\eta}_1)$ and $(\overline{p},\overline{q})\leq(p_1,q_1)$. 
\end{proof}

\begin{crl}
If $\mathbf{P}$ is a $\varsigma$-iteration of Cohen proper forcings, then $(\mathbf{P}_0,\mathbf{P})$ is Cohen proper.
\end{crl}

\section{Getting all together.}\label{getting_all_together}

Throughout this section, $\mathbb{Q}_\kappa$ is the forcing from Definition \ref{Q_k_definition}. First let us see how the forcing $\mathbb{Q}_\kappa$ fits into the Cohen proper iterations framework.

\begin{lemma}
$\mathbb{Q}_\kappa$ is a Cohen proper iteration.
\end{lemma}

\begin{proof}
We follow the same notation from section \ref{one_tukey_type_start}. It is not hard to see that any pair of the form $(\mathbb{S},\mathbb{S}\times\mathbb{P}_\kappa(\mathcal{F}))$ is Cohen proper, whenever $\mathcal{F}$ is a suitable filter.

Thus, we can define $\mathbb{Q}_0^0=\mathbb{S}$ and $\mathbb{Q}_0^1=\mathbb{S}*1$, where $1$  is the trivial forcing. Now, for $\gamma\notin D_0$, we make $\mathbb{S}_\gamma\Vdash\dot{\mathbb{Q}}_0^\gamma=\dot{\mathbb{S}}$ and $\mathbb{S}_\gamma\Vdash\dot{\mathbb{Q}}_1^\gamma=\dot{\mathbb{S}}*1$. If $\gamma\in D_0$, let $\beta\in\omega_2$ be such that $\gamma=\alpha_\beta$. Then make $\mathbb{S}_{\alpha_\beta}\Vdash\dot{\mathbb{Q}}_0^{\alpha_\beta}=\dot{\mathbb{S}}$ and $\mathbb{S}_{\alpha_\beta}\Vdash\dot{\mathbb{Q}}_1^{\alpha_\beta}=\dot{\mathbb{S}}\times\dot{\mathbb{P}}_\kappa(\dot{\mathcal{U}}_{\alpha_\beta})^{V[G_{\alpha_\beta}]}$. It is easy to see that this iteration is isomorphic to $\mathbb{Q}_\kappa$.
\end{proof}

\begin{crl}\label{Q_k_cohen_preserving}
$\mathbb{Q}_\kappa$ is Cohen preserving.
\end{crl}

The next two lemmas are easy to prove.

\begin{lemma}\label{sacks_two_step_decomposition}
Let $\mathbb{P}$ and $\mathbb{Q}$ be forcings such that $\mathbb{P}\lessdot\mathbb{Q}$. Assume $\mathbb{Q}$ is Cohen preserving. Then, $\mathbb{P}$ forces $[\mathbb{Q}:\mathbb{P}]$ is Cohen preserving.
\end{lemma}

\begin{lemma}\label{core_posets_regular_embeding}
For any $\gamma\in\omega_2$, $\mathbb{S}_{\alpha_{\gamma}}*\dot{\mathbb{P}}_{\kappa}(\dot{\mathcal{U}}_{\alpha_\gamma})^{V[\dot{G}_{\alpha_\gamma}]}$ is a regular subforcing of $\mathbb{Q}_\kappa$.
\end{lemma}

\begin{lemma}[see \cite{Wimmers}]\label{chain_condition_subfilter}
Let $\mathbb{P}$ be a forcing with the $2^\omega$-c.c. Then if $\dot{\mathcal{U}}$ is a $\mathbb{P}$-name for an ultrafilter, there is a $2^\omega$-saturated filter in the ground model which is forced to be contained in $\dot{\mathcal{U}}$.    
\end{lemma}

\begin{proof}
Define $\mathcal{F}=\{A\in[\omega]^\omega:\mathbb{P}\Vdash A\in \dot{\mathcal{U}}\}$. Assume $\mathcal{F}$ is not $2^{\omega}$-saturated and let $\{A_\alpha:\alpha\in 2^\omega\}\subseteq\mathcal{F}^+$ be such that for any different $\alpha,\beta\in 2^\omega$, $A_\alpha\cap A_\beta\in\mathcal{F}^*$. Then, for each $\alpha\in 2^\omega$, there is $p_\alpha\in\mathbb{P}$ such that $p_\alpha\Vdash A_\alpha\notin \dot{\mathcal{U}}$. Note that for different $\alpha,\beta\in 2^\omega$, we have $p_\alpha$ and $p_\beta$ are incompatible: assume otherwise and let $q$ be a common extension, so $q\Vdash A_\alpha\in\dot{\mathcal{U}}$ and $q\Vdash A_\beta\in\dot{\mathcal{U}}$, which implies $q\Vdash A_\alpha\cap A_\beta\in\dot{\mathcal{U}}$. On the other hand, we also have $\omega\setminus(A_\alpha\cap A_\beta)\in\mathcal{F}$, which means $\mathbb{P}\Vdash \omega\setminus(A_\alpha\cap A_\beta)\in\dot{\mathcal{U}}$, which implies $q\Vdash \omega\setminus(A_\alpha\cap A_\beta)\in\dot{\mathcal{U}}$, which is a contradiction. Therefore, $\mathcal{F}$ is $2^\omega$-saturated, and by definition is a subset of $\dot{\mathcal{U}}$ in the forcing extension by $\mathbb{P}$.
\end{proof}

\begin{lemma}\label{saturated_subfilter}
Let $G$ be a $\mathbb{S}_{\omega_2}$-generic filter over $V$. Let $\dot{\mathcal{U}}$ be a $[\mathbb{Q}_{\kappa}:\mathbb{S}_{\omega_2}]$-name for an ultrafilter. Then there is a $2^\omega$-saturated filter $\mathcal{F}$(in $V[G]$) such that $[\mathbb{Q}_{\kappa}:\mathbb{S}_{\omega_2}]\Vdash \mathcal{F}\subseteq\dot{\mathcal{U}}$.
\end{lemma}

\begin{proof}
Since $\mathbb{Q}_\kappa$ has the $\omega_2$-c.c and $G\subseteq\mathbb{S}_{\omega_2}$ is generic over $V$, then, in $V[G]$ we have that $[\mathbb{Q}_{\kappa}:\mathbb{S}_{\omega_2}]$ has the $\omega_2$-c.c. and  $2^\omega=\omega_2$. Then, by Lemma \ref{chain_condition_subfilter}, we have that in $V[G]$, there is a $2^\omega$-saturated filter $\mathcal{F}$ which is forced to be a subfilter of $\dot{\mathcal{U}}$.
\end{proof}

\begin{thm}\label{main_theorem}
Assume $V$ is a model of $\mathsf{ZFC}+\mathsf{GCH}+\diamondsuit(S)$, where $S=\{\alpha\in\omega_2:\mathsf{cof}(\alpha)=\omega_1\}$. Let $\kappa\ge \omega_2$ be a regular cardinal. Then there is a forcing extension which satisfies the following:
\begin{enumerate}
    \item $2^\omega=\kappa$.
    \item All ultrafilters are Tukey top.
    \item There is no $\mathsf{nwd}$-ultrafilter.
\end{enumerate}

\end{thm}

\begin{proof}
Let $\mathbb{Q}_\kappa$ the forcing from Definition \ref{Q_k_definition} and $G$ a generic filter over $V$. Let $H=G\upharpoonright\mathbb{S}_{\omega_2}=\{p\in\mathbb{S}_{\omega_2}:(\exists(p_0,q_0)\in G)(p_0=p)\}$ be the $\mathbb{S}_{\omega_2}$-generic filter over $V$, and for each $\alpha\in\omega_2$, $H_\alpha$ the restriction of $H$ to $\mathbb{S}_\alpha$. Also, by Lemma \ref{core_posets_regular_embeding}, $\mathbb{S}_{\alpha_\gamma}*\dot{\mathbb{P}}_{\kappa}(\dot{\mathcal{U}}_{\alpha_\gamma})^{V[\dot{G}_{\alpha_\gamma}]}$ is a regular suborder of $\mathbb{Q}_\kappa$, so let $F_{\alpha_\gamma}$ be the projection of $G$ to $\mathbb{S}_{\alpha_\gamma}*\dot{\mathbb{P}}_{\kappa}(\dot{\mathcal{U}}_{\alpha_\gamma})^{V[\dot{G}_{\alpha_\gamma}]}$. To short the notation we write $\mathbb{Q}/F_{\alpha_\gamma}$ instead of $[\mathbb{Q}_\kappa:\mathbb{S}_{\alpha_\gamma}*\dot{\mathbb{P}}_\kappa(\dot{\mathcal{U}}_{\alpha_\gamma})]$, and $\mathbb{Q}_\kappa/H$ instead of $[\mathbb{Q}_\kappa:\mathbb{S}_{\beta_2}]$.

Note that by Corollary \ref{Q_k_cohen_preserving} and Lemma \ref{sacks_two_step_decomposition}, for each $\gamma\in\omega_2$, $\mathbb{Q}_\kappa/F_{\alpha_\gamma}$ is Cohen preserving over $V[F_{\alpha_{\gamma}}]$. Also note that $\mathbb{Q}_\kappa$ is forcing equivalent to $\mathbb{S}_{\alpha_\gamma}*\dot{\mathbb{P}}_\kappa(\dot{\mathcal{U}}_{\alpha_\gamma})^{V[\dot{G}_{\alpha_\gamma}]}*(\mathbb{Q}_\kappa/\tilde{F}_{\alpha_\gamma})$, so by Corollary \ref{two_steps_sacks_prop_Tukey_above}, in the forcing extension by $\mathbb{Q}_\kappa$, we have that any ultrafilter extending $\dot{\mathcal{U}}_{\alpha_\gamma}$ is Tukey top.

We work in $V[H]$. Let $\dot{\mathcal{U}}$ be a $(\mathbb{Q}_\kappa/H)$-name for an ultrafilter on $\omega$. Since $\mathbb{Q}_\kappa/H$ is $\omega_2$-c.c., by Lemma \ref{saturated_subfilter}, there is a $2^\omega$-saturated filter $\mathcal{F}$ which is forced, by $\mathbb{Q}_\kappa/H$, to be contained in $\dot{\mathcal{U}}$. By Lemma \ref{suitable_reflection}, there is an $\omega_1$-club subset of $\omega_2$ on which $\mathcal{F}$ reflects as a suitable filter. Since $\langle A_\alpha:\alpha\in S\rangle$ is a $\diamondsuit(S)$-guessing sequence, there is $\delta\in S$ such that $A_\delta$ codifies $\mathcal{F}\cap V[H_\delta]$ and $\mathcal{F}\cap V[H_\delta]$ is a suitable filter. Let $\gamma$ be such that $\delta=\alpha_\gamma$. Then we have that $\dot{\mathbb{P}}_\kappa(\dot{\mathcal{U}}_{\alpha_\gamma})^{V[H_{\alpha_\gamma}]}=\dot{\mathbb{P}}_\kappa(\mathcal{F}\cap V[H_{\alpha_\gamma}])^{V[H_{\alpha_\gamma}]}$.  Note that $\mathcal{U}_{\alpha_\gamma}\subseteq\mathcal{F}\subseteq\mathcal{U}$, so, by the previous paragraph, we have that $\mathcal{U}$ is Tukey top.

Alternatively, without the use of the first paragraph, we can argue as follows, which gives a more pictorial representation. Let us now go to the forcing extension V[G]. By the previous paragraph, we have $\mathcal{U}_{\alpha_\gamma}\subseteq\mathcal{F}\subseteq\mathcal{U}$. Let $\varphi_{\alpha_\gamma}:\omega\to 2^\kappa$ be the Kat\v{e}tov function introduced by $\dot{\mathbb{P}}_{\kappa}(\dot{\mathcal{U}}_{\alpha_\gamma})$. For each $\alpha\in\kappa$ and $i\in 2$, let $B^\alpha_i=\{x\in 2^\kappa:x(\alpha)=i\}$, and define $D^\alpha_i={\varphi}_{\alpha_\gamma}^{-1}[B_\alpha^i]$. Then we have that for each $\vec{\nu}\in\mathcal{C}_\kappa\cap V[H_{\alpha_\gamma}]$, ${\varphi}^{-1}_{\alpha_\gamma}[{X}_{\vec{\nu}}]\in{\mathcal{U}}_{\alpha_\gamma}^*$. In particular, we have that for each $\vec{\nu}\in\mathcal{C}_\kappa\cap V$, ${\varphi}^{-1}_{\alpha_\gamma}[{X}_{\vec{\nu}}]\in{\mathcal{U}}_{\alpha_\gamma}^*$. 

Let $x\subseteq\kappa$ be an infinite subset having order type $\omega$ and $f:x\to 2$. Since $\mathbb{Q}_\kappa$ is Cohen preserving, by Lemma \ref{cohen_preserving_kappa}  we find $\vec{\nu}\in\mathcal{C}_\kappa\cap V$ such that $\bigcap_{\alpha\in{x}} {B}^{\alpha}_{{f}(\alpha)}\subseteq{X}_{\vec{\nu}_0}$. By the last remark of the previous paragraph, ${\varphi}^{-1}_{\alpha_\gamma}[{X}_{\vec{\nu}_0}]\in{\mathcal{U}}_{\alpha_\gamma}^*$, which implies ${\varphi}^{-1}_{\alpha_\gamma}[\bigcap_{\alpha\in{x}} B^{\alpha}_{{f}(\alpha)}]\in{\mathcal{U}}_{\alpha_\gamma}^*$. Thus, $\bigcap_{\alpha\in{x}}{D}^\alpha_{{f}(\alpha)}\in{\mathcal{U}}_{\alpha_\gamma}^*$. Note that this implies that any ultrafilter extending $\mathcal{U}_{\alpha_\gamma}$ is Tukey top. Since $\mathcal{U}_{\alpha_\gamma}=\mathcal{F}\cap V[H_{\alpha_\gamma}]\subseteq\mathcal{F}\subseteq\mathcal{U}$, we have that $\mathcal{U}$ is Tukey top.
\end{proof}

\section{Squares and products}\label{squares_and_products}

In this section we present a different approach based on the existence of square
sequences. Note that, working in the Sacks model, the quotient forcing $[\mathbb{Q}_\kappa : \mathbb{S}_\kappa]$ is
naturally equivalent to a subposet of the countable support product $\bigotimes_{\beta\in S}\dot{\mathbb{P}}_\kappa(\dot{\mathcal{U}}_\beta)$.
In this section we provide a construction of a reduced product which keeps all the
relevant properties of $\mathbb{Q}_\kappa$, although it looks quite different in what the support of
the conditions refers. This construction is strongly based on Theorem \ref{thm_reduced_product} below, which is interesting by itself. The ingenious proof of Theorem \ref{thm_reduced_product} is due to the work of the second author. First we deal with the proof of this theorem and then we show how to apply it to construct a model where there all ultrafilters are Tukey top.

\begin{thm}\label{thm_reduced_product}
Suppose that $n\in\gw$ is a natural number and $\Box_m$ holds for all $m<n$ and $\cof(meager)=\aleph_1$ holds. Let $\langle \mathbb{P}_\ga\colon\ga\in\gw_n\rangle$ be a sequence of $\sigma$-proper, Cohen preserving posets of cardinality $\aleph_1$. Then there is a poset $\mathbf{P}$ which:

\begin{enumerate}
\item  preserves $\aleph_1$;
\item $\mathbf{P}$ is Cohen-preserving;
\item has $\aleph_2$-c.c.;
\item adds a generic filter for all $\mathbb{P}_\ga$ for $\ga\in\gw_n$.
\end{enumerate}
\end{thm}

\noindent The poset $\mathbf{P}$ is a suborder of the countable support product $\bigotimes_{\alpha\in\omega_n} \mathbb{P}_\ga$, and it is of cardinality $\max(\aleph_1,\aleph_n)$. For $n=0$ or $1$ the square assumption is either meaningless or provable in ZFC, so in such a circumstance the reduced product exists in ZFC plus $\cof(meager)=\aleph_1$. It is not clear whether the theorem can be generalized past $\aleph_\gw$.

The construction of the reduced product relies on two tools. The first one is a well-known consequence of the square assumption. To describe the tool properly, write $\kappa=\gw_n$, and for every countable set $a\subseteq\omega$, define $s(a)=\langle sup(a\cap\omega_m):m\leq n\rangle$ for each $m\leq n$.

\begin{lemma}\label{noetherian_stationary_set}
There is a set $S\subset [\omega_n]^{\omega}$ which is

\begin{enumerate}
\item stationary;
\item closed under intersections;
\item contains all singletons;
\item the function $s\restriction S$ is an injection.
\item $(S,\subseteq)$ is a wellfounded order.
\item For any countable $a\in S$, the set $\{b\in S:b\subseteq a\}$ is countable.
\item For any $a\in S$ and any finite $F\subseteq S\cap\mathcal{P}(a)$ such that $a\notin F$, $a\setminus \bigcup F\neq\emptyset$.
\end{enumerate}
\end{lemma}

\begin{proof}
First let us define some objects that will be useful. For each $\alpha\in\omega_n$, let $f:\vert\alpha\vert\to\alpha$ be a bijection. Let $\langle C_\alpha:\alpha\in \lim(\omega_{n})\rangle$ be a simultaneously square sequence for all $\omega_m$ for each $m< n$, that is for each limit ordinal $\alpha<\omega_n$, the following holds:
\begin{enumerate}
    \item $C_\alpha$ is a club subset of $\alpha$.
    \item If $\alpha\in [\omega_m,\omega_{m+1})$ for $m>0$ and $cof(\alpha)<\omega_m$, then $ot(C_\alpha)<\omega_m$.
    \item For each $\alpha\in \lim(\omega_n)$, if $\beta$ is a limit point of $C_\alpha$, then $C_\beta=C_\alpha\cap\beta$.
\end{enumerate}

For each $\alpha\in \lim(\omega_n)$, let $g_\alpha:ot(C_\alpha)\to C_\alpha$ be the increasing enumeration of $C_\alpha$. Then define $S$ to be the family of all countable sets $a\subset\omega_n$ such that $\omega\cup\{\omega_m:m< n\}\subseteq a$, and $a$ is closed under the successor function, the functions $f_\alpha$ and $f_\alpha^{-1}$ for each $\alpha \in a$, $g_\alpha$ and $g_\alpha^{-1}$ for each limit ordinal $\alpha\in a\cup rng(s(a))$, and additionally, for each limit ordinal $\alpha\in a\cup rng(s(a))$, $C_\alpha\cap a$ is cofinal in $a\cap \alpha$.

We prove that $S$ satisfies the lemma in a sequence of claims.

\begin{claim} If $a\in S$ and $\gamma$ is a limit point of $a$, then $C_\gamma\cap a$ is cofinal in $\gamma\cap a$ and $a$ is closed under $g_\gamma$.
\end{claim}

\begin{proof}
Let $\gamma$ be a limit ordinal of $a$. If $\gamma\in a\cup rng(s(a))$, it follows directly from the definition. So let us assume $\gamma\notin a\cup rng(s(a))$. Note that this implies that $\alpha=\min\{\delta\in a:\gamma<\delta\}$ is well defined and is a limit ordinal. Thus, $C_\alpha\cap \alpha$ is cofinal in $a\cap \alpha=a\cap \gamma$, which implies that $\gamma$ is a limit point of $C_\alpha$. By the coherence property of the square sequence, we have $C_\alpha\cap\gamma= C_\gamma$, so $C_\gamma$ is cofinal in $a\cap \gamma$. Also note that since $a$ is closed under $g_\alpha$ and $C_\gamma$ is an initial segment of $C_\alpha$, $a$ should be closed under $g_\gamma$.
\end{proof} 

\begin{claim} If $a, b\in S$ and $\gamma\in\omega_n$ is a limit point of both $a$ and $b$, then $\gamma$ is a limit point of $a\cap b$.
\end{claim}

\begin{proof}
The proof is by induction on $\gamma$. Let $\gamma$ be a limit point of both $a$ and $b$ and assume the claim is true for all $\delta <\gamma$ which are limit points of both $a$ and $b$. By the previous claim we have that $C_\gamma\cap a$ is cofinal in $a\cap\gamma$, and $C_\gamma\cap b$ is cofinal in $b\cap\gamma$. Also, both $a$ and $b$ are closed under $g_\gamma^{-1}:C_\gamma\to ot(C_\gamma)$. Let $\alpha=ot(C_\gamma)$. Since $a$ and $b$ are closed under $g_\gamma^{-1}$, we have that $\alpha$ is a limit point of both $a$ and $b$, and since $\alpha<\gamma$, by our induction hypothesis, $\alpha$ is a limit point of $a\cap b$, that is $a\cap b\cap \alpha$ is cofinal in $\alpha$. Since $a$ and $b$ are closed under $g_\gamma$, $a\cap b$ is closed under $g_\gamma$, so $g_\gamma[a\cap b\cap\alpha]\subseteq a\cap b$, but since $a\cap b$ is cofinal in $\alpha$, we have that $g_{\gamma}[a\cap b\cap\alpha]$ is cofinal in $\gamma$, so $\gamma$ is a limit point of $a\cap b$.
\end{proof}

\begin{claim} If $a, b\in S$ are countable sets such that $s(a)(m)=s(b)(m)$ for each $m\leq n$, then $a=b$.
\end{claim}

\begin{proof}
Note that the previous claim implies that $a\cap b$ is cofinal in both $a$ and $b$, so it is sufficient to prove that for each $\gamma\in a\cap b$ which is a limit ordinal of $a\cap b$, $a\cap\gamma=b\cap \gamma$ holds; we prove this by induction on $\gamma$. First note that by definition $\omega\subseteq a\cap b$, so we have $a\cap\omega=b\cap \omega$ and $\omega\in a\cap b$. Now, assume $\gamma=\beta+1$, $\beta\in a\cap b$ and it holds that $a\cap\beta=b\cap\beta$. Since both $a$ and $b$ are closed under the successor function, the induction hypothesis implies $\gamma\in a\cap b$, so $a\cap\gamma=b\cap \gamma$ follows. If $\gamma\in a \cap b$ is a limit ordinal of $a\cap b$, we have $a\cap\gamma=\bigcup_{\delta\in a\cap b\cap\gamma}a\cap\delta=\bigcup_{\delta\in a\cap b\cap\gamma}b\cap\delta=b\cap\gamma$. If $\gamma=\omega_m$ for some $m<n$, by hypothesis we have $s(a)(m)=s(b)(m)$, so our induction hypothesis implies again that $a\cap\omega_m=a\cap s(a)(m)=b\cap s(b)(m)=b\cap \omega_m$. Finally, if $\gamma\in a\cap b$ is a limit ordinal and $\gamma\notin \{\omega_m:m<n\}$, and $\gamma$ is not a limit point of $a\cap b$, then we have $a\cap\gamma= f_\gamma[a\cap\vert\gamma\vert]=f_\gamma[b\cap\vert\gamma\vert]=b\cap\gamma$.
\end{proof}

\begin{claim} For any $a,b\in S$, $a\cap b\in S$.
\end{claim}

\begin{proof}
Note that it is sufficient to prove that for any limit point $\gamma$ of $ a\cap b$, $C_\gamma\cap a\cap b$ is cofinal in $a\cap b\cap \gamma$. If $\gamma$ is a limit point of $a\cap b$, then it is a limit point of both $a$ and $b$, so they are closed under $g_{\gamma}$ and $g_\gamma^{-1}$. Note that this implies that $ot(C_\gamma)$ is a limit point of both $a$ and $b$, so $ot(C_\gamma)$ is a limit point of $a\cap b$. Since $a\cap b$ is closed under $g_\gamma$, and $a\cap b\cap ot(C_\gamma)$ is cofinal in $ot(C_\gamma)$, we have that $g_{\gamma}[a\cap b\cap ot(C_\gamma)]$ is cofinal in $\gamma$, but $g_\gamma[a\cap b]\subseteq a\cap b$, so $C_\gamma\cap a\cap b$ is cofinal in $a\cap b\cap \gamma$. 
\end{proof}

\begin{claim}\label{claim_87}
The set $S$ is stationary.
\end{claim}

\begin{proof}
Fix $h:\omega_n^{<\omega}\to\omega_n$. Let $\mathcal{M}\prec H(\omega_{n+2})$ be a countable elementary submodel such that containing all the following objects:
\begin{enumerate}
    \item $h$.
    \item $\{\omega_0,\ldots,\omega_{n}\}$.
    \item $\langle C_{\alpha}:\alpha\in\lim(\omega_{n+1})\rangle$.
    \item $\{f_\alpha:\alpha\in\omega_n\}$.
    \item $\{g_{\alpha}:\alpha\in \lim(\omega_{n+1})\}$.
\end{enumerate}
Let $a=\omega_n\cap\mathcal{M}$. We claim that $a\in S$. It is clear that for each $\alpha\in a$, $a$ is closed under $f_\alpha$ and $f_{\alpha}^{-1}$, and the same holds for $g_{\alpha}, g_{\alpha}^{-1}$ for any limit ordinal $\alpha\in a$. It is also clear that $a$ is closed under $h$. So it is sufficient to prove that for any limit ordinal $\gamma\in a\cup rng(s(a))$, $a\cap C_\gamma$ is cofinal in $a\cap \gamma$. If $\gamma\in a$ it follows from the fact that $C_\gamma\in \mathcal{M}$ and $\mathcal{M}$ knows that $C_\gamma$ is cofinal in $\gamma$. So let us assume that $\gamma=s(a)(m)$ for some $m\leq n$. By definition $s(a)(m)=\sup(a\cap\omega_m)$. We also have $\omega_m\in\mathcal{M}$, so $C_{\omega_m}\in\mathcal{M}$ and $C_{\omega_m}\cap\mathcal{M}$ is cofinal in $\omega_m\cap\mathcal{M}$. This implies that $s(a)(m)$ is a limit point of $C_{\omega_m}$, and by the coherence property, $C_{s(a)(m)}=C_{\omega_m}\cap s(a)({m})$. Thus, we have $C_{s(a)(m)}$ is cofinal in $a\cap s(a)({m})$.
\end{proof}

\begin{claim} $(S,\subseteq)$ is wellfounded.
\end{claim}
\begin{proof}
Let $\{a_k:k\in\omega\}\subseteq S$ be a $\subseteq$-decreasing sequence. For each $m\leq n$, we have that $\langle s(a_k)({m}):k\in\omega\rangle$ is a decreasing sequence of ordinals, so there is $j_m\in\omega$ such that $\langle s(a_k)({m}):k\ge j_m\rangle$ is constant. Let $j=\max\{j_0,\ldots,j_n\}$. Then the sequence $\langle\langle s(a_k)({m}):m\leq n\rangle:k\ge j\rangle$ is constant. Now, an application of Claim \ref{claim_87} shows that for all $k,k'\ge j$, $a_k=a_{k'}$.
\end{proof}

\begin{claim} For any $a\in S$, $\{b\in S:b\subseteq a\}$ is countable.
\end{claim}

\begin{proof}
For any $b\in S$, $b$ is uniquely determined by $s(b)$, and such $b$, has countably many limit points, so for any $a\in S$ such that $a\subseteq b$, $s(a)$ can take countably many values, which implies that $a$ can take countably many different values at most.
\end{proof}

The proof of the lemma is complete.
\end{proof}

Now we continue with the proof of Theorem \ref{thm_reduced_product}. Note that we can assume that each $\mathbb{P}_\alpha$ has $\omega_1$ as domain, $0$ is the maximum of $(\mathbb{P}_\alpha,\leq_\alpha)$, and for each limit ordinal $\delta\in\omega_1$, if $\dot{\tau}$ is a $(\mathbb{P}_\alpha\cap\delta,\leq_{\alpha}\upharpoonright\delta\times\delta)$-name for an open dense subset of $\omega^{<\omega}$, then for any condition $p\in \mathbb{P}_\alpha\cap\delta$ there is $q\in \mathbb{P}_\alpha$ such that $q\leq p$ and
\begin{equation*}
    \{t\in\omega^{<\omega}:q\Vdash \check{t}\in\dot{\tau}\}
\end{equation*}
 is open dense.

Let $K=\langle (f_\alpha,D_\alpha):\alpha\in\omega_1\rangle$ be a sequence such that for each $\alpha\in\omega_1$, $f:\alpha\to\omega$ is a bijection, and for any open dense subset $D\subseteq\omega^{<\omega}$ there is $\alpha\in\omega_1$ such that $D_\alpha\subseteq D$. This sequence will remain fixed for the rest of the section.

We are coming to the second tool needed for the construction of the reduced product. It is related to a deep theorem of Woodin: Theorem 7 from \cite{Woodin1996}. This theorem states that whenever $\vec{B}$ is an $\kappa$-sequence of ordinals, with $\kappa\ge\omega_1$, then $L[\vec{B}]$ is a generalized Prikry forcing extension of the relativized core model $K_A$, for some $A\subseteq\kappa$, and satisfies $\mathsf{GCH}$ for cardinals $\lambda\ge\kappa$. We need the particular case for $\vec{S}$ being an $\omega_1$-sequence of ordinals, so $L[\vec{B}]$ is a model of $\mathsf{ZFC}+\mathsf{GCH}$. Below we state the portion of the theorem we need; the complete statement of the theorem can be found in \cite{Woodin1996}.

\begin{thm}[W. H. Woodin, see \cite{Woodin1996}]\label{woodin_theorem}
If $V=L[\vec{B}]$ and $\vec{B}$ is an $\omega_1$-sequence of ordinals, then  $\mathsf{GCH}$ holds.
\end{thm}

To state the next lemma, for every countable set $b\subset\kappa$ write $\cl(b)$ for the inclusion-smallest element of $S$ containing all ordinals in $b$ and all limit points of $b$ as well. This set exists by the well-foundedness of $S$ and its closure under intersections.

\begin{lemma}\label{models_system}
There is a system $\langle M_a, \prec_a\colon a\in S\rangle$ such that for each $a\in S$, $M_a$ is a transitive model of ZFC containing the following:

\begin{enumerate}
\item $S\cap\power(\cl(a))$;
\item $K$ and the function $\langle\langle \ga, \mathbb{P}_\ga\rangle\colon\ga\in a\rangle$;
\item $\prec_a$, which in $M_a$ is a well-ordering of $\gw^\gw$ of order type $\gw_1$;
\item the function $\{\langle b, (M_b\cap V_{\omega_n+5}, \prec_b)\rangle\colon b\subseteq a, b\in S\}$.
\end{enumerate}
\end{lemma}

\begin{proof}
The model $M_a$ is $L[K][S\cap\mathcal{P}(\cl(a))][\langle\langle\ga, \mathbb{P}_\ga\rangle\colon \ga\in a\rangle]$. One point here is that this is a model of ZFC, since the set $S\cap\mathcal{P}(\cl(a))$ is injectively mapped into $Ord^{n+1}$ by the function $s$ and thereby canonically well-ordered. Items (1, 2) are obvious. For item (3), note that $S\cap\mathcal{P}(cl(a))$ is a countable set, so $M_a=L[\vec{B}_a]$ for some $\gw_1$-sequence $\vec{B}_a$ of ordinals and it computes $\aleph_1$ correctly by the choice of the sequence $K$. It follows from \ref{woodin_theorem} that $M_a$ is a model of $\mathsf{CH}$. In the model $M_a$, let $\prec_a$ be the first well-ordering of the reals of ordertype $\gw_1$ in the natural constructibility well-order. For the last item, note that if $b\subseteq a$ are sets in $S$, then $M_a$ can compute $\cl(b)$ as the inclusion-smallest set in $S\cap\mathcal{P}(\cl(a))$ containing all points of $b$ and all limit points of $b$.
\end{proof}

\noindent We are now ready to define the reduced product poset.

\begin{dfn}\label{reduced_product_definition}
Under the assumptions of Theorem \ref{thm_reduced_product} and Lemma \ref{models_system}, we define a reduced product $\mathbf{P}$ as follows:  for a condition $p\in\bigotimes_{\alpha\in\omega_n}\mathbb{P}_\alpha$, declare $p\in \mathbf{P}$ if and only if the following conditions hold:
\begin{enumerate}
    \item $dom(p)\subseteq\kappa$ is a countable set and a boolean combination of elements of $S$.
    \item For every $\ga\in dom(p)$, $p(\ga)\in \mathbb{P}_\ga$.
    \item For every set $b\in S$, the function $p\restriction b$ belongs to $M_b$.
\end{enumerate}
The ordering is by coordinatewise strengthening.
\end{dfn}

\noindent It is clear that each condition $p\in \mathbf{P}$ can be extended to one whose domain is in $S$, but more complicated domains will be notationally useful below. Since the set $S$ is closed under intersections, the following claim provides a practical criterion for membership in the reduced product.

\begin{claim}
For a function $p$ in the full countable support product, the following are equivalent:

\begin{enumerate}
\item $p\in \mathbf{P}$;
\item $dom(p)$ is a boolean combination of elements of $S$, and there exists $a\in S$ such that $dom(p)\subseteq a$ and for every $b\subseteq a$ in the set $S$, $p\restriction b\in M_b$.
\end{enumerate}
\end{claim}

\begin{dfn}
Let $d\subseteq \kappa$ be a set which is a boolean combination of elements of $S$. Define $P_d$ as the poset of all conditions in $P$ whose domain is a subset of $d$.
\end{dfn}

\begin{lemma}
For each $d$ which is a boolean combination of elements of $S$, $\mathbf{P}_d$ is a regular subposet of $\mathbf{P}$.
\end{lemma}

\noindent Whenever $d\subset\kappa$ is a set which is a boolean combination of elements of $S$, the poset $\mathbf{P}_d$ of all conditions $p\in P$ whose domain is a subset of $d$ is a regular subposet of $\mathbf{P}$.

\begin{claim}
\label{claim1}
For every set $d\subset\kappa$ which is a boolean combination of elements of $S$, $\mathbf{P}_d$ is a regular subposet of $\mathbf{P}$.
\end{claim}

\begin{proof}
Fix one of such sets $d$ and let $p\in \mathbf{P}$ be a condition. Then $p\restriction d$ is a condition in $\mathbf{P}$ and in $\mathbf{P}_d$ as well. It is enough to show that any condition $q\in \mathbf{P}_d$ such that $q\leq p\restriction d$, $q$ and $p$ are compatible. Let $e=\kappa\setminus d$ and define $r=q\cup p\restriction e$. It is enough to prove that $r\in \mathbf{P}$, so we have to prove that for each $b\in S$, $r\restriction b\in M_b$ holds. This is clear though:
on the one side, $r\restriction d\cap b=q\restriction b\in M_b$ holds since $q\in \mathbf{P}$; on the other side, $r\restriction e\cap b=p\restriction e\cap b\in M_b$ holds since $p\in \mathbf{P}$. Thus, $r\upharpoonright b=(q\upharpoonright b)\cup(p\upharpoonright e\cap b)\in M_b$.
\end{proof}

\noindent The following is immediate from the definitions.

\begin{claim}
Let $b\in S$ be any set. The sequence $\langle \mathbf{P}_c\colon c\in S\cap\mathcal{P}(b)\rangle$ belongs to $M_b$.
\end{claim}

\noindent Towards the proof that the reduced product preserves $\aleph_1$, consider a countable elementary submodel $N\prec H(\theta)$ such that $\mathbf{P}, \langle\mathbf{P}_c:c\in S\rangle, S\in N$ and $N\cap\omega_n\in S$. It is necessary to investigate several subposets of $\mathbf{P}\cap N$ and of $\mathbf{P}$. Let $b\subseteq N\cap\omega_n$ be an element of $S$, let $\bar a$ be a finite tuple of subsets of $N\cap\omega_n$ which are in addition in $S$ (none of these need to be in $N$, and that is one of the key difficulties). Write $\mathbf{P}_{b, \bar a}$ for the poset $\mathbf{P}_{b\setminus\bigcup \bar a}$ and $\mathbf{Q}_{b, \bar a}$ for the poset $\mathbf{P}_{b, \bar a}\cap N$. These posets satisfy natural factorization properties:

\begin{claim}
If $c\subseteq b$ is in $S$ then 

\begin{enumerate}
\item $\mathbf{Q}_{b, \bar a}$ is naturally isomorphic to the product $\mathbf{Q}_{c, \bar a}\times\mathbf{Q}_{b, \bar a^\smallfrown c}$;
\item  $\mathbf{P}_{b, \bar a}$ is naturally isomorphic to the product $\mathbf{P}_{c, \bar a}\times \mathbf{P}_{b, \bar a^\smallfrown c}$.
\end{enumerate}
\end{claim}

\begin{proof}
It is sufficient to prove the first item, as the second follows from Claim~\ref{claim1}. Consider any $p\in \mathbf{Q}_{b, \bar a}$ and its two projections $p_0=p\restriction c$ and $p_1=p\setminus p_0$. The main point is that even though $c\in N$ does not have to hold, by the local countability assumption the condition $p_0$ does belong to $N$, and so does $p_1$. It is not difficult to check that if $p_0'\leq p_0$ and $p_1'\leq p_1$ are strengthenings in the appropriate posets, then $p_0'\cup p_1'\in \mathbf{Q}_{b, \bar a}$ is a condition stronger than $p$.
\end{proof}

\begin{claim}
The model $M_b$ contains the posets $\mathbf{P}_b\cap N$ and $\mathbf{Q}_{b, \bar a}$.
\end{claim}

\begin{proof}
One after another, we prove that the following objects belong to the
model $M_b$:
\begin{itemize}
    \item $S\cap\mathcal{P}(\cl(b))$. This is by the choice of the model $M_b$ (Lemma \ref{models_system}(1));
    \item $N\cap\cl(b)$. Since $N\cap\omega_n\in S$ and $S$ is closed under intersections, we have $N\cap \cl(b)\in S\cap\mathcal{P}(\cl(b))$. The previous cluase and transitivity of $M_b$ imply $N\cap\cl(b)\in M_b$;
    \item $N\cap S\cap\mathcal{P}(\cl(b))$. By elementarity of the model $N$, this set is equal to $\{c\in S\cap\power(b)\colon rng(s(c))\subseteq N\cap\cl(b)\}$, which is an element of $M_b$, by the previous item;
    \item a sequence $\langle \prec'_c:c\in S\cap\mathcal{P}(b)\rangle$ such that for every $c\in S\cap\mathcal{P}(b)$, $\prec_c'\in M_c$ is a well ordering of the poset $\mathbf{P}_c$ of order type $\omega_1$. To find such a system in the model $M_b$, just in each model $M_c$ define $\prec_c'$ from $\prec_c$ and the collapse maps on the sequence $K$. Note that each model $M_c$ satisfies the Continuum Hypothesis by Lemma \ref{models_system}(3). One can easily produce such a definition which is uniform in $c$, and then the sequence $\langle \prec'_c:c\in S\cap\mathcal{P}(b)\rangle$ will belong to $M_b$;
    \item $N\cap \mathbf{P}_b$. To see this, in $M_b$ observe that for a condition $p \in \mathbf{P}_b$, $p \in N$ if and only if there is $c \in N \cap S \cap \mathcal{P}(b)$ such that $dom(p) \subseteq c$ and $p$ is $\delta$-th element of $\mathbf{P}_c$ in the $\prec_c'$-order, for some $\delta \in N \cap \omega_1$. This provides a definition of the set $N \cap \mathbf{P}_b$ in the model $M_b$;
    \item $\mathbf{Q}_{b,\bar{a}}$. This poset is obtained from $\mathbf{P}_b \cap N$ by restriction of the conditions in $\mathbf{P}_b \cap N$ to $b\setminus\bigcup\bar{a}$.

\end{itemize}
This completes the proof of the claim.

\end{proof}

\noindent The following is a key claim justifying the whole construction.

\begin{claim}
\label{keyclaim}
For every $b, \bar a$ in $S$ which are subsets of $N\cap\omega_n$ and such that $b\setminus\bigcup\bar{a}\neq\emptyset$, for every condition $q\in \mathbf{Q}_{b, \bar a}$ and for every $\mathbf{Q}_{b, \bar a}$-name $\dot C$ for an open dense subset of $\omega^{<\omega}$ there is a condition $p\in \mathbf{P}_{b, \bar a}$ and an open dense subset $D\subset\omega^{<\omega}$ such that $p\leq q$ and $p\Vdash\check D\subset\dot C$.
\end{claim}

\begin{proof}
This is proved by transfinite induction on rank of $b$ in the inclusion ordering of $S$. If $b$ is a singleton, this is part of the initial assumptions on the posets $\mathbb{P}_\ga$. Now, suppose that $b\in S$ is arbitrary and $b\subseteq N\cap\omega_n$, and the statement has been proved for all sets in $S$ of lower rank. Let $\bar a$ be a finite tuple of sets in $S$ such that $\bar a\subseteq N\cap \omega_n$. We will show that for every $\mathbf{Q}_{b,\bar a}$-name $\dot C$ for an open dense subset of $\omega^{<\omega}$ and $q\in \mathbf{Q}_{b,\bar a}$, there is a condition $p\leq q$ and a dense subset $D\subset\omega^{<\omega}$ such that $p\Vdash\check D\subset\dot C$.

Fix a condition $q\in \mathbf{Q}_{b, \bar a}$ and let $\dot C$ be a $\mathbf{Q}_{b, \bar a}$-name for a dense open subset of $\omega^{<\omega}$.
Define $E(\dot{C})=\{\langle r, t\rangle\colon r\in \mathbf{Q}_{b, \bar a}\land r\Vdash t\in\dot C\}$, so $E({\dot{C}})\subseteq \mathbf{Q}_{b, \bar a}\times\omega^{<\omega}$ is open dense, and it has an open dense subset $D_0\in M_b$ by the initial assumptions on $M_b$. 

Now we work in $M_b$. Let $\{c_n\colon n\in\gw\}$ be an enumeration of all proper subsets of $b$ which are in $S$. Define $\bar{a}_0=\bar{a}$ and for $n>0$, $\bar{a}_{n+1}=\bar{a}_n^\frown c_n$. By going to a subsequece if necessary, we can assume that:
\begin{enumerate}
    \item For each $n\in\omega$, $c_{n}\setminus\bigcup\bar{a}_{n}\neq\emptyset$.
    \item For each $c\in S$ subset of $b$ such that $c\setminus\bigcup\bar{a}\neq\emptyset$, if $c$ does not appear in $\{c_n:n\in\omega\}$ then there is $k\in\omega$ such that $c\subseteq\bigcup_{l\leq k}c_l\cup\bigcup\bar{a}$.
\end{enumerate}
Now, for each $n\in\omega$ and $r\in \mathbf{Q}_{b,\bar{a}_n}$, let $\pi^n_0(r)=r\upharpoonright c_n$ and $\pi_1^n(r)=r\setminus\pi_0^n(r)$, so $$\pi_n:\mathbf{Q}_{b,\bar{a}_n}\to\mathbf{Q}_{c_n,\bar{a}_n}\times\mathbf{Q}_{b,\bar{a}_{n+1}}$$ given by $\pi_n(p)=(\pi_0^n(p),\pi_1^n(p))$ is an order isomorphism. Let $\{t_n\colon n\in\gw\}$ be an enumeration of $\omega^{<\omega}$.
Now we build sequences $\langle\dot{C}_n,D_n, p_n,q_n,\rho_n,s_n:n\in\omega\rangle$ as follows:

\begin{enumerate}
\item Let $q_0=q$, $\dot{C}_0=\dot{C}$, $D_0$ is the open dense $E(\dot{C})$ set above defined, and $(\rho_0,s_0)\in D_0$ is such that $(\rho_0,s_0)\leq(q_0,t_0)$. Now define $p_0=\emptyset$.

\item For $n>0$, assume everything is defined up to $n-1$. Let $q_n=\pi_1^{n-1}(\rho_{n-1})$. Define $\dot{C}_{n}$ to be the $\mathbf{Q}_{c_{n},\bar{a}_{n-1}}$-name,
\begin{equation*}
    \dot{C}_{n}=\{\langle(\pi_1^{n-1}(r),\check{s}),\pi_0^{n-1}(r)\rangle:(r,s)\in D_{n-1}\}
\end{equation*}
Note that $\dot{C}_{n}$ is a $\mathbf{Q}_{c_n,\bar{a}_{n-1}}$-name for an open dense subset of $\mathbf{Q}_{b,\bar{a}_{n}}\times\omega^{<\omega}$, and $\pi_0^{n-1}(\rho_{n-1})\in \mathbf{Q}_{c_n,\bar{a}_{n-1}}$, so by induction hypothesis we can find a condition $p_n\in \mathbf{P}_{c_{n-1},\bar{a}_{n-1}}$ and an open dense $D_{n}\subseteq \mathbf{Q}_{b,\bar{a}_{n}}\times\omega^{<\omega}$ such that $p_n\leq \pi_0^{n-1}(\rho_{n-1})$ and $p_n\Vdash D_{n}\subseteq\dot{C}_{n}$. Finally, let $(\rho_{n},s_{n})\in D_{n}$ be such that $(\rho_n,s_n)\leq(q_{n},t_n)$.

\end{enumerate}

After the construction is finished, define $p_\omega=\bigcup_{n\in\omega}p_n$. It follows from the construction that $p_\omega\in \mathbf{P}_{b,\bar{a}}$ and forces $D_\omega=\{s_n:n\in\omega\}$ to be a subset of $\dot{C}$. Clearly $D_\omega$ is a dense subset of $\omega^{<\omega}$.
\end{proof}

Now we finish the proof of Theorem \ref{thm_reduced_product}. For item (2) of the theorem, suppose $q\in \mathbf{P}$ is a condition and $\tau$ is a $\mathbf{P}$-name such that $q\Vdash\tau\subset\omega^{<\omega}$ is open dense. Use the correctness assumption on the true ground to find a countable elementary submodel $N$ of a large structure containing $\mathbf{P},q, \tau$ such that $N\cap \omega_n\in S$. Let $\sigma=\{\langle \check t,r\rangle\colon r\in \mathbf{P}\cap N, t\in\omega^{<\omega}, r\Vdash_{\mathbf{P}}\check t\in\tau\}$. Then $\sigma$ is a $\mathbf{P}\cap N$-name and $q\Vdash\sigma\subset\omega^{<\omega}$ is open dense by elementarity of the model $N$. Apply Claim~\ref{keyclaim} with $\bar a=\emptyset$ and $b=N\cap\omega_n$ to find a condition $p\leq q$ in $\mathbf{P}$ and an open dense subset $D\subset\omega^{<\omega}$ such that $p\Vdash_\mathbf{P}\check D\subset\sigma$. Then $p\Vdash_\mathbf{P}\check D\subseteq\tau$ as desired.

For item (1) of the theorem, suppose that $q\in \mathbf{P}$ is a condition and $\dot{\tau}$ is a $\mathbf{P}$-name such that $q\Vdash\dot{\tau}\colon\gw\to\check\gw_1$ is a surjective function. Since In $V$ there are $\omega_1$ open dense subsets of $\omega^{<\omega}$ whose complements generate the ideal $\mathsf{nwd}$, $q$ should force the existence of an open dense set which diagonalizes each dense set from $V$. Let $\dot{D}$ be a $\mathbf{P}$-name for such set. In particular, we have that for all open dense set $E\subseteq\omega^{<\omega}$ in $V$, $q\Vdash E\setminus \dot{D}$ is infinite. Let $N$ be a countable elementary submodel such that $\mathbf{P}, q, \dot{D}\in N$ and $b=N\cap \omega_n\in S$. Then, as before, we can find $p\leq q$ and  an open dense set $D_0$ in $V$ such that $p\Vdash D_0\subseteq \dot{D}$, which is a contradiction.

Item (3) of the theorem is independent of the previous work. Suppose that $A=\langle p_\ga\colon\ga\in\gw_2\rangle$ is a collection of conditions in $\mathbf{P}$; strengthening each of these conditions we may assume that each of them has domain in $S$. We must find two distinct and compatible elements of the set $A$. Let $\langle N_\gb\colon\gb\in\gw_1\rangle$ be an $\in$-tower
of countable elementary submodels of a large structure containing A such that
for each $\beta\in\omega_1$, $N_\beta\cap\omega_n\in S$ holds. Let $\alpha\in\omega_2\setminus\bigcup_{\beta\in\omega_1}N_\beta$ be an ordinal, and consider the condition $p_\alpha$. The countable set $c = dom(p_\alpha)\cap \bigcup_{\beta\in \omega_1}N_\beta$ is a subset
of one of the models on the tower; let $\beta\in\omega_1$ be such that $c\subseteq N_\beta$, and write
$d = N_\beta \cap\omega_n \in S$. Note that $p_\alpha\upharpoonright c \in \mathbf{P}_d$ holds. The poset $\mathbf{P}_d$ has cardinality $\omega_1$
by the all-important assumption that $M_d$ satisfies $\mathsf{CH}$; therefore, $\mathbf{P}_d\subseteq\bigcup_{\beta\in\omega_1}N_\beta$. It follows that there must be an ordinal $\gamma\in\omega_1$ such that $p_\alpha\upharpoonright c\in N_\gamma$. By
elementarity of the model $N_\gamma$ there must be an ordinal $\alpha'\in N_\gamma$ such that $p_{\alpha'}$
contains $p_\alpha\upharpoonright c$ as a subset. Then $p_\alpha$ and $p_{\alpha'}$ are compatible elements of the set
$A$ as required.
\\

Now we direct our attention to show how to use Theorem \ref{thm_reduced_product} to prove the main theorem. In what follows we extend the meaning of $\mathbb{P}_{\kappa}(\mathcal{U})$ to allow $\kappa$ to be an ordinal and not just a cardinal, that is, for $\gamma\in\mathsf{Ord}$ which is not a cardinal, we think of $\mathbb{P}_{\gamma}(\mathcal{U})$ as
\begin{equation*}
    \mathbb{P}_{\gamma}(\mathcal{U})=\{p\in\mathbb{P}_{\vert\gamma\vert^+}(\mathcal{U}): dom(p)\subseteq\gamma\}
\end{equation*}

Let $V$ be a model of $\mathsf{ZFC}+V=L$, and for each uncountable $\gamma\in\omega_2$, let $h_\gamma:\omega_1\to\gamma$ be a bijection such that $h_\gamma(0)=0$, letting $h_{\omega_1}$ to be the identity. In $V$, let $\mathbb{S}_{\omega_2}$ be the countable support iteration of the Sacks forcing and let $G_{\omega_2}$ be a $\mathbb{S}_{\omega_2}$-generic filter. Now, going to $V[G_{\omega_2}]$, following Definition \ref{Q_k_definition}, let $\{\dot{\mathcal{U}}_{\alpha_\beta}:\beta\in\omega_2\}$ be the family of suitable filters as there chosen. Now, in $V[G_{\omega_2}]$, define $h_\gamma^*:\mathbb{P}_{\omega_1}(\dot{\mathcal{U}}_{\alpha_\beta})^{V[G_{\alpha_\beta}]}\to\mathbb{P}_{\gamma}(\dot{\mathcal{U}}_{\alpha_\beta})^{V[G_{\alpha_\beta}]}$ as:
\begin{enumerate}
    \item $dom(h_\gamma^*(p))=\{h_{\gamma}(\delta):\delta\in dom(p)\}$.
    \item For each $\delta\in dom(h_\gamma^*(p))$, $h_\gamma^*(p)(\delta)=p(h_\gamma^{-1}(\delta))$
\end{enumerate}

Note that $h_{\gamma}^*$ is actually an order isomorphism. Now, in $V[G_{\omega_2}]$, for each $\beta<\omega_2$, let $\mathbb{P}_\beta=\mathbb{P}_{\omega_1}(\dot{\mathcal{U}}_{\alpha_\beta})^{V[G_{\alpha_\beta}]}$. Let $S\subseteq[\omega_2]^\omega$ be the stationary set from Lemma \ref{noetherian_stationary_set}. Now we apply Lemma \ref{models_system} to get a system $\langle(M_b,\prec_b):b\in S\rangle$. Note that for each $b\in S$, $\beta\in b$ and each ordinal $\delta\in[\omega_1,\omega_2)$, in $M_b$, the forcing $\mathbb{P}_{\delta}(\dot{\mathcal{U}}_{\alpha_\beta})^{V[G_{\alpha_\beta}]}$ can be recovered from $\mathbb{P}_{\omega_1}(\dot{\mathcal{U}}_{\alpha_\beta})^{V[G_{\alpha_\beta}]}$, as the former can be written as:
\begin{equation*}
    \mathbb{P}_{\delta}(\dot{\mathcal{U}}_{\alpha_\beta})^{V[G_{\alpha_\beta}]}=h_\delta^*[\mathbb{P}_{\omega_1}(\dot{\mathcal{U}}_{\alpha_\beta})^{V[G_{\alpha_\beta}]}]
\end{equation*}
Moreover, $\mathbb{P}_{\omega_2}(\dot{\mathcal{U}}_{\alpha_\beta})^{V[G_{\alpha_\beta}]}$ can be recovered in $M_b$, since we also have,
\begin{equation*}
    \mathbb{P}_{\omega_2}(\dot{\mathcal{U}}_{\alpha_\beta})^{V[G_{\alpha_\beta}]}=\bigcup_{\gamma\in[\omega_1,\omega_2)}h_\gamma^*[\mathbb{P}_{\omega_1}(\dot{\mathcal{U}}_{\alpha_\beta})^{V[G_{\alpha_\beta}]}]
\end{equation*}

Thus, $\mathbb{P}_{\omega_2}(\dot{\mathcal{U}}_{\alpha_\beta})^{V[G_{\alpha_\beta}]}\in M_b$. This implies that Definition \ref{reduced_product_definition} can be extended to the $\omega_2$ case, and more generally, for each uncountable $\delta\leq\omega_2$, we define the following:

\begin{dfn}\label{reduced_product_definition_omega_2_case}
Under the current assumptions, for each $\delta\in[\omega_1,\omega_2]$, define $\overline{\mathbf{P}}^{\delta}$ as follows:  for a condition $p\in\bigotimes_{\beta\in\omega_2}\mathbb{P}_\delta(\dot{\mathcal{U}}_{\alpha_\beta})^{V[G_{\alpha_\beta}]}$, declare $p\in \overline{\mathbf{P}}^{\delta}$ if and only if the following hold:
\begin{enumerate}
    \item $dom(p)\subseteq\omega_2$ is a countable set and a boolean combination of elements of $S$.
    \item For every $\beta\in dom(p)$, $p(\beta)\in \mathbb{P}_\delta(\dot{\mathcal{U}}_{\alpha_\beta})^{V[G_{\alpha_\beta}]}$.
    \item For every set $b\in S$, the function $p\restriction b$ belongs to $M_b$.
\end{enumerate}
The ordering is by coordinatewise strengthening.
\end{dfn}

An immediate consequence of Theorem \ref{thm_reduced_product} is that for any $\delta\in\omega_2$, $\overline{\mathbf{P}}^{\delta}$ is $\aleph_{2}$-c.c. On the other hand, we need to prove that $\overline{\mathbf{P}}^{\omega_2}$ is $\aleph_2$-c.c. 

\begin{claim}\quad
\begin{enumerate}
    \item For each $\delta\in [\omega_1,\omega_2)$, $\overline{\mathbf{P}}^{\delta}$ is isomorphic to $\overline{\mathbf{P}}^{\omega_1}$.
    \item $\overline{\mathbf{P}}^{\omega_2}=\bigcup_{\gamma\in[\omega_1,\omega_2)}\overline{\mathbf{P}}^{\gamma}$.  
\end{enumerate}
\end{claim}

\begin{proof}
(1) Recall that $h_\delta^*:\mathbb{P}_{\omega_1}(\dot{\mathcal{U}}_{\alpha_\beta})^{V[G_{\alpha_\beta}]}\to\mathbb{P}_{\delta}(\dot{\mathcal{U}}_{\alpha_\beta})^{V[G_{\alpha_\beta}]}$ is an order isomorphism for each $\delta\in[\omega_1,\omega_2)$. This isomorphism naturally lifts to an isomorphism between $\overline{\mathbf{P}}^{\omega_1}$ and $\overline{\mathbf{P}}^{\delta}$ as follows: for $p\in\overline{\mathbf{P}}^{\omega_1}$, define $\overline{h}_\delta(p)$ as:
\begin{enumerate}
    \item[i)] $dom(\overline{h}_\delta(p))=dom(p)$.
    \item[ii)] For $\gamma\in dom(\overline{h}_\delta(p))$, define $\overline{h}_\delta(p)(\gamma)=h_\delta^*(p(\gamma))\in\mathbb{P}_\delta(\dot{\mathcal{U}}_{\alpha_\gamma})^{V[G_{\alpha_\gamma}]}$.
\end{enumerate}

\noindent (2) Follows from the fact that $\mathbb{P}_{\omega_2}(\dot{\mathcal{U}}_{\alpha_\beta})^{V[G_{\alpha_\beta}]}=\bigcup_{\gamma\in[\omega_1,\omega_2)}h_\gamma^*[\mathbb{P}_{\omega_1}(\dot{\mathcal{U}}_{\alpha_\beta})^{V[G_{\alpha_\beta}]}]$, that each condition in $\overline{\mathbf{P}}^{\omega_2}$ has countable domain, and that $M_b$ satisfies $\mathsf{CH}$ for all $b\in S$.
\end{proof}

\begin{dfn}
For each set $b$ which is a boolean combination of elements of $S$, and $\delta\in[\omega_1,\omega_2]$, define $\overline{\mathbf{P}}_b^{\delta}$ as the family of all conditions $p\in\overline{\mathbf{P}}^\delta$ whose domain is a subset of $b$. The order on $\overline{\mathbf{P}}^{\delta}_b$ is inherited from the order on $\overline{\mathbf{P}}^\delta$.
\end{dfn}

First, we are aiming to prove  that for all $\delta\in\omega_2$, $\overline{\mathbf{P}}^{\delta}\lessdot\overline{\mathbf{P}}^{\omega_2}$, from which it will follow easily that $\overline{\mathbf{P}}^{\omega_2}$ preserves $\omega_1$, is $\aleph_2$-c.c. and Cohen preserving.

\begin{lemma}\label{regular_restriction}
Fix $\beta\in\omega_2$ and $\delta<\omega_2$. Working in $V[G_{\alpha_\beta}]$, the following property holds:
\begin{enumerate}
    \item [($\star$)] For any simple condition $q\in\mathbb{P}_{\omega_2}(\dot{\mathcal{U}}_{\alpha_\beta})^{V[G_{\alpha_\beta}]}$,  
    $q\upharpoonright\delta$ is a condition in $\mathbb{P}_\delta(\dot{\mathcal{U}}_{\alpha_\beta})^{V[G_{\alpha_\beta}]}$ and for any $p\leq q\upharpoonright\delta$, 
    $p\in\mathbb{P}_{\delta}(\dot{\mathcal{U}}_{\alpha_\beta})^{V[G_{\alpha_\beta}]}$, $p$ and $q$ are compatible (in $\mathbb{P}_{\omega_2}(\dot{\mathcal{U}}_{\alpha_\beta})^{V[G_{\alpha_\beta}]}$).
\end{enumerate}
\end{lemma}

\begin{proof}
That $q\upharpoonright\delta$ is a condition in $\mathbb{P}_\delta(\dot{\mathcal{U}}_{\alpha_\beta})^{V[G_{\alpha_\beta}]}$ is immediate from the definition of the forcing. Let $p\in\mathbb{P}_\delta(\dot{\mathcal{U}}_{\alpha_\beta})^{V[G_{\alpha_\beta}]}$ be a condition extending $q\upharpoonright\delta$. Let $\langle\gamma_n:n\in\omega\rangle$ be the canonical enumeration of $dom(q)$ and define $B=\{n\in\omega:\gamma_n\notin\delta\}$. For each $n\in B$, let $k_n\in\omega$ be such that
\begin{equation*}
\bigcup_{j\leq n}E^{q(0)}_j\subseteq dom(h_{p(0)})\cup\bigcup_{l<k_n} E^{p(0)}_l    
\end{equation*}
note that such $k_n$ exists because $q$ is a simple condition. Note that for each $n\in\omega$, $\vec{E}_{p(0)}*k_n\sqsubseteq \vec{E}_{q(0)}*(n+1)=\vec{E}_{q(\gamma_n)}$. Now define a condition $r\in\mathbb{P}_{\omega_2}(\dot{\mathcal{U}}_{\alpha_\beta})^{V[G_{\alpha_\beta}]}$ as follows:
\begin{enumerate}
    \item For $\gamma\in dom(p)$, let $r(\gamma)=p(\gamma)$.
    \item For $\gamma\in dom(q)\setminus \delta$, let $n\in B$ be such that $\gamma=\gamma_n$, and define $r(\gamma)=(h(q(\gamma_n)))^{(*,\vec{E}_{p(0)}*k_n)}$.
\end{enumerate}
It is clear from the definition that $r$ is a condition, and extends both $p$ and $q$.
\end{proof}

\begin{crl}\label{crl_regular_restriction_transitive_models}
For any $\beta\in\omega_2$, $\delta\in[\omega_1,\omega_2)$ and $b\in S$, the following holds:
\begin{enumerate}
    \item The property $(\star)$ in Lemma \ref{regular_restriction} holds in any transitive model of $\mathsf{ZFC}$ which contains $\mathbb{P}_{\omega_2}(\dot{\mathcal{U}}_{\alpha_\beta})^{V[G_{\alpha_\beta}]}$, its order relation and $\delta$.
    \item $V[G_{\omega_2}]\vDash\mathbb{P}_{\delta}(\dot{\mathcal{U}}_{\alpha_\beta})^{V[G_{\alpha_\beta}]}\lessdot\mathbb{P}_{\omega_2}(\dot{\mathcal{U}}_{\alpha_\beta})^{V[G_{\alpha_\beta}]}$.
    \item For any $\beta\in b$, $M_b\vDash\mathbb{P}_{\delta}(\dot{\mathcal{U}}_{\alpha_\beta})^{V[G_{\alpha_\beta}]}\lessdot\mathbb{P}_{\omega_2}(\dot{\mathcal{U}}_{\alpha_\beta})^{V[G_{\alpha_\beta}]}$.
\end{enumerate}
\end{crl}

\begin{proof}
(1) Follows from the definition of the order on $\mathbb{P}_{\omega_2}(\dot{\mathcal{U}}_{\alpha_\beta})^{V[G_{\alpha_\beta}]}$, the transitivity of the model and the fact that $\mathbb{P}_{\delta}(\dot{\mathcal{U}}_{\alpha_\beta})^{V[G_{\alpha_\beta}]}$ is definible from $\delta$ and $\mathbb{P}_{\omega_2}(\dot{\mathcal{U}}_{\alpha_\beta})^{V[G_{\alpha_\beta}]}$.

(2) Let $\mathcal{A}\subseteq\mathbb{P}_\delta(\dot{\mathcal{U}}_{\alpha_\beta})^{V[G_{\alpha_\beta}]}$ be a maximal antichain and $q\in\mathbb{P}_{\omega_2}(\dot{\mathcal{U}}_{\alpha_\beta})^{V[G_{\alpha_\beta}]}$ any condition. Extend $q$ to a simple condition $q_0$. By (1), for any condition $p\in \mathbb{P}_{\delta}(\dot{\mathcal{U}}_{\alpha_\beta})^{V[G_{\alpha_\beta}]}$ such that $p\leq q_0\upharpoonright\delta$, $p$ and $q_0$ are compatible. Since $\mathcal{A}$ is a maximal antichain, there is $p_0\in\mathcal{A}$ compatible with $q_0\upharpoonright\delta$. Let $p_1\in\mathbb{P}_\delta(\dot{\mathcal{U}}_{\alpha_\beta})^{V[G_{\alpha_\beta}]}$ be a common extension of $p_0$ and $q_0\upharpoonright\beta$. It follows that $p_0$ and $q$ are compatible. Thus $\mathcal{A}$ is a maximal antichain in $\mathbb{P}_{\omega_2}(\dot{\mathcal{U}}_{\alpha_\beta})^{V[G_{\alpha_\beta}]}$.

(3) This is identical to (2).
\end{proof}

Following former notation, $\overline{\mathbf{P}}_{b,\overline{a}}^\delta$ denotes the poset $\overline{\mathbf{P}}_{b\setminus\bigcup\overline{a}}^\delta$.

\begin{lemma}\label{factor_P}
Let $b\in S$ and $\overline{a}\subseteq S$ be such that $\overline{a}$ is finite and $b\setminus \bigcup\overline{a}\neq\emptyset$. Let $c\in S$ be a proper subset of $b$. Then, for each $\delta\in [\omega_1,\omega_2]$, $\overline{\mathbf{P}}^\delta_{b,\overline{a}}$ is isomorphic to $\overline{\mathbf{P}}_{c,\overline{a}}^\delta\times\overline{\mathbf{P}}_{b,\overline{a}^\frown c}^\delta$.
\end{lemma}

\begin{proof}
This follows similar lines to the proof of Claim \ref{claim1}.    
\end{proof}

\begin{lemma}\label{regular_restriction_reduced_product}
Fix $\delta<\omega_2$, $b\in S$ and $\overline{a}\subseteq S$ such that $b\setminus \bigcup\overline{a}\neq\emptyset$. Then, for any condition $q\in\overline{\mathbf{P}}_{b,\overline{a}}^{\omega_2}$, the following holds:
\begin{enumerate}
    \item $q\upharpoonright\upharpoonright\delta=\langle q(\gamma)\upharpoonright\delta:\gamma\in dom(q)\rangle$ is a condition in $\overline{\mathbf{P}}^{\delta}_{b,\overline{a}}$.
    \item If $q(\beta)$ is a simple condition for all $\beta\in dom(q)$, then, any condition $p\in\overline{\mathbf{P}}^{\delta}_{b,\overline{a}}$ extending $q\upharpoonright\upharpoonright\delta$ is compatible with $q$ (in $\overline{\mathbf{P}}^{\omega_2}_{b,\overline{a}}$).
\end{enumerate}
 
\end{lemma}

\begin{proof}
(1) Let $\delta$, $b$ and $\overline{a}$ as in the hypothesis. Pick a condition $q\in\mathbf{P}^{\omega_2}_{b,\overline{a}}$. 
By definition of $\overline{\mathbf{P}}^{\omega_2}$, for each $c\in S$, $q\upharpoonright c\in M_c$, from which it follows that $\langle q(\beta)\upharpoonright\delta:\beta\in c\cap dom(q)\rangle\in M_c$. Thus, $q\upharpoonright\upharpoonright\delta=\langle q(\beta)\upharpoonright\delta:\beta\in dom(q)\rangle$ is a condition in $\overline{\mathbf{P}}^{\delta}_{b,\overline{a}}$. 

(2) The proof is by induction on the rank of $b\in S$. If $b=\{\beta\}$, this is Corollary \ref{crl_regular_restriction_transitive_models}. Fix $b\in S$ and assume the assertion holds for all $d\in S$ having smaller rank than the rank of $b$. We work in $M_b$. Let $\{c_n:n\in\omega\}$ be an enumeration of $\{c\in S:c\subseteq b\land c\neq b\}$.  Define $\overline{a}_0=\{d\cap b:d\in \overline{a}\}$, and for $n>0$, $\overline{a}_n={\overline{a}_{n-1}}^\frown c_{n-1}$. By going to a subsequece if necessary, we can assume that:
\begin{enumerate}
    \item For each $n\in\omega$, $c_{n}\setminus\bigcup\bar{a}_{n}\neq\emptyset$.
    \item For each $c\in S$ subset of $b$ such that $c\setminus\bigcup\bar{a}\neq\emptyset$, if $c$ does not appear in $\{c_n:n\in\omega\}$ then there is $k\in\omega$ such that $c\subseteq\bigcup_{l\leq k}c_l\cup\bigcup\bar{a}$.
\end{enumerate}

Let $p\in\overline{\mathbf{P}}_{b,\overline{a}}^{\delta}$ be a condition stronger than $q\upharpoonright\upharpoonright\delta$. Now construct sequences $\langle p_n,q_n,r_n:n\in\omega\rangle$ as follows:
\begin{enumerate}
    \item First define $p_{-1}=p$, $q_{-1}=q$.
    \item For all $n\in\omega\cup\{-1\}$, $p_n\in\overline{\mathbf{P}}^\delta_{b,\overline{a}_{n+1}}$, $q_n\in\overline{\mathbf{P}}^{\omega_2}_{b,\overline{a}_{n+1}}$ and $p_n\leq q_n\upharpoonright\upharpoonright\delta$.
    \item By Lemma \ref{factor_P}, we have $q_{n-1}\upharpoonright c_n\in\overline{\mathbf{P}}^{\omega_2}_{c_n,\overline{a}_n}$, $p_{n-1}\upharpoonright c_n\in \overline{\mathbf{P}}^{\delta}_{c_n,\overline{a}_n}$, and $p_{n-1}\upharpoonright c_n\leq (q_{n-1}\upharpoonright c_n)\upharpoonright\upharpoonright \delta$. By induction hypothesis we can find a condition $r\in\overline{\mathbf{P}}_{c_n,\overline{a}_n}^{\omega_2}$ extending both $p_{n-1}\upharpoonright c_n$ and $q_{n-1}\upharpoonright c_n$; let $r_n$ be one of such conditions. Now define $p_{n}=p_{n-1}\upharpoonright (\omega_2\setminus c_n)$ and $q_{n}=q_{n-1}\upharpoonright(\omega_2\setminus c_n)$. Note that $p_{n}\in\overline{\mathbf{P}}_{b,\overline{a}_n^\frown c_n}^\delta$, $q_{n}\in\overline{\mathbf{P}}^{\omega_2}_{b,\overline{a}_n^\frown c_n}$ and $p_n\leq q_n\upharpoonright\upharpoonright\delta$.
\end{enumerate}

Now define $p_\omega=\bigcup_{n\in\omega}r_n$. Let us see that $p_\omega$ is a condition in $\overline{\mathbf{P}}^{\omega_2}_{b,\overline{a}}$. Pick any $d\in S$. If $d\cap(b\setminus\bigcup\overline{a})=\emptyset$ then $p_\omega\upharpoonright d=\emptyset\in M_d$. Assume otherwise and let $c=d\cap b$. By definition of the enumeration $\{c_n:n\in\omega\}$, there is $n\in\omega$ such that $c\subseteq\bigcup_{l\leq n}c_l\cup\bigcup\overline{a}$, so we have that $p_\omega\upharpoonright d=p_\omega\upharpoonright c=\bigcup_{j\leq n}r_j\upharpoonright c$. By induction hypothesis, each $r_j$ is a condition in $\overline{\mathbf{P}}_{c_j,\overline{a}_j}^{\omega_2}\subseteq\overline{\mathbf{P}}^{\omega_2}$, which implies that $r_j\upharpoonright c\in M_{c}$; thus, $p_\omega\upharpoonright c=\bigcup_{j\leq n}r_j\upharpoonright c\in M_{c}$. Therefore, $p_\omega$ is a condition in $\overline{\mathbf{P}}^{\omega_2}$. Since $dom(p_\omega)\subseteq b\setminus\bigcup\overline{a}$, we have that $p_\omega$ is a condition in $\overline{\mathbf{P}}^{\omega_2}_{b,\overline{a}}$. By construction we have that $p_\omega$ extends both $p$ and $q$.

\end{proof}

\begin{crl}
For any $\delta\in[\omega_1,\omega_2)$ and $b\in S$, $M_b\vDash\overline{\mathbf{P}}^{\delta}_b\lessdot\overline{\mathbf{P}}^{\omega_2}_b$ and $V[G_{\omega_2}]\vDash\overline{\mathbf{P}}^{\delta}_b\lessdot\overline{\mathbf{P}}^{\omega_2}_b$.
\end{crl}

\begin{proof}
Let $\mathcal{A}\subseteq\overline{\mathbf{P}}^\delta_b$ be a maximal antichain and $q\in\overline{\mathbf{P}}^{\omega_2}_b$ a condition. We can assume that for each $\beta\in dom(q)$, $q(\beta)$ is a simple condition. By Lemma \ref{regular_restriction_reduced_product}(1) (applied with $\overline{a}=\emptyset$), we have $q\upharpoonright\upharpoonright\delta\in\overline{\mathbf{P}}^\delta_b$ so we can find a condition $p\in\mathcal{A}$ compatible with $q\upharpoonright\upharpoonright \delta$. Let $r\in\overline{\mathbf{P}}^\delta_b$ be a common extension of $q\upharpoonright\upharpoonright\delta$ and $p$. By an application of Lemma \ref{regular_restriction_reduced_product}(2), $r$ and $q$ are compatible.
\end{proof}

\begin{lemma}\label{reduced_prod_delta_reg_sub_red_prod_total}
For any $\delta\in [\omega_1,\omega_2)$, $\overline{\mathbf{P}}^{\delta}\lessdot\overline{\mathbf{P}}^{\omega_2}$.
\end{lemma}

\begin{proof}
Let $\mathcal{A}\subseteq\overline{\mathbf{P}}^{\delta}$ be a maximal antichain. Pick an arbitrary condition $q_0\in\overline{\mathbf{P}}^{\omega_2}$. Extend $q_0$ to a condition $q_1$ such that $dom(q_1)\in S$ and let $b=dom(q_1)$. We can assume that for each $\beta\in b$, $q_1(\beta)$ is a simple condition in $\mathbb{P}_{\omega_2}(\dot{\mathcal{U}}_{\alpha_\beta})^{V[G_{\alpha_\beta}]}$.

By Lemma \ref{regular_restriction_reduced_product}(1), $q_1\upharpoonright\upharpoonright\delta$ is a condition in $\overline{\mathbf{P}}^{\delta}_b$, which implies that $q_1\upharpoonright\upharpoonright\delta$ is a condition in $\overline{\mathbf{P}}^{\delta}$. Since $\mathcal{A}$ is a maximal antichain, there is $p'\in\mathcal{A}$ compatible with $q_1\upharpoonright\upharpoonright\delta$, so we find a condition $p^*\in\overline{\mathbf{P}}^\delta$ stronger than both $q_1\upharpoonright\upharpoonright\delta$ and $p'$. We can assume that $dom(p^*)\in S$.

Note that $p^*\upharpoonright b\in\overline{\mathbf{P}}^\delta_b$ and $p^*\upharpoonright b\leq q_1\upharpoonright\upharpoonright \delta$, so $p^*\upharpoonright b$ and $q_1$ are compatible in $\overline{\mathbf{P}}^{\omega_2}_b$. Let $p^{**}\in\overline{\mathbf{P}}^{\omega_2}_b$ be a common extension of $p^*\upharpoonright b$ and $q_1$. It is not difficult to see that $p^{**}\cup p^*\upharpoonright(\omega_2\setminus b)$ is a condition in $\overline{\mathbf{P}}^{\omega_2}$ extending $q_1$ and $p^*$, which implies that $q_0$ and $p'\in\mathcal{A}$ are compatible.
\end{proof}

\begin{lemma}
$\overline{\mathbf{P}}^{\omega_2}$ is $\aleph_2$-c.c.
\end{lemma}
\begin{proof}
We proceed as in the $\omega_1$-case. Fix $\{p_\alpha:\alpha\in\omega_2\}\subseteq\overline{\mathbf{P}}^{\omega_2}$. We can assume that for each $\alpha$, $dom(p_\alpha)\in S$. Let now $\langle N_\alpha:\alpha\in\omega_1\rangle$ be an $\in$-increasing sequence of countable elementary submodels such that $S,\omega_1,\omega_2,\overline{\mathbf{P}}^{\omega_2},\{p_\alpha:\alpha\in\omega_2\}\in N_0$, and for each $\alpha\in\omega_1$, $N_\alpha\cap\omega_2\in S$. Define $N_{\omega_1}=\bigcup_{\beta\in\omega_1}N_\beta$. First note that $\overline{\mathbf{P}}^{\omega_2}\cap N_{\omega_1}=\bigcup_{\gamma\in[\omega_1,\omega_2)\cap N_{\omega_1}}\overline{\mathbf{P}}^{\gamma}\cap N_{\omega_1}$.

Pick $\alpha_0\in\omega_2\setminus N_{\omega_1}$, let $A=dom(p_{\alpha_0})\cap N_{\omega_1}$ and define $q=p_{\alpha_0}\upharpoonright A$. Since $A$ is countable, there is $\gamma\in\omega_1$ such that $A\subseteq N_\gamma$. Since $N_\gamma\cap\omega_2\in S$, we have that $A= dom(q_{\alpha_0})\cap B$, where $B=\omega_2\cap N_\gamma$. Moreover, since $S,N_\gamma\in N_{\gamma+1}$, we have $B\in N_{\gamma+1}$. Note that $q\in \overline{\mathbf{P}}^{\omega_2}$ and $q\subseteq N_{\omega_1}$ both hold true. Also, we can find $\gamma_0\in[\omega_1,\omega_2)\cap N_{\omega_1}$ such that for each $\delta\in A$, $q(\delta)\in \mathbb{P}_{\gamma_0}(\dot{\mathcal{U}}_{\alpha_\delta})^{V[G_{\alpha_\delta}]}$. For such $\gamma_0$, $\overline{\mathbf{P}}^{\gamma_0}\in N_{\omega_1}$ holds true, and we also have that $q\in\overline{\mathbf{P}}^{\gamma_0}_B\in N_{\omega_1}$. However, $\overline{\mathbf{P}}^{\gamma_0}_B$ has cardinality $\omega_1$ and $N_{\omega_1}$ knows this, which implies $\overline{\mathbf{P}}^{\gamma_0}_B\subseteq N_{\omega_1}$, from which it follows that $q\in N_{\omega_1}$. By elementarity, we can find $\alpha_1\in\omega_2$ such that $p_{\alpha_1}\in N_{\omega_1}$ and  $q\subseteq p_{\alpha_1}$. Therefore, $p_{\alpha_0}$ and $p_{\alpha_1}$ are compatible.

\end{proof}

\begin{lemma}
$\overline{\mathbf{P}}^{\omega_2}$ is Cohen preserving.
\end{lemma}
\begin{proof}
Let $\dot{D}$ be a $\overline{\mathbf{P}}^{\omega_2}$-name for an open dense subset of $\omega^{<\omega}$ and $p\in\overline{\mathbf{P}}^{\omega_2}$ an arbitrary condition. For each $s\in \omega^{<\omega}$, let $A_s$ be a maximal antichain in $\overline{\mathbf{P}}^{\omega_2}$ of conditions deciding the formula $``s\in\dot{D}"$. Each one of these antichains has cardinality $\omega_1$ so we can find $\delta\in\omega_2$ such that $\{p\}\cup\bigcup_{s\in\omega^{<\omega}}A_s\subseteq\overline{\mathbf{P}}^{\delta}$. Thus, we can think of $\dot{D}$ as a $\overline{\mathbf{P}}^{\delta}$-name. $\overline{\mathbf{P}}^\delta$ is isomorphic to $\overline{\mathbf{P}}^{\omega_1}$, so, by an application of Theorem \ref{thm_reduced_product}, we can find a condition $q\leq p$, $q\in\overline{\mathbf{P}}^\delta$, and an open dense set $E\subseteq\omega^{<\omega}$ such that $q\Vdash_{\overline{\mathbf{P}}^\delta} E\subseteq\dot{D}$. Finally, from Lemma \ref{reduced_prod_delta_reg_sub_red_prod_total}, we have that $\overline{\mathbf{P}}^{\delta}$ is a regular suborder of $\overline{\mathbf{P}}^{\omega_2}$, from which it follows that $q\Vdash_{\overline{\mathbf{P}}^{\omega_2}} E\subseteq\dot{D}$
\end{proof}

\begin{lemma}
$\overline{\mathbf{P}}^{\omega_2}$ preserves $\omega_1$.    
\end{lemma}

\begin{proof}
Let $\dot{f}$ be a $\overline{\mathbf{P}}^{\omega_2}$-name and $p\in\overline{\mathbf{P}}^{\omega_2}$ a condition such that $p\Vdash\dot{f}:\omega\to\omega_1$. For each $n\in\omega$, let $A_n\subseteq\overline{\mathbf{P}}^{\omega_2}$ be a maximal antichain of conditions deciding the value of $\dot{f}(n)$. Each $A_n$ has cardinality at most $\omega_1$, so we can find $\delta\in\omega_2$ such that $\{p\}\cup\bigcup_{n\in\omega}A_n\subseteq\overline{\mathbf{P}}^{\delta}$. Again, we can see $\dot{f}$ as a $\overline{\mathbf{P}}_{\delta}$-name, and we can find $q\leq p$, $q\in\overline{\mathbf{P}}^{\delta}$ and $\gamma\in\omega_1$ such that $q\Vdash_{\overline{\mathbf{P}}^{\delta}} rng(\dot{f})\subseteq\gamma$. As in the previous lemma, this implies that $q\Vdash_{\overline{\mathbf{P}}^{\omega_2}} rng(\dot{f})\subseteq\gamma$.

\end{proof}

\section{Remarks and open questions.}

Let us recall that the Katowice problem asks whether the boolean algebras $\mathcal{P}(\omega)/fin$ and $\mathcal{P}(\omega_1)/fin$ can be isomorphic. Assuming the existence of such isomorphism, there are many implications which are known to be consistent with $\mathsf{ZFC}$; here we present one more. In \cite{hrusak_combinatics}, M. Hru\v{s}\'ak pointed out that the existence of such an isomorphism implies that the class of Tukey types of ultrafilters in $\omega$ is the same as the class of Tukey types of ultrafilters on $\omega_1$. We remark this is consistent.

\begin{prp}
Assume all ultrafilters on $\omega$ are Tukey top and $2^\omega=2^{\omega_1}$. Then all  ultrafilters on $\omega_1$ are Tukey top. Moreover, for any cardinal $\kappa$ such that $2^{\kappa}=2^\omega$, all ultrafilters on $\kappa$ have maximal cofinal type.\footnote{The consistency of all uniform ultrafilters on $\omega_1$ being Tukey top was obtained before the present version of the paper by J. Cruz-Chapital (see \cite{chapital_barely}), and also, independently, by J. T. Benhamou, T. Moore and  L. Serafin. (to appear, \cite{benhamou_justin_serafin}).}
\end{prp}

\begin{proof}
Let $\mathcal{U}$ be an ultrafilter on $\omega_1$. We can assume that $\mathcal{U}$ is an uniform ultrafilter. Since $\omega_1$ is not measurable, there is a partition $\{P_n:n\in\omega\}$ of $\omega_1$ such that $P_n\notin\mathcal{U}$ for each $n\in\omega$. Define $f:\omega_1\to\omega$ as $f(\alpha)=n$ if and only if $\alpha\in P_n$, and let $\mathcal{F}=\{A\subseteq\omega:f^{-1}[A]\in\mathcal{U}\}$. Then $\mathcal{F}$ is an ultrafilter on $\omega$. Let $\{A_\alpha:\alpha\in 2^\omega\}\subseteq\mathcal{F}$ be a family witnessing $\mathcal{F}$ is Tukey top. Then $\{f^{-1}[A_\alpha]:\alpha\in 2^\omega\}$ witnesses $\mathcal{U}$ is Tukey top. The case for arbitrary $\kappa$ such that $2^\kappa=2^\omega$ is similar.
\end{proof}

This proposition and Theorem \ref{main_theorem} clearly imply the following:

\begin{crl}
It is consistent that $\omega$ and $\kappa$ have the same Tukey types of ultrafilters for any $\kappa$ such that $2^\omega=2^\kappa$.     
\end{crl}

However, it is not clear that in any of the models here presented there is an isomorphism between the aforementioned boolean algebras, and there is no reason to believe there is, since the models are not constructed for such thing to happen.

We recall the two questions of section \ref{basic_remarks}:

\begin{question}
Is it consistent that there is no $\mathsf{nwd}$-ultrafilter but there is an ultrafilter not Tukey above $[\omega_1]^{<\omega}$?
\end{question}

\begin{question}
Is there a non-Tukey top ultrafilter in Shelah's model for no $\mathsf{nwd}$-ultrafilters? A basically generated ultrafilter?
\end{question}

The characterization of $[\kappa]^{<\omega}\leq_T\mathcal{U}$ given by \ref{critical_ideal} opens the possibility of characterizing the generic existence of ultrafilters non-Tukey above $[\kappa]^{<\omega}$, so our first question is:

\begin{question}
How can the generic existence of ultrafilters non-Tukey above $[\kappa]^{<\omega}$ be characterized? 
\end{question}

Let us recall that J. Brendle and J. Fla\v{s}kov\'a have investigated the generic existence of $\mathscr{I}$-ultrafilters where $\mathscr{I}$ is an ideal on $\omega$, and introduced the cardinal invariant $\mathfrak{ge}(\mathscr{I})$ for a given ideal $\mathscr{I}$. They proved that generic existence of $\mathscr{I}$-ultrafilters is equivalent to the equality $\mathfrak{ge}(\mathscr{I})=\mathfrak{c}$. Under the current context, ultrafilters which are not Tukey above $[\kappa]^{<\omega}$ are characterized as exactly the same class as $\tau_\kappa$-ultrafilters, so the natural question is whether there is a similar characteri\-zation of generic existence in terms of some relation between appropriated cardinal invariants.

\begin{question}
Is there a nice characterization of generic existence of Tukey top ultrafilters?
\end{question}

Regarding Question 11.4, it is possible to prove that in the side by side Silver model, Tukey top ultrafilters exist generically, and actually, for each filter $\mathcal{F}$ of character less than $2^\omega$, it is possible to add $\omega_1$-many sets to $\mathcal{F}$ to make sure that any ultrafilter extending $\mathcal{F}$ is Tukey top.

\begin{question}
Is any basically generated ultrafilter Sacks indestructible?
\end{question}

\begin{question}
Is there a non-Tukey top ultrafilter in the model from Theorem \ref{thm1}?
\end{question}

\begin{question}
Is it consistent that all ultrafilters are Tukey Top and there is a $\mathsf{nwd}$-ultrafilter?
\end{question}

It may be possible that the following questions are still far away to be reached just from the current consistency tools and machinery. Despite this, they are still interesting. Let us define the \emph{high spectrum} of the Tukey types of ultrafilters on $\omega$ as the set of cardinals $\kappa\ge\omega_1$ for which there is an ultrafilter such that $\mathcal{U}\equiv[\kappa]^{<\omega}$,
\begin{equation*}
    \mathsf{spec}^+(\leq_T,\omega)=\{\kappa:(\exists \mathcal{U}\in\beta\omega^*)(\mathcal{U}\equiv[\kappa]^{<\omega})\}
\end{equation*}

\begin{question}
How complex can the high spectrum be?
\end{question}

\begin{question}
Is the conjunction of the next assertions consistent?
\begin{enumerate}
    \item There are non-Tukey top ultrafilters.
    \item For each ultrafilter $\mathcal{U}$ there is $\kappa\leq \mathfrak{c}$ such that $\mathcal{U}\equiv[\kappa]^{<\omega}$?
\end{enumerate}
\end{question}

Let us say that the high spectrum is convex if all infinite cardinals not bigger that $2^\omega$ belong to $\mathsf{spec}^+(\leq_T,\omega)$. 

\begin{question}
Is it consistent that the high spectrum is not convex?
\end{question}

\begin{question}
What is the high spectrum in the side-by-side Sacks model?
\end{question}

\begin{question}
Is it consistent that the cardinality of $\mathsf{spec}^+(\leq_T,\omega)$ can be any positive natural number?    
\end{question}

\begin{question}
Is it consistent that the cardinality of $\mathsf{spec}^+(\leq_T,\omega)$ can be $\aleph_0$?    
\end{question}

More generally, 

\begin{question}
What are the possible cardinalities of $\mathsf{spec}^+(\leq_T,\omega)$?
\end{question}

Assume $\mathsf{MA}+2^\omega>\omega_1$, and let $\kappa\in[\omega_1,2^\omega)$ be a cardinal. Define $\tau(\kappa,\omega)$ as:
\begin{equation*}
\tau(\kappa,\omega)=\{\mathcal{U}\in\beta\omega^*:[\kappa]^{<\omega}\leq_T\land [\kappa^+]^{<\omega}\nleq_T\mathcal{U}\}
\end{equation*}

and give $\tau(\kappa,\omega)$ the Tukey ordering,
\begin{question}
How complex is the structure of $(\tau(\kappa,\omega),\leq_T)$?
\end{question}

\begin{question}
How different is $\tau(\kappa,\omega)$ from $\tau(\lambda,\omega)$ when $\kappa\neq\lambda$?
\end{question}

\section{Acknowledgments.} 
The first author was supported by the ``Programme to support prospective human resources – post Ph.D. candidates" of the Czech Academy of Sciences, project L100192251, and by the Czech Academy of Sciences CAS (RVO 67985840), and also supported by
 Universidad Nacional Autónoma de México Postdoctoral
 Program (POSDOC). The research for this paper started when the first author was a postdoc at Institute of Mathematics of Czech Academy of Sciences, and concluded after he became a postdoc in the Universidad Aut\'onoma de M\'exico (UNAM), under the posdoctoral fellowship program ``Elisa Acu\~na". He is grateful to both institutions. The second author was supported by NSF grant DMS 2348371.

\bibliographystyle{alpha}
\bibliography{references}

\end{document}